\documentclass[final,1p,times]{elsarticle}

\usepackage{amsmath, amsfonts, amssymb, amsthm}
\usepackage{appendix}
\usepackage{graphicx}
\usepackage{subfigure}

\journal{Applied and Computational Harmonic Analysis}

\newcommand{\A}{\mathbb{A}}

\newcommand{\C}{\mathcal{C}}
\newcommand{\D}{\mathbb{D}}
\newcommand{\Dt}{\mathcal{D}^{(t)}}
\newcommand{\G}{\mathcal{G}}
\newcommand{\HH}{\mathcal{H}}
\newcommand{\I}{\mathcal{I}}
\newcommand{\K}{\mathbb{K}}
\newcommand{\N}{\mathbb{N}}

\newcommand{\R}{\mathbb{R}}
\newcommand{\T}{\mathbb{T}}
\newcommand{\X}{\mathcal{X}}
\newcommand{\x}{\mathrm{\bf x}}

\theoremstyle{plain}
\newtheorem{assumption}{Assumption}
\newtheorem{theorem}{Theorem}[section]

\newtheorem{lemma}[theorem]{Lemma}

\theoremstyle{definition}
\newtheorem{remark}[theorem]{Remark}

\begin{document}

\begin{frontmatter}

\title{Diffusion maps for changing data\tnoteref{t1}}
\tnotetext[t1]{To appear in {\it Applied and Computational Harmonic
    Analysis}. arXiv:1209.0245.}

\author{Ronald R. Coifman}
\ead{coifman@math.yale.edu}

\author{Matthew J. Hirn\corref{cormatt}}
\ead{matthew.hirn@yale.edu}
\ead[url]{www.math.yale.edu/$\sim$mh644}

\cortext[cormatt]{Corresponding author}

\address{
Yale University \\
Department of Mathematics \\
P.O. Box 208283 \\
New Haven, Connecticut 06520-8283 \\
USA
}

\begin{abstract}
Graph Laplacians and related nonlinear mappings into low dimensional
spaces have been shown to be powerful tools for organizing high
dimensional data. Here we consider a data set $X$ in which the graph
associated with it changes depending on some set of parameters. We analyze this type of
data in terms of the diffusion distance and the corresponding
diffusion map. As the data changes over the parameter space, the low
dimensional embedding changes as well. We give a way to go between
these embeddings, and furthermore, map them all into a common
space, allowing one to track the evolution of $X$ in its intrinsic
geometry. A global diffusion distance is also defined, which gives a
measure of the global behavior of the data over the parameter
space. Approximation theorems in terms of randomly sampled data are
presented, as are potential applications. 
\end{abstract}


\end{frontmatter}

\section{Introduction} \label{sec: introduction}

In this paper we consider a changing graph depending on certain
parameters, such as time, over a fixed set
of data points. Given a set of parameters of interest, our goal is to
organize the data in such a way that we can perform meaningful
comparisons between data points derived from different parameters. In
some scenarios, a direct comparison may be possible; on the other
hand, the methods we develop are more general and can handle
situations in which the changes to the data prevent direct comparisons
across the parameter space. For example, one may consider situations
in which the mechanism or sensor measuring the data changes,
perhaps changing the observed dimension of the data. In order to make
meaningful comparisons between different realizations of the data, we
look for invariants in the data as it changes. We model the data
set as a normalized, weighted graph, and measure the similarity
between two points based on how the local subgraph around each point
changes over the parameter space. The framework we
develop will allow for the comparison of any two points derived from
any two parameters within the graph, thus allowing one to organize not
only along the data points but the parameter space as well.

An example of this type of data comes from hyperspectral image analysis. A
hyperspectral image is in fact a set of images of the same scene that
are taken at different wavelengths. Put together, these images form a
data cube in which the length and width of the cube correspond to spatial
dimensions, and the height of the cube corresponds to the different
wavelengths. Thus each pixel is in fact a vector corresponding to the
spectral signature of the materials contained in that pixel. Consider
the situation in which we are given two hyperspectral images of the
same scene, and we wish to highlight the anomalous (e.g., man made)
changes between the two. Assume though, that for each data set, different
cameras were used which measured different wavelengths, perhaps also
at different times of day under different weather conditions. In such
a scenario a direct comparison of the spectral signatures between
different days becomes much more difficult. Current
work in the field often times goes under the heading change detection,
as the goal is to often find small changes in a large scene; see
\cite{eismann:hsiChangeDetection2008} for more details.

Other possible areas for applications come from the modeling of social
networks as graphs. The relationships between people change over time
and determining how groups of people interact and evolve is a new and
interesting problem that has usefulness in marketing and other
areas. Financial markets are yet another area that lends itself to
analysis conducted over time, as are certain evolutionary biological
questions and even medical problems in which patient tests are updated
over the course of their lives.

The tools developed in this paper are inspired by high dimensional data analysis, in
which one assumes that the data has a hidden, low dimensional
structure (for example, the data lies on a low dimensional
manifold). The goal is to construct a mapping that parameterizes this
low dimensional structure, revealing the intrinsic geometry of the
data. We are interested in high dimensional data the evolves over some
set of paramaters, for example time. We are particularly interested in the
case in which one does not have a given metric by which to compare the
data across time, but can only compare data points from the same time
instance. The hyperspectral data situation described above is one such
example of this scenario; due to the differing sensor measurements at
different times, a direct comparison of images is impossible.

Let $\I$ denote our parameter space, and let $X_{\alpha}$, with
$\alpha \in \I$, be the data in question. The elements of our data
set are fixed, but the graph changes depending on the parameter
$\alpha$. In other words, there is a known bijection between
$X_{\alpha}$ and $X_{\beta}$ for $\alpha, \beta \in \I$, but the
corresponding graph weights of $X$ have changed between the two
parameters. For a fixed $\alpha$, the diffusion maps framework developed in
\cite{coifman:diffusionMaps2006} gives a multiscale way of organizing
$X_{\alpha}$. If $X_{\alpha}$ has a low dimensional structure, then
the diffusion map will take $X_{\alpha}$ into a low dimensional
Euclidean space that characterizes its geometry. More specifically,
the diffusion mapping maps $X_{\alpha}$ into
a particular $\ell^2$ space in which the usual $\ell^2$ distance
corresponds to the diffusion distance on $X_{\alpha}$; in the case of
a low dimensional data set, the $\ell^2$ space can be ``truncated'' to
$\R^d$, with the standard Euclidean distance. However, for
different parameters $\alpha$ and $\beta$, the diffusion map may take
$X_{\alpha}$ and $X_{\beta}$ into different $\ell^2$ spaces, thus
meaning that one cannot take the standard $\ell^2$ distance between
the elements of these two spaces. Our contribution here is to
generalize the diffusion maps framework so that it works independently
of the parameter $\alpha$. In particular, we derive formulas for
the distance between points in different embeddings that are in terms
of the individual diffusion maps of each space. It is even possible to
define a mapping from one embedding to the other, so that after
applying this mapping the standard $\ell^2$ distance can once again
be used to compute diffusion distances. In particular, this additional
mapping gives a common parameterization of the data across all of $\I$
that characterizes the evolving intrinsic geometry of the data. Once
this generalized framework has been
established, we are able to define a global distance between all of
$X_{\alpha}$ and $X_{\beta}$ based on the behavior of the diffusions
within each data set. This distance in turn allows one to model the
global behavior of $X_{\alpha}$ as it changes over $\I$.

Earlier results that use diffusion maps to
compare two data sets can be found in
\cite{coifman:compareSystems2005}. Furthermore, there is recent work contained in
\cite{lee:multiscaleTimeSeriesGraphs2011} that also involves combining
diffusion geometry principles via tree structures with evolving graphs. In
\cite{abdallah:diffusionThesis2010}, the author considers the case of an
evolving Riemannian manifold on which a diffusion process is spreading
as the manifold evolves. In our work, we separate out the two
processes, effectively using the diffusion process to organize the
evolution of the data. Also tangentially related to this work are the
results contained in \cite{memoli:shapeDistance2011} on shape
analysis, in which shapes are compared via their heat kernels. More
generally, this paper fits into the larger class of research that
utilizes nonlinear mappings into low dimensional spaces in order to
organize potentially high dimensional data; examples include locally
linear embedding (LLE) \cite{roweis:lle2000}, ISOMAP
\cite{tenenbaum:isomap2000}, Hessian LLE \cite{donoho:hessianlle2003},
Laplacian eigenmaps \cite{belkin:laplacianEigen2003}, and the
aforementioned diffusion maps \cite{coifman:diffusionMaps2006}.

An outline of this paper goes as follows: in the next section, we take
care of some notation and review the diffusion mapping first presented
in \cite{coifman:diffusionMaps2006}. In Section \ref{sec: graph
  distance} we generalize the diffusion distance for a data set that
changes over some parameter space, and show that it can be
computed in terms the spectral embeddings of the corresponding
diffusion operators. We also show how to map each of the embeddings
into one common embedding in which the $\ell^2$ distance is equal to
the diffusion distance. The global diffusion distance between graphs
is defined in Section \ref{sec: two graph distances}; it is also seen
to be able to be computed in terms of the eigenvalues and
eigenfunctions of the relevant diffusion operators. In Section
\ref{sec: convergence of finite approximations} we set up and state two
random sampling theorems, one for the diffusion distance and one for
the global diffusion distance. The proofs of these theorems are given
in \ref{sec: proof of finite graph approximation estimate}. Section
\ref{sec: applications} contains some applications, and we conclude
with some remarks and possible future directions in Section \ref{sec:
  conclusion}. 

\section{Notation and preliminaries} \label{sec: notation and preliminaries}

In this section we introduce some basic notation and review certain
preliminary results that will motivate our work.

\subsection{Notation}

Let $\R$ denote the real numbers and let $\N \triangleq \{1, 2, 3,
\ldots \}$ be the natural numbers. Often we will use constants that depend on certain
variables or parameters. We let $C(\cdot)$, $C_1(\cdot)$,
$C_2(\cdot)$, etc, denote these constants; note that they can change
from line to line.

We recall some basic notation from operator theory. Let $\HH$ be a
real, separable Hilbert space with scalar product $\langle \cdot,
\cdot \rangle$ and norm $\| \cdot \|$. Let $A: \HH \rightarrow \HH$ be
a bounded, linear operator, and let $A^*$ be its adjoint. The operator
norm of $A$ is defined as: 
\begin{equation*}
\| A \| \triangleq \sup_{\| f \| = 1} \| Af \|.
\end{equation*}
A bounded operator $A$ is Hilbert-Schmidt if 
\begin{equation*}
\sum_{i \geq 1} \| Ae^{(i)} \|^2 < \infty
\end{equation*}
for some (and hence any) Hilbert basis $\{ e^{(i)} \}_{i \geq 1}$. The
space of Hilbert-Schmidt operators is also a Hilbert space with scalar
product
\begin{equation*}
\langle A, B \rangle_{HS} \triangleq \sum_{i \geq 1} \langle Ae^{(i)}, Be^{(i)} \rangle.
\end{equation*}
We denote the corresponding norm as $\| \cdot \|_{HS}$. Note that if
an operator is Hilbert-Schmidt, then it is compact. A Hilbert-Schmidt
operator is trace class if
\begin{equation*}
\sum_{i \geq 1} \langle \sqrt{A^*A} e^{(i)}, e^{(i)} \rangle < \infty
\end{equation*}
for some (and hence any) Hilbert basis $\{ e^{(i)} \}_{i \geq 1}$. For any
trace class operator $A$, we have
\begin{equation*}
\mathrm{Tr}(A) \triangleq \sum_{i \geq 1} \langle Ae^{(i)}, e^{(i)} \rangle < \infty,
\end{equation*}
where $\mathrm{Tr}(A)$ is called the trace of $A$. The space of trace
class operators is a Banach space endowed with the norm
\begin{equation*}
\| A \|_{TC} \triangleq \mathrm{Tr}(\sqrt{A^*A}).
\end{equation*}
Note that the different operator norms are related as follows:
\begin{equation*}
\| A \| \leq \| A \|_{HS} \leq \| A \|_{TC}.
\end{equation*}
For more information on trace class operators, Hilbert Schmidt
operators and related topics, we refer the reader to
\cite{simon:traceIdeals2005}. 

\subsection{Diffusion maps} \label{sec: diffusion maps}

In this section we consider just a single data set that does not
change and review the notion of diffusion maps on
this data set. We assume that we are given a measure space $(X, \mu)$, consisting of
data points $X$ that are distributed according to $\mu$. We also have a positive,
symmetric kernel $k: X \times X \rightarrow \R$ that encodes how
similar two data points are. From $X$ and $k$, one can construct a
weighted graph $\Gamma \triangleq (X,k)$, in which the vertices of
$\Gamma$ are the data points $x \in X$, and the weight of the edge
$xy$ is given by $k(x,y)$. 

The diffusion maps framework developed in \cite{coifman:diffusionMaps2006} gives a
multiscale organization of the data set $X$. Additionally, if $X
\subset \R^d$ is high dimensional, yet lies on a low dimensional manifold, the
diffusion map gives an embedding into Euclidean space that parameterizes the
data in terms of its intrinsic low dimensional geometry. The idea is
that the kernel $k$ should only measure local similarities within $X$
at small scales, so as to be able to ``follow'' the low dimensional
structure. The diffusion map then pieces together the local
similarities via a random walk on $\Gamma$.

Define the density, $m: X \rightarrow \R$, as
\begin{equation} \label{eqn: degree map}
m(x) \triangleq \int\limits_X k(x,y) \, d\mu (y), \quad \text{for all
} x \in X.
\end{equation}
We assume that the density $m$ satisfies
\begin{equation} \label{eqn: degree bounds}
m(x) > 0, \quad \text{for } \mu \text{ a.e. } x \in X,
\end{equation}
and
\begin{equation} \label{eqn: degree map in L1}
m \in L^1(X, \mu).
\end{equation}
Given \eqref{eqn: degree bounds}, the weight function
\begin{equation*}
p(x,y) \triangleq \frac{k(x,y)}{m(x)}
\end{equation*}
is well defined for $\mu \otimes \mu$ almost every $(x,y) \in X \times X$. Although
$p$ is no longer symmetric, it does satisfy the following useful
property:
\begin{equation*}
\int\limits_X p(x,y) \, d\mu (y) = 1, \quad \text{for } \mu
\text{ a.e. } x \in X.
\end{equation*}
Therefore we can view $p$ as the transition kernel of a Markov chain
on $X$. Equivalently, if $p \in L^2(X \times X, \mu \otimes \mu)$, the integral operator
$P: L^2(X, \mu) \rightarrow L^2(X, \mu)$, defined as
\begin{equation*}
(Pf)(x) \triangleq \int\limits_X p(x,y) f(y) \, d\mu (y), \quad \text{for all }
f \in L^2(X, \mu),
\end{equation*}
is a diffusion operator. In particular, the value $p(x,y)$ represents
the probability of transition in one time step from the vertex $x$ to
the vertex $y$, which is proportional to the edge weight $k(x,y)$. For
$t \in \N$, let $p^{(t)}(x,y)$ represent the probability of transition in
$t$ time steps from the node $x$ to the node $y$; note that $p^{(t)}$ is
the kernel of the operator $P^t$. As shown in
\cite{coifman:diffusionMaps2006}, running the Markov chain forward, or
equivalently taking powers of $P$, reveals relevant geometric
structures of $X$ at different scales. In particular, small powers of
$P$ will segment the data set into several smaller clusters. As $t$ is
increased and the Markov chain diffuses across the
graph $\Gamma$, the clusters evolve and merge together until in the
limit as $t \rightarrow \infty$ the data set is grouped into one
cluster (assuming the graph is connected).

The phenomenon described above can be encapsulated by the diffusion
distance at time $t$ between two vertices $x$ and $y$ in the graph
$\Gamma$. In order to define the diffusion distance, we first note
that the Markov chain constructed above has the stationary
distribution $\pi: X \rightarrow \R$, where
\begin{equation*}
\pi (x) = \frac{m(x)}{\int_X m(y) \, d\mu (y)}.
\end{equation*}
Combining \eqref{eqn: degree bounds} and \eqref{eqn: degree map in L1}
we see that $\pi (x)$ is well defined for $\mu \text{ a.e. } x \in
X$. The diffusion distance between $x,y \in X$ is then defined as:
\begin{eqnarray*}
\widetilde{D}^{(t)} (x,y)^2 & \triangleq & \left\| p^{(t)}(x,\cdot) -
  p^{(t)}(y,\cdot) \right\|_{L^2(X, d\mu/\pi)}^2 \\
& = & \int\limits_X \left( p^{(t)}(x,u) - p^{(t)}(y,u) \right)^2 \, \frac{d\mu(u)}{\pi (u)}.
\end{eqnarray*}
A simplified formula for the diffusion distance can be found by
considering the spectral decomposition of $P$. Define the kernel $a: X
\times X \rightarrow \R$ as
\begin{equation*}
a(x,y) \triangleq \frac{\sqrt{m(x)}}{\sqrt{m(y)}} p(x,y) =
\frac{k(x,y)}{\sqrt{m(x)} \sqrt{m(y)}}, \quad \text{for } \mu \otimes
\mu \text{ a.e. } (x,y) \in X \times X.
\end{equation*}
If $a \in L^2(X \times X, \mu \otimes \mu)$, then $P$ has
a discrete set of eigenfunctions $\{ \upsilon^{(i)} \}_{i \geq 1}$ with corresponding
eigenvalues $\{ \lambda^{(i)} \}_{i \geq 1}$. It can then be shown that
\begin{equation} \label{eqn: simplified diff dist}
\widetilde{D}^{(t)} (x,y)^2 = \sum_{i \geq 1}
\left(\lambda^{(i)}\right)^{2t} \left( \upsilon^{(i)} (x) - \upsilon^{(i)} (y) \right)^2.
\end{equation}
Inspired by \eqref{eqn: simplified diff dist},
\cite{coifman:diffusionMaps2006} defines the diffusion map
$\Upsilon^{(t)}: X \rightarrow \ell^2$ at diffusion time $t$ to be:
\begin{equation*}
\Upsilon^{(t)} (x) \triangleq \left( \left(\lambda^{(i)}\right)^t
  \upsilon^{(i)} (x) \right)_{i \geq 1}. 
\end{equation*}
Therefore, the diffusion distance at time $t$ between $x,y \in X$
is equal to the $\ell^2$ norm of the difference between
$\Upsilon^{(t)}(x)$ and $\Upsilon^{(t)}(y)$:
\begin{equation*}
\widetilde{D}^{(t)}(x,y) = \left\| \Upsilon^{(t)}(x) - \Upsilon^{(t)}(y) \right\|_{\ell^2}.
\end{equation*}

One can also define a second diffusion distance in terms of the
symmetric kernel $a$ as opposed to the asymmetric kernel $p$. In
particular, define the operator $A: L^2(X, \mu) \rightarrow L^2(X,
\mu)$ as
\begin{equation*}
(Af)(x) \triangleq \int\limits_X a(x,y) f(y) \, d\mu(y), \quad \text{for all }
f \in L^2(X, \mu).
\end{equation*}
Like the diffusion operator $P$, the operator $A$ and its powers,
$A^t$, reveal the relevant geometric structures of the data set
$X$. Letting $a^{(t)}: X \times X \rightarrow \R$ denote the kernel of
the operator $A^t$, we can define another diffusion distance $D^{(t)}: X
\times X \rightarrow \R$ as follows:
\begin{eqnarray*}
D^{(t)}(x,y)^2 & \triangleq & \left\| a^{(t)}(x, \cdot) - a^{(t)}(y, \cdot)
\right\|_{L^2(X, \mu)}^2 \\
& = & \int\limits_X \left( a^{(t)}(x, u) - a^{(t)}(y, u) \right)^2 \, d\mu(u).
\end{eqnarray*}
As before, we consider the spectral decomposition of $A$. Let $\{
\lambda^{(i)} \}_{i \geq 1}$ and $\{ \psi^{(i)} \}_{i \geq 1}$ denote the
eigenvalues and eigenfunctions of $A$ (indeed, the nonzero eigenvalues
of $P$ and $A$ are the same), and define the diffusion map
$\Psi^{(t)}: X \rightarrow \ell^2$ (corresponding to $A$) as
\begin{equation*}
\Psi^{(t)}(x) = \left( \left(\lambda^{(i)}\right)^t \psi^{(i)}(x) \right)_{i \geq 1}.
\end{equation*}
Then, under the same assumptions as before, we have
\begin{equation} \label{eqn: symmetric diffusion maps formula}
D^{(t)}(x,y)^2 = \left\| \Psi^{(t)}(x) - \Psi^{(t)}(y) \right\|_{\ell^2}^2 = \sum_{i
  \geq 1} \left(\lambda^{(i)}\right)^t \left( \psi^{(i)}(x) -
  \psi^{(i)}(y) \right)^2.
\end{equation}

We make a few remarks concerning the differences between the two
formulations. First, we note that the original diffusion distance
$\widetilde{D}^{(t)}$ is defined as an $L^2$ distance under the weighted
measure $d\mu/\pi$. The second diffusion distance, $D^{(t)}$, due to the
symmetric normalization built into the kernel $a$, is defined only in
terms of the underlying measure $\mu$. Furthermore, the eigenfunctions
of $A$ are orthogonal, unlike the eigenfunctions of $P$. Finally, as
we have already noted, the eigenvalues of $P$ and $A$ are in fact the
same, and furthermore they are contained in $(-1, 1]$. If the graph
$\Gamma$ is connected, then the eigenfunction of $P$ with eigenvalue one is simply
the function that maps every element of $X$ to one. The corresponding
eigenfunction of $A$ though is the square root of the density, i.e.,
$\sqrt{m(x)}$. Thus, while both versions of the diffusion distance
merge smaller clusters into large clusters as $t$ grows, $\widetilde{D}^{(t)}$
will merge every data point into the same cluster in the limit as $t
\rightarrow \infty$, while $D^{(t)}$ will reflect the behavior of
the density $m$ in the limit as $t \rightarrow \infty$. 

Finally, recalling the discussion at the beginning of this section and
regardless of the particular operator used ($P$ or $A$), if $X$ has
a low dimensional structure to it, then the number of significant
eigenvalues will be small. In this case, from \eqref{eqn: symmetric
  diffusion maps formula} it is clear that one can in fact map $X$
into a low dimensional Euclidean space via the dominant
eigenfunctions while nearly preserving the diffusion distance. 

\section{Generalizing the diffusion distance for changing
  data} \label{sec: graph distance}

In this section we generalize the diffusion maps framework for data
sets with input parameters.

\subsection{The data model}

We now turn our attention to the original problem introduced at the
beginning of this paper. In its most general form, we are given a
parameter space $\I$ and a data set $X_{\alpha}$ that depends on
$\alpha \in \I$. The data points of $X_{\alpha}$ are given by
$x_{\alpha}$. The parameter space $\I$ can
be continuous, discrete, or completely arbitrary. Recall from the
introduction that we are working under the assumption that there is an
a priori known bijective correspondence between
$X_{\alpha}$ and $X_{\beta}$ for any $\alpha, \beta \in \I$ (in
\ref{sec: non-bijective correspondence} we discuss relaxing this
assumption).

We consider the following model throughout the remainder of this
paper. We are given a single measure space $(X,\mu)$ that we think of
as changing over $\I$. The changes in $X$ are encoded by a family of
metrics $d_{\alpha}: X \times X \rightarrow \R$, so that for each
$\alpha \in \I$ we have a metric measure space $X_{\alpha} =
(X,\mu,d_{\alpha})$. The measure $\mu$ here represents some underlying
distribution of the points in $X$ that does not change over
$\I$. There is no a priori assumption of a universal metric
$d: (X \times \I) \times (X \times \I) \rightarrow \R$ that can be
used to discern the distance between points taken from $X_{\alpha}$
and $X_{\beta}$ for arbitrary $\alpha, \beta \in \I$, $\alpha \neq
\beta$. 

\begin{remark} \label{rem: universal metric}
If such a universal metric does exist, then one
could still use the techniques developed in this paper, by defining
the metrics $d_{\alpha}$ in terms of the restriction of the universal
metric $d$ to the parameter $\alpha$. Alternatively, the original
diffusion maps machinery could be used by defining a kernel $k: (X
\times \I) \times (X \times \I) \rightarrow \R$ in terms of the
universal metric $d$. Further discussion along these lines is given in
Section \ref{sec: historical graph}.
\end{remark}

\subsection{Defining the diffusion distance on a family of
  graphs} \label{sec: defining the diffusion distance on Xtau}

Our goal is to reveal the relevant geometric structures of $X$
across the entire parameter space $\I$, and to furthermore have a way of
comparing structures from one parameter to other structures derived from a second
parameter. To do so, we shall generalize the diffusion distance so that we
can compare diffusions derived from different parameters. For each
instance of the data $X_{\alpha} = (X,\mu,d_{\alpha})$, we derive a
kernel $k_{\alpha}: X \times X \rightarrow \R$. The first step is to once
again consider each pairing $X$ and $k_{\alpha}$ as a weighted
graph, which we denote as $\Gamma_{\alpha} \triangleq (X, k_{\alpha})$.

Updating our notation for this dynamic setting, for each parameter $\alpha
\in \I$ we have the density $m_{\alpha}: X \rightarrow \R$ defined as 
\begin{equation*}
m_{\alpha}(x) \triangleq \int\limits_X k_{\alpha}(x,y) \, d\mu(y),
\quad \text{for all } \alpha \in \I, \enspace x \in X.
\end{equation*}
For reasons that shall become clear later, we slightly strengthen the
assumptions on $m_{\alpha}$ as compared to those in equations
\eqref{eqn: degree bounds} and \eqref{eqn: degree map in L1}. In
particular, we assume that
\begin{equation*}
m_{\alpha}(x) > 0, \quad \text{for all } \alpha \in \I, \enspace x \in X,
\end{equation*}
and
\begin{equation*}
m_{\alpha} \in L^1(X,\mu), \quad \text{for all } \alpha \in \I.
\end{equation*}
We then define two classes of kernels $a_{\alpha}: X \times X
\rightarrow \R$ and $p_{\alpha}: X \times X \rightarrow
\R$ in the same manner as earlier:
\begin{equation} \label{eqn: original atau}
a_{\alpha}(x,y) \triangleq \frac{k_{\alpha}(x,y)}{\sqrt{m_{\alpha}(x)} \sqrt{m_{\alpha}(y)}}, \quad
\text{for all } \alpha \in \I, \enspace (x,y) \in X \times X,
\end{equation}
and
\begin{equation*}
p_{\alpha}(x,y) \triangleq \frac{k_{\alpha}(x,y)}{m_{\alpha}(x)}, \quad
\text{for all } \alpha \in \I, \enspace (x,y) \in X \times X.
\end{equation*}
Assume that $a_{\alpha}, p_{\alpha} \in L^2(X \times X, \mu \otimes
\mu)$. Their corresponding integral operators are given by
$A_{\alpha}:L^2(X,\mu) \rightarrow L^2(X,\mu)$ and
$P_{\alpha}:L^2(X,\mu) \rightarrow L^2(X,\mu)$, where
\begin{equation} \label{eqn: original operator Atau}
(A_{\alpha} f)(x) \triangleq \int\limits_X a_{\alpha}(x,y)
f(y) \, d\mu(y), \quad \text{for all } \alpha \in \I, \enspace f \in L^2(X, \mu),
\end{equation}
and
\begin{equation*}
(P_{\alpha} f)(x) \triangleq \int\limits_X p_{\alpha}(x,y)
f(y) \, d\mu(y), \quad \text{for all } \alpha \in \I, \enspace f \in L^2(X, \mu).
\end{equation*}
Finally, we let $a_{\alpha}^{(t)}$ and $p_{\alpha}^{(t)}$ denote the
kernels of the integral operators $A_{\alpha}^t$ and $P_{\alpha}^t$,
respectively.

Returning to the task at hand, in order to compare $\Gamma_{\alpha}$
with $\Gamma_{\beta}$, it is possible to use the operators $A_{\alpha}$ and
$A_{\beta}$ or $P_{\alpha}$ and $P_{\beta}$. We choose to perform our
analysis using the symmetric operators, as it shall simplify certain
things. For now, consider the function $a_{\alpha}(x, \cdot)$ for a fixed $x \in X$. We
think of this function in the following way. Consider the graph
$\Gamma_{\alpha}$, and imagine dropping a unit of mass on the node
$x$ and allowing it to spread, or diffuse, throughout
$\Gamma_{\alpha}$. After one unit of time, the amount of mass that has
spread from $x$ to some other node $y$ is proportional to
$a_{\alpha}(x,y)$. Similarly, if we want to let the mass spread
throughout the graph for a longer period of time, we can, and the
amount of mass that has spread from $x$ to $y$ after $t$
units of time is then proportional to $a_{\alpha}^{(t)}(x, y)$. The
diffusion distance at time $t$, which is the $L^2$ norm of
$a_{\alpha}^{(t)}(x, \cdot) - a_{\alpha}^{(t)}(y, \cdot)$, is then
comparing the behavior of the diffusion centered at $x$ with the
behavior of the diffusion centered at $y$. We wish to extend
this idea for different parameters $\alpha$ and $\beta$. In other words,
we wish to have a meaningful distance between $x$ at parameter $\alpha$ and
$y$ at parameter $\beta$ that is based on the same principle of measuring how
their respective diffusions behave.

Our solution is to generalize the diffusion distance in the following
way. For each diffusion time $t \in \N$, we define a dynamic
diffusion distance $D^{(t)}: (X \times \I) \times (X \times \I)
\rightarrow \R$ as follows. Let $x_{\alpha} \triangleq (x,\alpha) \in X
\times \I$, and set
\begin{eqnarray*}
D^{(t)}(x_{\alpha}, y_{\beta})^2 & \triangleq & \left\| a_{\alpha}^{(t)}(x, \cdot) -
  a_{\beta}^{(t)}(y, \cdot) \right\|_{L^2(X, \mu)}^2 \\
& = & \int\limits_X \left( a_{\alpha}^{(t)}(x,u) - a_{\beta}^{(t)}(y,u) \right)^2
\, d\mu(u).
\end{eqnarray*}
This notion of distance can be thought of as comparing how the
neighborhood of $x_{\alpha}$ differs from the neighborhood of
$y_{\beta}$. In particular, if we are comparing the same data point
but at different parameters, for example $x_{\alpha}$ and $x_{\beta}$, the
diffusion distance between them will be small if their neighborhoods
do not change much from $\alpha$ to $\beta$. On the other hand, if say
a large change occurs at $x$ at parameter $\beta$, then the
neighborhood of $x_{\beta}$ should differ from the neighborhood of
$x_{\alpha}$ and so they will have a large diffusion distance between them.

Some more intuition about the quantity $D^{(t)}(x_{\alpha}, y_{\beta})$
can be derived from the triangle inequality. In particular, one
application of it gives
\begin{equation*}
D^{(t)}(x_{\alpha}, y_{\beta}) \leq D^{(t)}(x_{\alpha}, x_{\beta}) +
D^{(t)}(x_{\beta}, y_{\beta}).
\end{equation*}
Thus we see that $D^{(t)}(x_{\alpha}, y_{\beta})$ is bounded from above
by the change in $x$ from $\alpha$ to $\beta$ (i.e. the
quantity $D^{(t)}(x_{\alpha}, x_{\beta})$) plus the diffusion distance
between $x$ and $y$ in the graph $\Gamma_{\beta}$ (i.e. the quantity
$D^{(t)}(x_{\beta}, y_{\beta})$).

\begin{remark}
As noted earlier, we have chosen to generalize the diffusion
distance in terms of the symmetric kernels $a_{\alpha}$ as opposed to the
asymmetric kernels $p_{\alpha}$. The primary reason for this choice is
that when using the kernel $p_{\alpha}$ to compute the diffusion
distance between $x$ and $y$, we must use the weighted
measure $d\mu/\pi_{\alpha}$, where $\pi_{\alpha}$ denotes the
stationary distribution of the Markov chain on $\Gamma_{\alpha}$. Thus,
when computing the diffusion distance between $x_{\alpha}$ and
$y_{\beta}$, one must incorporate this weighted measure as well. Since
the stationary distribution will invariably change from $\alpha$ to
$\beta$, the most natural generalization in this case would be:
\begin{equation*}
\widetilde{D}^{(t)}(x_{\alpha},y_{\beta})^2 \triangleq \int\limits_X
\left( \frac{p_{\alpha}^{(t)}(x,u)}{\sqrt{\pi_{\alpha}(u)}} -
    \frac{p_{\beta}^{(t)}(y,u)}{\sqrt{\pi_{\beta}(u)}} \right)^2 \, d\mu(u).
\end{equation*}

Alternatively, in \cite{hirn:bistochasticKernel2012}, we describe how to
construct a bi-stochastic kernel $b: X \times X \rightarrow \R$ from a
more general affinity function. The kernel is bi-stochastic under a
particular weighted measure $\Omega^2 \mu$, where $\Omega: X \rightarrow
\R$ is derived from the affinity function. In this case, one can
define yet another alternate diffusion distance as:
\begin{equation*}
\widehat{D}^{(t)}(x_{\alpha},y_{\beta})^2 \triangleq \int\limits_X
\left( b_{\alpha}^{(t)}(x,u) \, \Omega_{\alpha}(u) -
  b_{\beta}^{(t)}(y,u) \, \Omega_{\beta}(u) \right)^2 \, d\mu(u).
\end{equation*}

In either case, the results that follow can be translated for these
particular diffusion distances by following the same arguments and
making minor modifications where necessary.
\end{remark}

\subsection{Diffusion maps for $\G = \{ \Gamma_{\alpha} \}_{\alpha \in
    \I}$} \label{sec: diffusion maps for Gammatau}

Analogous to the diffusion distance for a single graph $\Gamma = (X,k)$, we can
write the diffusion distance for $\G \triangleq \{ \Gamma_{\alpha} \}_{\alpha \in
  \I}$ in terms the spectral decompositions of $\{ A_{\alpha}
\}_{\alpha \in \I}$. We first collect the following mild, but
necessary, assumptions, some of which have already been stated.
\begin{assumption} \label{assumption 1}
We assume the following properties:
\begin{enumerate}
\item
$(X,\mu)$ is a $\sigma$-finite measure space and $L^2(X,\mu)$ is
separable.
\item
The kernel $k_{\alpha}$ is positive definite and symmetric for all $\alpha
\in \I$.
\item
For each $\alpha \in \I$, $m_{\alpha} \in L^1(X, \mu)$ and $m_{\alpha} > 0$.
\item \label{itm: trace class assumption}
For any $\alpha \in \I$, the operator $A_{\alpha}$ is trace class.
\end{enumerate}
\end{assumption}

A few remarks concerning the assumed properties. First, the reader may
have noticed that we replaced the assumption that $k_{\alpha}$ be
positive with the stronger assumption that it is positive
definite. This combined with the third property that $m_{\alpha}(x) >
0$ for all $x \in X$, implies that $a_{\alpha}$ is also positive
definite. Thus the operators $A_{\alpha}$ are positive and self
adjoint.

If one wished to revert back to the weaker assumption that $k_{\alpha}$
merely be positive, then the following adjustment could be
made. Clearly the symmetrically normalized kernel $a_{\alpha}$ will
still be positive, but the operator $A_{\alpha}$ may not be. However, one
could replace $A_{\alpha}$, for each $\alpha \in \I$, with the graph
Laplacian $L_{\alpha}: L^2(X,\mu) \rightarrow L^2(X,\mu)$, which is
defined as
\begin{equation*}
L_{\alpha} \triangleq \frac{1}{2}(I - A_{\alpha}),
\end{equation*}
where $I: L^2(X,\mu) \rightarrow L^2(X,\mu)$ is the identity
operator. The graph Laplacian $L_{\alpha}$ is a positive operator with
eigenvalues contained in $[0,1]$. The analysis that follows would
still apply with only minor adjustments. 

The fourth item that $A_{\alpha}$ be trace
class plays a key role in the results of this section, and itself implies
that these operators are Hilbert-Schmidt and so also compact. Thus, as a further
consequence, $a_{\alpha} \in L^2(X \times X, \mu \otimes \mu)$ for each
$\alpha \in \I$. Ideally, one would replace the fourth item with a
condition on the kernel $k_{\alpha}$ that implies that $A_{\alpha}$ is
trace class. Unfortunately, unlike the case of Hilbert Schmidt
operators, there is not a simple theorem of this nature. Further
information on trace class integral operators, as well as various
results, can be found in \cite{simon:traceIdeals2005,
  brislawn:kernelsTraceOperators1988,
  brislawn:traceableIntegralKernels1991}. 

We note that assumptions three and four are both satisfied if for each
$\alpha \in \I$ the kernel $k_{\alpha}$ is continuous, bounded from
above and below, and if the measure of $X$ is finite. That is, if for each $\alpha$,
\begin{equation*}
0 < C_1(\alpha) \leq k_{\alpha}(x,y) \leq C_2(\alpha) < \infty, \quad
\text{for all } (x,y) \in X \times X,
\end{equation*}
and
\begin{equation*}
\mu(X) < \infty,
\end{equation*}
then we can derive assumptions three and four.

As an immediate consequence of the properties contained in Assumption
\ref{assumption 1}, we see from the Spectral Theorem that for each $\alpha$ the operator
$A_{\alpha}$ has a countable collection of positive eigenvalues and orthonormal
eigenfunctions that form a basis for $L^2(X,\mu)$. Let $\{
\lambda_{\alpha}^{(i)} \}_{i \geq 1}$ and $\{\psi_{\alpha}^{(i)} \}_{i
  \geq 1}$ be the eigenvalues and a set of orthonormal eigenfunctions
of $A_{\alpha}$, respectively, so that
\begin{equation*}
(A_{\alpha} \psi_{\alpha}^{(i)})(x) = \lambda_{\alpha}^{(i)}
\psi_{\alpha}^{(i)}(x), \quad \text{for } \mu \text{ a.e. } x \in X,
\end{equation*}
and 
\begin{equation*}
\langle \psi_{\alpha}^{(i)}, \psi_{\alpha}^{(j)} \rangle_{L^2(X, \mu)} = \delta (i - j), \quad
\text{for all } i,j \geq 1.
\end{equation*}
Furthermore, as noted in \cite{coifman:diffusionMaps2006}, the
eigenvalues of $P_{\alpha}$ are bounded in absolute value by one, with at
least one eigenvalue equaling one. Since the eigenvalues of $A_{\alpha}$ and
$P_{\alpha}$ are the same, we also have
\begin{equation*}
1 = \lambda_{\alpha}^{(1)} \geq \lambda_{\alpha}^{(2)} \geq
\lambda_{\alpha}^{(3)} \geq \ldots, 
\end{equation*}
where $\lambda_{\alpha}^{(i)} \rightarrow 0$ as $i \rightarrow \infty$.

As with the original diffusion distance defined on a single data set,
our generalized notion of the diffusion distance for dynamic data
sets has a simplified form in terms of the spectral decompositions of
the relevant operators. 

\begin{theorem} \label{thm: simplified generalized diffusion distance}
Let $(X, \mu)$ be a measure space and $\left\{ k_{\alpha} \right\}_{\alpha
  \in \I}$ a family of kernels defined on $X$. If $(X, \mu)$ and
$\{k_{\alpha}\}_{\alpha \in \I}$
satisfy the properties of Assumption \ref{assumption 1}, then the
diffusion distance at time $t$ between $x_{\alpha}$ and $y_{\beta}$
can be written as:
\begin{align}
D^{(t)}(x_{\alpha}, y_{\beta})^2 = &\sum_{i \geq 1} \left(
  \lambda_{\alpha}^{(i)} \right)^{2t} \psi_{\alpha}^{(i)}(x)^2 + \sum_{j \geq 1}
\left( \lambda_{\beta}^{(j)} \right)^{2t} \psi_{\beta}^{(j)}(y)^2
\nonumber \\
&- 2 \sum_{i,j \geq 1} \left( \lambda_{\alpha}^{(i)} \right)^t \left(
  \lambda_{\beta}^{(j)} \right)^t \psi_{\alpha}^{(i)}(x) \,
\psi_{\beta}^{(j)}(y) \, \langle \psi_{\alpha}^{(i)}, \psi_{\beta}^{(j)}
\rangle_{L^2(X,\mu)}, \label{eqn: generalized diff dist formula}
\end{align}
where for each pair $(\alpha, \beta) \in \I \times \I$, equation
\eqref{eqn: generalized diff dist formula} converges in $L^2(X \times
X, \mu \otimes \mu)$. If, additionally, $k_{\alpha}$ is
continuous for each $\alpha \in \I$, $X \subseteq \R^d$ is closed, and
$\mu$ is a strictly positive Borel measure, then \eqref{eqn:
  generalized diff dist formula} holds for all $(x,y) \in X \times X$. 
\end{theorem}

Notice that equation \eqref{eqn: generalized diff dist formula}  is in fact an
extension of the formula given for the diffusion distance on a single
data set. Indeed, if one were to take $x_{\alpha}$ and
$y_{\beta} = y_{\alpha}$, the formula given in \eqref{eqn: generalized
  diff dist formula} would simplify to \eqref{eqn: symmetric diffusion
  maps formula} with the underlying kernel taken to be
$k_{\alpha}$. Thus, it is natural to define the diffusion map $\Psi_{\alpha}^{(t)}:X
\rightarrow \ell^2$ for the parameter $\alpha$ and diffusion time $t$ as
\begin{equation} \label{eqn: alpha diffusion map}
\Psi_{\alpha}^{(t)}(x) \triangleq \left( \left( \lambda_{\alpha}^{(i)} \right)^t
  \psi_{\alpha}^{(i)}(x) \right)_{i \geq 1}.
\end{equation}
For $v \in \ell^2$, let $v[i]$ denote the $i^{\text{th}}$ element of
the sequence $u$. Using \eqref{eqn: alpha diffusion map}, one can
write equation \eqref{eqn: generalized diff dist formula} as
\begin{equation}
D^{(t)}(x_{\alpha}, y_{\beta})^2 = \big\|\Psi_{\alpha}^{(t)}(x)
\big\|_{\ell^2}^2 + \big\|\Psi_{\beta}^{(t)}(y)\big\|_{\ell^2}^2 - 2
\sum_{i,j \geq 1} \Psi_{\alpha}^{(t)}(x)[i] \,
\Psi_{\beta}^{(t)}(y)[j] \, \langle \psi_{\alpha}^{(i)},
\psi_{\beta}^{(j)} \rangle_{L^2(X, \mu)}. \label{eqn: generalized diff
  maps formula}
\end{equation}
In particular, one has in general that
\begin{equation*}
D^{(t)}(x_{\alpha},y_{\beta}) \neq \left\|\Psi_{\alpha}^{(t)}(x) - \Psi_{\beta}^{(t)}(y)\right\|_{\ell^2}.
\end{equation*}
Intuitively, the thing to take away from this discussion is that for
each parameter $\alpha \in \I$, the diffusion map
$\Psi_{\alpha}^{(t)}$ maps $X$ into an $\ell^2$ space that itself also
depends on $\alpha$. The $\ell^2$ embedding corresponding to $\alpha$
is not the same as the $\ell^2$ embedding corresponding to $\beta \in
\I$, but equation \eqref{eqn: generalized diff maps formula} gives a way of
computing distances between the different $\ell^2$ embeddings.

Also, once again paralleling the
original diffusion distance, we see that if the eigenvalues of
$A_{\alpha}$ and $A_{\beta}$ decay sufficiently fast, then
the diffusion distance can be well approximated by a small, finite
number of eigenvalues and eigenfunctions of these two operators. In
particular, we need only map $\Gamma_{\alpha}$ and $\Gamma_{\beta}$ into finite
dimensional Euclidean spaces.

\begin{proof}[Proof of Theorem \ref{thm: simplified generalized diffusion distance}]
We first use the fact that for each $\alpha \in \I$, $A_{\alpha}$ is a
positive, self-adjoint, trace class
operator. Thus $A_{\alpha}$ is Hilbert-Schmidt, and so we know that
for each $\alpha \in \I$ (see, for example, Theorem 2.11 from
\cite{simon:traceIdeals2005}),
\begin{equation} \label{eqn: Diagonalize a1}
a_{\alpha}^{(t)} (x,y) = \sum_{i \geq 1}
\left(\lambda_{\alpha}^{(i)}\right)^t \psi_{\alpha}^{(i)}(x) \,
\psi_{\alpha}^{(i)}(y), \quad \text{with convergence in } L^2(X \times
X, \mu \otimes \mu).
\end{equation}
If the additional assumptions hold that $k_{\alpha}$ is continuous,
$X$ is a closed subset of $\R^d$, and
$\mu$ is a strictly positive Borel measure, then by Mercer's Theorem
(see \cite{mercer:mercersTheorem1909, minh:mercersTheoremFeature2006})
equation \eqref{eqn: Diagonalize a1} will hold for all $(x,y) \in X \times X$. In this
case the proof can be easily amended to get the stronger result; we
omit the details. 

Expand the formula for $D^{(t)}(x_{\alpha}, y_{\beta})$ as follows:
\begin{equation}
D^{(t)}(x_{\alpha}, y_{\beta})^2 = \int\limits_X \left( a_{\alpha}^{(t)}
  (x,u)^2 - 2a_{\alpha}^{(t)}(x,u) \, a_{\beta}^{(t)}(y,u) +
  a_{\beta}^{(t)}(y,u)^2 \right) \, d\mu (u). \label{eqn: expand integrand of dt} 
\end{equation}
We shall evaluate each of the three terms in \eqref{eqn: expand
  integrand of dt} separately. For the cross term we have,
\begin{equation}
\int\limits_X a_{\alpha}^{(t)}(x,u) \, a_{\beta}^{(t)} (y,u) \, d\mu (u)
= \int\limits_X \left( \sum_{i,j \geq 1} \left(\lambda_{\alpha}^{(i)}\right)^t
\left(\lambda_{\beta}^{(j)}\right)^t \psi_{\alpha}^{(i)}(x) \,
\psi_{\beta}^{(j)}(y) \, \psi_{\alpha}^{(i)}(u) \, \psi_{\beta}^{(j)}(u) \right) \,
d\mu (u), \label{eqn: need to switch inside int and sum}
\end{equation}
with convergence in $L^2(X \times X, \mu \otimes \mu)$. At this point
we would like to switch the integral and the summation in line
\eqref{eqn: need to switch inside int and sum}; this can be done by
applying Fubini's Theorem, which requires one to show the
following:
\begin{equation} \label{eqn: suff condition to switch}
\sum_{i,j \geq 1} \int\limits_X \left|
  \left(\lambda_{\alpha}^{(i)}\right)^t
  \left(\lambda_{\beta}^{(j)}\right)^t \psi_{\alpha}^{(i)}(x) \,
  \psi_{\beta}^{(j)}(y) \, \psi_{\alpha}^{(i)}(u) \, \psi_{\beta}^{(j)}(u) \right| \, d\mu (u) < \infty.
\end{equation}
One can prove \eqref{eqn: suff condition to switch} for $\mu \otimes
\mu$ almost every $(x,y) \in X \times X$ through the use of
H\"{o}lder's Theorem and the fact that we assumed that $A_{\alpha}$ is a
trace class operator for each $\alpha \in \I$; we leave the details to
the reader. Thus for $\mu \otimes \mu$ almost every $(x,y) \in X
\times X$ we can switch the integral and the summation in line
\eqref{eqn: need to switch inside int and sum}, which gives:
\begin{equation}
\int\limits_X a_{\alpha}^{(t)}(x,u) \, a_{\beta}^{(t)}(y,u) \,
d\mu (u) = \sum_{i,j \geq 1} \left(\lambda_{\alpha}^{(i)}\right)^t
\left(\lambda_{\beta}^{(j)}\right)^t \psi_{\alpha}^{(i)}(x) \,
\psi_{\beta}^{(j)}(y) \, \langle \psi_{\alpha}^{(i)}, \psi_{\beta}^{(j)}
\rangle_{L^2(X, \mu)}, \label{eqn: need to switch int and sum}
\end{equation}
again with convergence in $L^2(X \times X, \mu \otimes \mu)$. A
similar calculation shows that, for each $\alpha \in \I$,
\begin{equation} \label{eqn: squared term 1}
\int\limits_X a_{\alpha}^{(t)}(x,u)^2 \, d\mu (u) = \sum_{i \geq 1}
\left(\lambda_{\alpha}^{(i)}\right)^{2t} \psi_{\alpha}^{(i)}(x)^2,
\quad \text{with convergence in } L^2(X,\mu).
\end{equation}
Combining equations \eqref{eqn: need to switch int and sum} and
\eqref{eqn: squared term 1} we arrive at the desired formula for
$D^{(t)}(x_{\alpha}, y_{\beta})$.
\end{proof}

\begin{remark} \label{rem: asymptotic diff dist}
One interesting aspect of the diffusion distance is its asymptotic
behavior as $t \rightarrow \infty$, and in particular that behavior
when each graph $\Gamma_{\alpha} \in \G$ is a 
connected graph. In this case, each operator $A_{\alpha}$ has
precisely one eigenvalue equal to one, and the corresponding
eigenfunction is the square root of the density (normalized), i.e.,
\begin{equation*}
1 = \lambda_{\alpha}^{(1)} > \lambda_{\alpha}^{(2)} \geq
\lambda_{\alpha}^{(3)} \geq \ldots, \quad \text{and} \quad
\psi_{\alpha}^{(1)} = \left.\sqrt{m_{\alpha}} \, \middle/ \,
  \left\|\sqrt{m_{\alpha}}\right\|_{L^2(X,\mu)}.\right.
\end{equation*}

To compute $\lim_{t \rightarrow \infty}
D^{(t)}(x_{\alpha},y_{\beta})$, we utilize equation \eqref{eqn:
  generalized diff dist formula} from Theorem \ref{thm: simplified
  generalized diffusion distance} and pull the limit as $t \rightarrow
\infty$ inside the summations. We justify the interchange of the limit
and the sum by utilizing the Dominated Convergence
Theorem. In particular, treat each sum as an integral over $\N$ with
the counting measure. Let us focus on the double summation in
\eqref{eqn: generalized diff dist formula}; the other two single
summations follow from similar arguments. For the double summation, we
have a sequence of functions
\begin{equation*}
f_t(i,j) \triangleq \left(\lambda_{\alpha}^{(i)}\right)^t
\left(\lambda_{\beta}^{(j)}\right)^t \psi_{\alpha}^{(i)}(x) \,
\psi_{\beta}^{(j)}(y) \, \langle \psi_{\alpha}^{(i)}, \psi_{\beta}^{(j)}
\rangle_{L^2(X, \mu)}.
\end{equation*}
We dominate the sequence $\{f_t\}_{t \geq 1}$ with the function
$g(i,j)$ as follows:
\begin{equation*}
\left|f_t(i,j)\right| \leq g(i,j) \triangleq \left| \lambda_{\alpha}^{(i)}
  \lambda_{\beta}^{(j)} \psi_{\alpha}^{(i)}(x) \, \psi_{\beta}^{(j)}(y) \right|.
\end{equation*}
We claim that $g$ is integrable over $\N \times \N$ with the counting
measure. To see this, first note:
\begin{equation*}
\sum_{i,j \geq 1} g(i,j) = \left(\sum_{i \geq 1} \left|
    \lambda_{\alpha}^{(i)} \psi_{\alpha}^{(i)}(x) \right| \right)
\left(\sum_{j \geq 1} \left| \lambda_{\beta}^{(j)}
    \psi_{\beta}^{(j)}(y) \right| \right).
\end{equation*}
Now define the function $h_{\alpha}: X \rightarrow \R$ as:
\begin{equation*}
h_{\alpha}(x) \triangleq \sum_{i \geq 1} \left|\lambda_{\alpha}^{(i)}
  \psi_{\alpha}^{(i)}(x) \right|.
\end{equation*}
Using Tonelli's Theorem, H\"{o}lder's Theorem, and the fact that
$A_{\alpha}$ is trace class, one can show that $h_{\alpha} \in
L^2(X,\mu)$. Thus, $h_{\alpha}(x) < \infty$ for $\mu$ almost every $x
\in X$. In particular, for $\mu \otimes \mu$ almost every $(x,y) \in X
\times X$, the function $g(i,j)$ is integrable. To conclude, the
Dominated Convergence Theorem holds, and for $\mu \otimes \mu$ almost
every $(x,y) \in X \times X$, we can interchange the summations and
the limit as $t \rightarrow \infty$.

From here, it is quite simple to show:
\begin{equation} \label{eqn: asymptotic diff dist}
\lim_{t \rightarrow \infty}D^{(t)}(x_{\alpha},y_{\beta})^2 =
\left(\psi_{\alpha}^{(1)}(x) - \psi_{\beta}^{(1)}(y) \right)^2 +
\psi_{\alpha}^{(1)}(x) \, \psi_{\beta}^{(2)}(y)
\left\| \psi_{\alpha}^{(1)} - \psi_{\beta}^{(1)} \right\|_{L^2(X,\mu)}^2.
\end{equation}
Recalling that the first eigenfunctions are simply the normalized
densities, we see that the asymptotic diffusion distance can be computed without
diagonalizing any of the diffusion operators. Furthermore, it is not
just the pointwise difference between the densities, but
rather the asymptotic diffusion distance is the pointwise difference plus a term that takes into
account the global difference between the two densities. It can be
used as a fast way of determing significant changes from $\alpha$
to $\beta$; see Section \ref{sec: change detection hsi data} for an example.
\end{remark} 

\subsection{Mapping one diffusion embedding into another}

As mentioned in the previous subsection, the diffusion map
$\Psi_{\alpha}^{(t)}$ takes $X$ into an $\ell^2$ space that itself
depends on $\alpha$. While \eqref{eqn: generalized diff maps formula}
gives a way of computing distances between two diffusion embeddings,
it is also possible to map the embedding $\Psi_{\beta}^{(t)}(X)$
into the $\ell^2$ space of $\Psi_{\alpha}^{(t)}(X)$. Furthermore, the operator that does so is
quite simple. The eigenfunctions $\{ \psi_{\alpha}^{(i)} \}_{i \geq
  1}$ are essentially a basis for the embedding of $X$ with parameter $\alpha$, while the eigenfunctions $\{\psi_{\beta}^{(i)} \}_{i \geq 1}$ are essentially a basis for the embedding of $X$ with parameter
$\beta$. The operator that maps one space into the other is
similar to the change of basis operator. Define $O_{\beta \rightarrow
  \alpha}: \ell^2 \rightarrow \ell^2$ as
\begin{equation*}
O_{\beta \rightarrow \alpha} v \triangleq \left( \sum_{j \geq 1} v[j]
  \, \langle \psi_{\alpha}^{(i)}, \psi_{\beta}^{(j)}
  \rangle_{L^2(X,\mu)} \right)_{i \geq 1}, \quad \text{for all } v \in \ell^2.
\end{equation*}

By the Spectral Theorem, we know that the eigenfunctions of
$A_{\alpha}$ can be taken to form an orthonormal basis for
$L^2(X,\mu)$. Thus, the operator $O_{\alpha \rightarrow \beta}$
preserves inner products. Indeed,
define the operator $S_{\alpha}: L^2(X,\mu) \rightarrow \ell^2$ as
\begin{equation*}
S_{\alpha}f \triangleq \left( \langle \psi_{\alpha}^{(i)}, f \rangle_{L^2(X,\mu)}
\right)_{i \geq 1}, \quad \text{for all } f \in L^2(X,\mu).
\end{equation*}
The adjoint of $S_{\alpha}$, $S_{\alpha}^*: \ell^2 \rightarrow
L^2(X,\mu)$, is then given by
\begin{equation*}
S_{\alpha}^*v = \sum_{i \geq 1} v[i] \, \psi_{\alpha}^{(i)}, \quad
\text{for all } v \in \ell^2.
\end{equation*}
Since $\{ \psi_{\alpha}^{(i)}\}_{i \geq 1}$ is an orthonormal basis
for $L^2(X,\mu)$, $S_{\alpha}^* S_{\alpha} =
I_{L^2(X,\mu)}$. Therefore, for any $v,w \in \ell^2$,
\begin{align}
\langle O_{\beta \rightarrow \alpha} v, O_{\beta \rightarrow \alpha}
w\rangle_{\ell^2} &= \sum_{j,k \geq 1} v[j] \, w[k] \left( \sum_{i \geq 1} \langle
  \psi_{\alpha}^{(i)}, \psi_{\beta}^{(j)} \rangle_{L^2(X,\mu)} \langle
  \psi_{\alpha}^{(i)}, \psi_{\beta}^{(k)} \rangle_{L^2(X,\mu)} \right) \nonumber \\
&= \sum_{j,k \geq 1} v[j] \, w[k] \, \langle S_{\alpha} \psi_{\beta}^{(j)}, S_{\alpha}
\psi_{\beta}^{(k)} \rangle_{\ell^2} \nonumber \\
&= \sum_{j,k \geq 1} v[j] \, w[k] \, \delta(j - k) \nonumber \\
&= \langle v, w \rangle_{\ell^2} \label{eqn: change of basis inner product}
\end{align}
As asserted, the operator $O_{\beta \rightarrow \alpha}$ preserves inner
products. In particular, it preserves norms, so we have
\begin{align*}
\big\| \Psi_{\alpha}^{(t)}(x) - O_{\beta
    \rightarrow\alpha}\Psi_{\beta}^{(t)}(y) \big\|_{\ell^2}^2 &=
\big\| \Psi_{\alpha}^{(t)}(x) \big\|_{\ell^2}^2 + \big\| O_{\beta
    \rightarrow \alpha} \Psi_{\beta}^{(t)}(y) \big\|_{\ell^2}^2 -
2\langle \Psi_{\alpha}^{(t)}(x), O_{\beta \rightarrow
  \alpha}\Psi_{\beta}^{(t)}(y) \rangle_{\ell^2} \\
&= \big\| \Psi_{\alpha}^{(t)}(x) \big\|_{\ell^2}^2 + \big\| \Psi_{\beta}^{(t)}(y)
\big\|_{\ell^2}^2 - 2 \sum_{i,j \geq 1}
\Psi_{\alpha}^{(t)}(x)[i] \, \Psi_{\beta}^{(t)}(y)[j] \, \langle \psi_{\alpha}^{(i)},
\psi_{\beta}^{(j)} \rangle_{L^2(X, \mu)} \\
&= D^{(t)}(x_{\alpha},x_{\beta}).
\end{align*}

Thus the operator $O_{\beta \rightarrow \alpha}$ maps the diffusion embedding
$\Psi_{\beta}^{(t)}(X)$ into the same $\ell^2$ space as the
diffusion embedding $\Psi_{\alpha}^{(t)}(X)$, and furthermore
preserves the diffusion distance between the two spaces; it is easy to
see that it also preserves the diffusion distance within
$\Gamma_{\beta}$. In particular, it is possible to view both embeddings in the same
$\ell^2$ space, where the $\ell^2$ distance is equal to the
diffusion distance both within each graph $\Gamma_{\alpha}$ and
$\Gamma_{\beta}$ and between the two graphs.

Suppose now that we have three or more parameters in $\I$ that are of
interest. Can we map all diffusion embeddings of these parameters into
the same $\ell^2$ space, while preserving the diffusion distances? The
answer turns out to be ``yes,'' and in fact we can use the same
mapping as before. Let $\gamma \in \I$ be the base parameter to which all
other parameters are mapped, and let $\alpha, \beta \in \I$ be two other
arbitrary parameters. We know that we can map the
embedding $\Psi_{\alpha}^{(t)}(X)$ into the $\ell^2$ space of
$\Psi_{\gamma}^{(t)}(X)$, and that we can also map the embedding
$\Psi_{\beta}^{(t)}(X)$ into the $\ell^2$ space of $\Psi_{\gamma}^{(t)}(X)$,
and that these mappings will preserve diffusion distances both within
$\Gamma_{\gamma}$, $\Gamma_{\alpha}$, and $\Gamma_{\beta}$, and also
between $\Gamma_{\gamma}$ and $\Gamma_{\alpha}$ as well as between
$\Gamma_{\gamma}$ and $\Gamma_{\beta}$. We just need to show that
they preserve the diffusion distance between points of
$\Gamma_{\alpha}$ and points of $\Gamma_{\beta}$. Using essentially
the same calculation as the one used to derive \eqref{eqn: change of
  basis inner product}, one can obtain the following for any $v,w \in \ell^2$:
\begin{equation*}
\langle O_{\alpha \rightarrow \gamma} v,
O_{\beta \rightarrow \gamma} w \rangle_{\ell^2} = \sum_{i,j \geq 1}
v[i] \, w[j] \, \langle \psi_{\alpha}^{(i)}, \psi_{\beta}^{(j)} \rangle_{L^2(X,\mu)}.
\end{equation*}
But then we have:
\begin{align*}
\big\| O_{\alpha \rightarrow \gamma}
  \vphantom{\Psi_{\beta}^{(t)}(x)} \Psi_{\alpha}^{(t)}(x)
  - O_{\beta \rightarrow \gamma}\Psi_{\beta}^{(t)}(y)
  \big\|_{\ell^2}^2 &= \big\| O_{\alpha \rightarrow
    \gamma}\Psi_{\alpha}^{(t)}(x) \big\|_{\ell^2}^2 + \big\|O_{\beta
    \rightarrow \gamma}\Psi_{\beta}^{(t)}(y)\big\|_{\ell^2}^2 -
  2\langle O_{\alpha \rightarrow \gamma}\Psi_{\alpha}^{(t)}(x), O_{\beta
  \rightarrow \gamma}\Psi_{\beta}^{(t)}(y) \rangle_{\ell^2}, \\
&= \big\| \Psi_{\alpha}^{(t)}(x) \big\|_{\ell^2}^2 + \big\| \Psi_{\beta}^{(t)}(y)
\big\|_{\ell^2}^2 - 2 \sum_{i,j \geq 1}
\Psi_{\alpha}^{(t)}(x)[i] \, \Psi_{\beta}^{(t)}(y)[j] \, \langle
\psi_{\alpha}^{(i)}, \psi_{\beta}^{(j)} \rangle_{L^2(X,\mu)} \\
&= D^{(t)}(x_{\alpha},y_{\beta}).
\end{align*}
Thus, after mapping the $\alpha$ and $\beta$ embeddings appropriately
into the $\gamma$ embedding, the $\ell^2$ distance is equal to all
possible diffusion distances. It is therefore possible to map each of the
embeddings $\{ \Psi_{\alpha}^{(t)}(X) \}_{\alpha \in \I}$ into the same
$\ell^2$ space. In particular, one can track the evolution of the
intrinsic geometry of $X$ as it changes over $\I$. We summarize this
discussion in the following theorem.

\begin{theorem} \label{thm: common embedding}
Let $(X,\mu)$ be a measure space and $\{ k_{\alpha} \}_{\alpha \in \I}$ a
family of kernels defined on $X$. Fix a parameter $\gamma \in \I$. If $(X, \mu)$ and
$\{k_{\alpha}\}_{\alpha \in \I}$ satisfy the properties of Assumption
\ref{assumption 1}, then for all $(\alpha, \beta) \in \I \times \I$,
\begin{equation*}
D^{(t)}(x_{\alpha}, y_{\beta}) = \left\| O_{\alpha \rightarrow \gamma}
\Psi_{\alpha}^{(t)}(x) - O_{\beta \rightarrow
  \gamma}\Psi_{\beta}^{(t)}(y) \right\|_{\ell^2}, \quad \text{with
convergence in } L^2(X \times X, \mu \otimes \mu).
\end{equation*}
\end{theorem}

\begin{remark} \label{rem: gamma choice}
The choice of the fixed parameter $\gamma \in \I$ is important in the
sense that the evolution of the intrinsic geometry of $X$ will be
viewed through the lens of the important features (i.e., the dominant
eigenfunctions) of $X$ at parameter $\gamma$. In particular, when
approximating the diffusion distance by a small number of dominant
eigenfunctions, one must be careful to select enough eigenfunctions at
the $\gamma$ parameter to sufficiently characterize the geometry of
the data across all of $\I$. 
\end{remark}

\subsection{Historical graph} \label{sec: historical graph}

As discussed in Remark \ref{rem: universal metric}, if one has a
universal metric $d: (X \times \I) \times (X \times \I) \rightarrow
\R$, then one can use the original diffusion maps framework to define
a single embedding for all of $X \times \I$. This embedding will be
derived from a graph on all of $X \times \I$, in which links between
any two points $x_{\alpha}$ and $y_{\beta}$ are possible. For this
reason, we think of this type of graph as a historical graph, as each
point is embedded according to its relationship with the data across
the entire parameter space (or all of time, if that is what $\I$ is). 

The diffusion distance $D^{(t)}(x_{\alpha},y_{\beta})$ defines a
measure of similarity between $x_{\alpha}$ and $y_{\beta}$ by
comparing the local neighborhoods of each point in their respective
graphs $\Gamma_{\alpha}$ and $\Gamma_{\beta}$. The comparison is, by
definition, indirect. In the case when
no universal metric exists, though, it is possible to use the
diffusion distance to create a historical graph in which every point
throughout $X \times \I$ is compared directly. 

Suppose, for example, that $\I \subset \R$ and that $\rho$ is a
measure for $\I$. Assume that $\rho (\I) < \infty$, $\mu (X) <
\infty$, $0 < C_1 \leq k_{\alpha}(x,y) \leq C_2 < \infty$ for all
$x,y \in X$, $\alpha \in \I$, and that the function $(x,y,\alpha)
\mapsto k_{\alpha}(x,y)$ is a measurable function from $(X \times X
\times \I, \mu \otimes \mu \otimes \rho)$ to $\R$. Then for each $t
\in \N$, one can define 
a kernel $\overline{k}_t: (X \times \I) \times (X \times \I)
\rightarrow \R$ as
\begin{equation*}
\overline{k}_t(x_{\alpha},y_{\beta}) \triangleq
e^{-D^{(t)}(x_{\alpha},y_{\beta})/\varepsilon}, \quad \text{for all }
(x_{\alpha},y_{\beta}) \in (X \times \I) \times (X \times \I),
\end{equation*}
where $\varepsilon > 0$ is a fixed scaling parameter. The kernel
$\overline{k}_t$ is a direct measure of similarity across $X$ and the
parameter space $\I$. Thus, when $\I$ is time, we think of $(X \times
\I, \overline{k}_t)$ as defining a historical graph in which all
points throughout history are related to one another. By our
assumptions, it is not hard to see that $0 < C_1(t) \leq
k_t(x_{\alpha},y_{\beta}) \leq C_2(t) < \infty$ for all
$x_{\alpha},y_{\beta} \in X \times \I$. Therefore we can define the
density $\overline{m}_t: X \times \I \rightarrow \R$,
\begin{equation*}
\overline{m}_t(x_{\alpha}) \triangleq \int\limits_{\I} \int\limits_X
\overline{k}_t(x_{\alpha},y_{\beta}) \, d\mu(y) \, d\rho(\beta), \quad
\text{for all } x_{\alpha} \in X \times \I,
\end{equation*}
as well as the normalized kernel $\overline{a}_t: (X \times \I) \times
(X \times \I) \rightarrow \R$,
\begin{equation*}
\overline{a}_t(x_{\alpha},y_{\beta}) \triangleq
\frac{\overline{k}_t(x_{\alpha},y_{\beta})}{\sqrt{\overline{m}_t(x_{\alpha})}
  \sqrt{\overline{m}_t(y_{\beta})}}, \quad \text{for all }
(x_{\alpha},y_{\beta}) \in (X \times \I) \times (X \times \I).
\end{equation*}
Once again using the given assumptions, one can conclude that
$\overline{a}_t \in L^2(X \times \I \times X \times \I, \mu
\otimes \rho \otimes \mu \otimes \rho)$. Thus it defines a Hilbert-Schmidt integral
operator $\overline{A}_t: L^2(X \times \I, \mu \otimes \rho)
\rightarrow L^2(X \times \I, \mu \otimes \rho)$,
\begin{equation*}
(\overline{A}_tf)(x_{\alpha}) \triangleq \int\limits_{\I}
\int\limits_X \overline{a}_t(x_{\alpha},y_{\beta}) f(y_{\beta}) \,
d\mu(y) \, d\rho(\beta), \quad \text{for all } f \in L^2(X \times \I,
\mu \otimes \rho).
\end{equation*}
Let $\{ \overline{\psi}_t^{(i)} \}_{i \geq 1}$ and $\{
\overline{\lambda}_t^{(i)} \}_{i \geq 1}$ denote the eigenfunctions
and eigenvalues of $\overline{A}_t$, respectively. The corresponding
diffusion map $\overline{\Psi}_t^{(s)}: (X \times \I) \rightarrow
\ell^2$ is given by:
\begin{equation*}
\overline{\Psi}_t^{(s)}(x_{\alpha}) \triangleq \left(
  \left(\overline{\lambda}_t^{(i)}\right)^s
  \overline{\psi}_t^{(i)}(x_{\alpha}) \right)_{i \geq 1}, \quad
\text{for all } x_{\alpha} \in X \times \I.
\end{equation*}

In the case when $\I$ is time, this diffusion map embeds the entire history of $X$ across
all of $\I$ into a single low dimensional space. Unlike the common
embedding defined by Theorem \ref{thm: common embedding}, each point
$x_{\alpha}$ is embedded in relation to the entire history of $X$, not just its
relationship to other points $y_{\alpha}$ from the same time. As such, for each $x \in
X$, one can view the trajectory of $x$ through time as it relates to
all of history, i.e., one can view:
\begin{align*}
&T_x: \I \rightarrow \ell^2 \\
&T_x(\alpha) \triangleq \overline{\Psi}_t^{(s)}(x_{\alpha}).
\end{align*}
In turn, the trajectories $\{ T_x \}_{x \in X}$ can be used to define
a measure of similarity between the data points in $X$ that takes into
account the history of each point.

\begin{remark}
It is also possible to define $\overline{k}_t$ in terms of the 
inner products of the symmetric diffusion kernels, i.e.,
\begin{equation*}
\overline{k}_t(x_{\alpha},y_{\beta}) \triangleq \int\limits_X
a_{\alpha}^{(t)}(x,u) \, a_{\beta}^{(t)}(y,u) \, d\mu(u).
\end{equation*}
\end{remark}

\begin{remark}
The diffusion distance and corresponding analysis contained in Section
\ref{sec: graph distance} can be extended to the more general case in which one has a
sequence of data sets $\{X_{\alpha}\}_{\alpha \in   \I}$ for which
there does not exist a bijective correspondence between each pair. If
there is a sufficiently large set $S$ such that $S \subset X_{\alpha}$
for each $\alpha \in \I$, then one can compute a diffusion distance
from any $x_{\alpha} \in X_{\alpha}$ to any $y_{\beta} \in X_{\beta}$ through
the common set $S$. See \ref{sec: non-bijective correspondence} for
more details.
\end{remark}

\section{Global diffusion distance} \label{sec: two
  graph distances}

Now that we have developed a diffusion distance between pairs of data points
from $(X \times \I) \times (X \times \I)$, it is possible to define a
global diffusion distance between $\Gamma_{\alpha}$ and
$\Gamma_{\beta}$. The aim here is to define a diffusion distance that gives a
global measure of the change in $X$ from $\alpha$ to 
$\beta$. In turn, when applied over the whole parameter space,
one can organize the global behavior of the data as it changes over
$\I$. For each diffusion time $t \in \N$, let $\Dt: \G \times \G
\rightarrow \R$ be this global diffusion distance, where 
\begin{align*}
\Dt(\Gamma_{\alpha}, \Gamma_{\beta})^2 & \triangleq \left\| A_{\alpha}^t -
  A_{\beta}^t \right\|_{HS}^2 \\
&= \left\| a_{\alpha}^{(t)} - a_{\beta}^{(t)} \right\|_{L^2(X \times
  X,\mu \otimes \mu)}^2 \\
&= \iint\limits_{X \times X} \left(a_{\alpha}^{(t)}(x,y) -
  a_{\beta}^{(t)}(x,y) \right)^2 \,d\mu(x) \, d\mu(y).
\end{align*}

In fact, since $\mu$ is a $\sigma$-finite measure, the global
diffusion distance can be written in terms of the pointwise
diffusion distance by applying Tonelli's Theorem:
\begin{equation*}
\Dt(\Gamma_{\alpha}, \Gamma_{\beta})^2 = \int\limits_X
D^{(t)}(x_{\alpha}, x_{\beta})^2 \, d\mu (x).
\end{equation*}
Thus the global diffusion distance measures the similarity
between $\Gamma_{\alpha}$ and $\Gamma_{\beta}$ by comparing the
behavior of each of the corresponding diffusions on each of the graphs. Therefore,
the global diffusion distance will be small if $\Gamma_{\alpha}$
and $\Gamma_{\beta}$ have similar geometry, and large if their
geometry is significantly different.

As with the pointwise diffusion distance $D^{(t)}$, the global diffusion
distance can be written in a simplified form in terms of the spectral
decompositions of the operators $A_{\alpha}$ and $A_{\beta}$. 

\begin{theorem} \label{thm: simplified global diff dist}
Let $(X,\mu)$ be a measure space and $\{k_{\alpha} \}_{\alpha \in \I}$ a
family of kernels defined on $X$. If $(X,\mu)$ and
$\{k_{\alpha}\}_{\alpha \in \I}$ satisfy the properties of Assumption
\ref{assumption 1}, then the global diffusion distance at time $t$
between $\Gamma_{\alpha}$ and $\Gamma_{\beta}$ can be written as:
\begin{equation} \label{eqn: global simple 2}
\Dt(\Gamma_{\alpha}, \Gamma_{\beta})^2 = \sum_{i,j \geq 1} \left(
  \left(\lambda_{\alpha}^{(i)}\right)^t - \left(\lambda_{\beta}^{(j)}\right)^t
\right)^2 \langle \psi_{\alpha}^{(i)}, \psi_{\beta}^{(j)} \rangle_{L^2(X, \mu)}^2.
\end{equation}
\end{theorem}

Equation \eqref{eqn: global simple 2} gives a new way to interpret the global diffusion graph
distance. The orthonormal basis  $\{ \psi_{\alpha}^{(i)} \}_{i \geq
  1}$ is a set of diffusion coordinates for $\Gamma_{\alpha}$, while
the orthonormal basis $\{\psi_{\beta}^{(j)} \}_{j \geq 1}$ is a set of
diffusion coordinates for $\Gamma_{\beta}$. Interpreting the  summands
of \eqref{eqn: global simple 2} in this context, we see that the global diffusion
distance measures the similarity of $\Gamma_{\alpha}$ and
$\Gamma_{\beta}$ by taking a weighted rotation of one coordinate
system into the other.

\begin{proof}[Proof of Theorem \ref{thm: simplified global diff dist}]
Since 
\begin{equation*}
\Dt(\Gamma_{\alpha}, \Gamma_{\beta})^2 = \int\limits_X
D_t(x_{\alpha}, x_{\beta})^2 \, d\mu(x),
\end{equation*}
we can build upon Theorem \ref{thm: simplified generalized diffusion
  distance}. In particular, we have
\begin{align*}
\Dt(\Gamma_{\alpha}, \Gamma_{\beta})^2 = \int\limits_X  &\left(\sum_{i \geq 1}
\left(\lambda_{\alpha}^{(i)}\right)^{2t} \psi_{\alpha}^{(i)}(x)^2 + \sum_{j \geq 1}
\left(\lambda_{\beta}^{(j)}\right)^{2t} \psi_{\beta}^{(j)}(x)^2 \right.\\
&\left.- 2 \sum_{i,j \geq 1} \left(\lambda_{\alpha}^{(i)}\right)^t
\left(\lambda_{\beta}^{(j)}\right)^t \psi_{\alpha}^{(i)} (x) \,
\psi_{\beta}^{(j)}(x) \, \langle \psi_{\alpha}^{(i)}, \psi_{\beta}^{(j)}
\rangle_{L^2(X, \mu)}\right) \, d\mu(x).
\end{align*}
As in the proof of Theorem \ref{thm: simplified generalized diffusion
  distance} we have three terms that we shall evaluate
separately. Focusing on the cross terms as before, we would like to
switch the integral and the summation; this time we need to show
\begin{equation} \label{eqn: suff cond to switch 2}
\sum_{i,j \geq 1} \int\limits_X \left| \left(\lambda_{\alpha}^{(i)}\right)^t
\left(\lambda_{\beta}^{(j)}\right)^t \psi_{\alpha}^{(i)} (x) \,
\psi_{\beta}^{(j)} (x) \, \langle \psi_{\alpha}^{(i)}, \psi_{\beta}^{(j)}
\rangle_{L^2(X, \mu)}\right| \, d\mu (x) < \infty.
\end{equation}
One can show \eqref{eqn: suff cond to switch 2} by using H\"{o}lder's
Theorem, the Cauchy-Schwarz inquality, and the assumption that
$A_{\alpha}$ is a trace class operator for each $\alpha \in \I$. Therefore we can
switch the integral and the summation, which gives:
\begin{equation}
\int\limits_X \sum_{i,j \geq 1} \left(\lambda_{\alpha}^{(i)}\right)^t
\left(\lambda_{\beta}^{(j)}\right)^t \psi_{\alpha}^{(i)}(x) \,
\psi_{\beta}^{(j)}(x) \, \langle \psi_{\alpha}^{(i)}, \psi_{\beta}^{(j)}
\rangle_{L^2(X, \mu)}\, d\mu (x) = \sum_{i,j \geq 1} \left(\lambda_{\alpha}^{(i)}\right)^t
\left(\lambda_{\beta}^{(j)}\right)^t \langle
\psi_{\alpha}^{(i)}, \psi_{\beta}^{(j)} \rangle_{L^2(X,
  \mu)}^2. \label{eqn: simplified cross terms} 
\end{equation}
A similar calculation also shows that for each $\alpha \in \I$,
\begin{equation}
\int\limits_X \sum_{i \geq 1} \left(\lambda_{\alpha}^{(i)}\right)^{2t}
\psi_{\alpha}^{(i)}(x)^2 \, d\mu(x) = \sum_{i \geq 1}
\left(\lambda_{\alpha}^{(i)}\right)^{2t}. \label{eqn: simplified L2
  qt}
\end{equation}
Putting \eqref{eqn: simplified cross terms} and \eqref{eqn: simplified
  L2 qt} together, we arrive at:
\begin{equation} \label{eqn: global diffusion final formula}
\Dt(\Gamma_{\alpha}, \Gamma_{\beta})^2 = \sum_{i \geq 1} \left(\lambda_{\alpha}^{(i)}\right)^{2t} + \sum_{j
  \geq 1} \left(\lambda_{\beta}^{(j)}\right)^{2t} - 2\sum_{i,j \geq 1} \left(\lambda_{\alpha}^{(i)}\right)^t \left(\lambda_{\beta}^{(j)}\right)^t \langle
\psi_{\alpha}^{(i)}, \psi_{\beta}^{(j)} \rangle_{L^2(X, \mu)}^2.
\end{equation}
Furthermore, recall that we have taken
$\{\psi_{\alpha}^{(i)}\}_{i \geq 1}$ and $\{ \psi_{\beta}^{(j)} \}_{j
  \geq 1}$ to be orthonormal bases for $L^2(X, \mu)$. In particular,
\begin{equation*}
\sum_{i \geq 1} \langle \psi_{\alpha}^{(i)}, \psi_{\beta}^{(j_0)} \rangle^2 = \sum_{j \geq
  1} \langle \psi_{\alpha}^{(i_0)}, \psi_{\beta}^{(j)} \rangle^2 = 1, \quad \text{for all
} i_0, j_0 \geq 1.
\end{equation*}
Therefore we can simplify \eqref{eqn: global diffusion final formula} to
\begin{equation*}
\Dt(\Gamma_{\alpha}, \Gamma_{\beta})^2 = \sum_{i,j \geq 1} \left(
  \left(\lambda_{\alpha}^{(i)}\right)^t -
  \left(\lambda_{\beta}^{(j)}\right)^t \right)^2 \langle
\psi_{\alpha}^{(i)}, \psi_{\beta}^{(j)} \rangle_{L^2(X, \mu)}^2.  
\end{equation*}
\end{proof}

\begin{remark}
As with the pointwise diffusion distance, the asymptotic behavior of
the global diffusion distance when $\G$ is a family of connected graphs is both interesting and easy to
characterize. Under the same connectivity assumptions as Remark
\ref{rem: asymptotic diff dist}, one can use \eqref{eqn: global
  diffusion final formula} to show that
\begin{equation*}
\lim_{t \rightarrow \infty} \Dt(\Gamma_{\alpha},\Gamma_{\beta})^2 =
2\left( 1-\left\langle \psi_{\alpha}^{(1)}, \psi_{\beta}^{(1)}
  \right\rangle_{L^2(X,\mu)}^2 \right).
\end{equation*}
\end{remark}

\section{Random sampling theorems} \label{sec: convergence
  of finite approximations}

In applications, the given data is finite and often times sampled from
some continuous data set $X$. In this section we examine the behavior
of the pointwise and global diffusion distances when applied to a
randomly sampled, finite collection of samples taken from $X$. 

\subsection{Updated assumptions}

In order to frame this discussion in the appropriate
setting, we update our assumptions on the measure space $(X, 
\mu)$ and the kernels $\{ k_{\alpha} \}_{\alpha \in \I}$. The results
from this section will rely heavily upon the work contained in
\cite{devito:learningExamples2005, rosasco:learningIntegralOps2010},
and so we follow their lead. First, for any $l \in \N$, let
$C_b^l(X)$ denote the set of continuous bounded functions on $X$ such
that all derivatives of order $l$ exist and are themselves continuous,
bounded functions.

\begin{assumption} \label{assumption 1.5}
We assume the following properties:

\begin{enumerate}
\item
The measure $\mu$ is a probability measure, so that $\mu(X) = 1$.
\item
$X$ is a bounded open subset of $\R^d$ that satisfies the cone
condition (see page $93$ of \cite{burenkov:sobolevSpacesDomains1998}).  
\item
For each $\alpha \in \I$, the kernel $k_{\alpha}$ is symmetric,
positive definite, and bounded from above and below, so that
\begin{equation*}
0 < C_1(\alpha) \leq k_{\alpha}(x,y) \leq C_2(\alpha) < \infty.
\end{equation*}
\item
For each $\alpha \in \I$, $k_{\alpha} \in C_b^{d+1}(X \times X)$.
\end{enumerate}
\end{assumption}

Note that every property from Assumption \ref{assumption 1} is either
contained in or can be derived from the properties in Assumption
\ref{assumption 1.5}. Therefore the results of the previous sections
still apply under these new assumptions.

The first assumption that $\mu$ be a probability measure is needed
since we will be randomly sampling points from $X$. The probability
measure from which we sample is $\mu$. The second and fourth
assumptions are necessary to apply certain Sobolev embedding theorems
which are integral to constructing a reproducing kernel Hilbert space
that contains the family of kernels $\{a_{\alpha}\}_{\alpha \in \I}$
and their empirical equivalents. More details can be found in
\ref{sec: proof of finite graph approximation estimate}.

\subsection{Sampling and finite graphs} \label{sec: sampling and
  finite graphs}

Consider the space $X$ and suppose that $X_n \triangleq \{x^{(1)}, \ldots, x^{(n)}\}
\subset X$ are sampled i.i.d. according to $\mu$. We are going to discretize the
framework we have developed to accommodate the samples $X_n$. Let
$\Gamma_{\alpha, n} \triangleq (X_n, k_{\alpha}|_{X_n})$ be the finite graph
with vertices $X_n$ and weighted edges given by
$k_{\alpha}|_{X_n}$. We now define the finite, matrix equivalents to
the continuous operators from Section \ref{sec: defining the diffusion
  distance on Xtau}. To start, first define for
each $\alpha \in \I$ the $n \times n$ matrices $\K_{\alpha}$ as:
\begin{equation*}
\K_{\alpha}[i,j] \triangleq \frac{1}{n} \, k_{\alpha}(x^{(i)}, x^{(j)}),
\quad \text{for all } i,j = 1, \ldots, n.
\end{equation*}
We also define the corresponding diagonal degree matrices $\D_{\alpha}$
as:
\begin{equation*}
\D_{\alpha}[i,i] \triangleq \frac{1}{n}\sum_{j=1}^n
k_{\alpha}(x^{(i)},x^{(j)}) = \sum_{j=1}^n \K_{\alpha}[i,j], \quad
\text{for all } i = 1, \ldots, n.
\end{equation*}
Finally, the discrete analog of the operator $A_{\alpha}$ is given by
the matrix $\A_{\alpha}$, which is defined as
\begin{equation*}
\A_{\alpha} \triangleq \D_{\alpha}^{-\frac{1}{2}} \K_{\alpha}
\D_{\alpha}^{-\frac{1}{2}}, \quad \text{for all } \alpha \in \I.
\end{equation*}

We can now define the pointwise and global diffusion distances for the
finite graphs $\G_n \triangleq \{\Gamma_{\alpha,n}\}_{\alpha \in \I}$ in terms of the
matrices $\{\A_{\alpha}\}_{\alpha \in \I}$. Set $x_{\alpha}^{(i)}
\triangleq (x^{(i)}, \alpha) \in X_n \times \I$, and let $D_n^{(t)}: (X_n
\times \I) \times (X_n \times \I) \rightarrow \R$ denote the empirical version of the
pointwise diffusion distance. We define it as:
\begin{align*}
D_n^{(t)}(x_{\alpha}^{(i)},x_{\beta}^{(j)})^2 &\triangleq
n^2\left\|\A_{\alpha}^t[i,\cdot] -
  \A_{\beta}^t[j,\cdot]\right\|_{\R^n}^2 \\
&= n^2 \sum_{k=1}^n \left(\A_{\alpha}^t[i,k] - \A_{\beta}^t[j,k]\right)^2.
\end{align*}
Let $\Dt_n: \G_n \times \G_n \rightarrow \R$ denote the empirical global
diffusion distance, where
\begin{align*}
\Dt_n(\Gamma_{\alpha,n}, \Gamma_{\beta,n})^2 &\triangleq
\left\|\A_{\alpha}^t - \A_{\beta}^t\right\|_{HS} \\
&= \sum_{i,j=1}^n \left( \A_{\alpha}^t[i,j] - \A_{\beta}^t[i,j] \right)^2.
\end{align*}
We then have the following two theorems relating $D_n^{(t)}$ to $D^{(t)}$
and $\Dt_n$ to $\Dt$, respectively.

\begin{theorem} \label{thm: ptwise diff dist sample}
Suppose that $(X,\mu)$ and $\{k_{\alpha}\}_{\alpha \in \I}$ satisfy the
conditions of Assumption \ref{assumption 1.5}. Let $n \in \N$ and
sample $X_n = \{x^{(1)}, \ldots, x^{(n)}\} \subset X$ i.i.d. according to
$\mu$; also let $t \in \N$, $\tau > 0$, and $\alpha,\beta \in \I$. Then, with
probability $1-2e^{-\tau}$, 
\begin{equation*}
\left| D^{(t)}(x_{\alpha}^{(i)},x_{\beta}^{(j)}) -
  D_n^{(t)}(x_{\alpha}^{(i)},x_{\beta}^{(j)}) \right| \leq C(\alpha,\beta,d,t)
\frac{\sqrt{\tau}}{\sqrt{n}}, \quad \text{for all } i,j = 1, \ldots, n.
\end{equation*}
\end{theorem}

\begin{theorem} \label{thm: global diff dist sample}
Suppose that $(X,\mu)$ and $\{k_{\alpha}\}_{\alpha \in \I}$ satisfy the
conditions of Assumption \ref{assumption 1.5}. Let $n \in \N$ and
sample $X_n = \{x^{(1)}, \ldots, x^{(n)}\} \subset X$ i.i.d. according to
$\mu$; also let $t \in \N$, $\tau > 0$, and $\alpha,\beta \in \I$. Then, with
probability $1-2e^{-\tau}$, 
\begin{equation*}
\left| \Dt(\Gamma_{\alpha},\Gamma_{\beta}) -
  \Dt_n(\Gamma_{\alpha,n},\Gamma_{\beta,n}) \right| \leq
C(\alpha,\beta,d,t) \frac{\sqrt{\tau}}{\sqrt{n}}.
\end{equation*}
\end{theorem}

\section{Applications} \label{sec: applications}

\subsection{Change detection in hyperspectral imagery
  data} \label{sec: change detection hsi data}

In this section we consider the problem of change detection in
hyperspectral imagery (HSI) data. Additionally, we use this particular
experiment to illustrate two important properties of the diffusion
distance. First, the representation of the data does not matter, even
if it is changing across the parameter space. Secondly, the diffusion
distance is robust to noise.

The main ideas are the following. A
hyperspectral image can be thought of as a data cube $\C$,
with dimensions $L \times W \times D$. The cube $\C$ corresponds to an
image whose pixel dimensions are $L \times W$. A hyperspectral camera
measures the reflectance of this image at $D$ different
wavelengths, giving one $D$ images, which, put together, give one the
cube $\C$. Thus we think of a hyperspectral image as a regular image,
but each pixel now has a spectral signature in $\R^D$.

The change detection problem is the following. Suppose you have one
scene for which you have several hyperspectral images taken at
different times. These images can be taken under different weather
conditions, lighting conditions, during different seasons of the year,
and even with different cameras. The goal is to determine what has
changed from one image to the next.

To test the diffusion distance in this setting, we used some of the data collected
in \cite{eismann:hsiChangeDetection2008}. Using a hyperspectral camera
that captured $124$ different wavelengths, the authors of
\cite{eismann:hsiChangeDetection2008} collected hyperspectral images
of a particular scene during August, September, October, and November
(one image for each month). In October, they also recorded a fifth
image in which they added two small tarp bundles so as to introduce small
changes into the scene as a means for testing change detection
algorithms. For our purposes, we selected a particular $100 \times 100 \times
124$ sub-cube across all five images that contains one of the
aforementioned introduced changes. Color images of the four months
plus the additional fifth image containing the tarp are
given in Figure \ref{fig: color images of four months}. In all five images one can
see in the foreground grass and in the background a tree line, with a
metal panel resting on the grass. In the additional fifth image, there is also a
small tarp sitting on the grass. The images were obviously taken
during different times of the year, ranging from Summer to Fall, and
it is also evident that the lighting is different from image to
image. One can see these changes in how the spectral signature of a
particular pixel changes from month to month; see Figure \ref{fig:
  original spectra} for an example of a grass pixel.

\begin{figure}
\center
\subfigure[August]{
\includegraphics[width=1.2in]{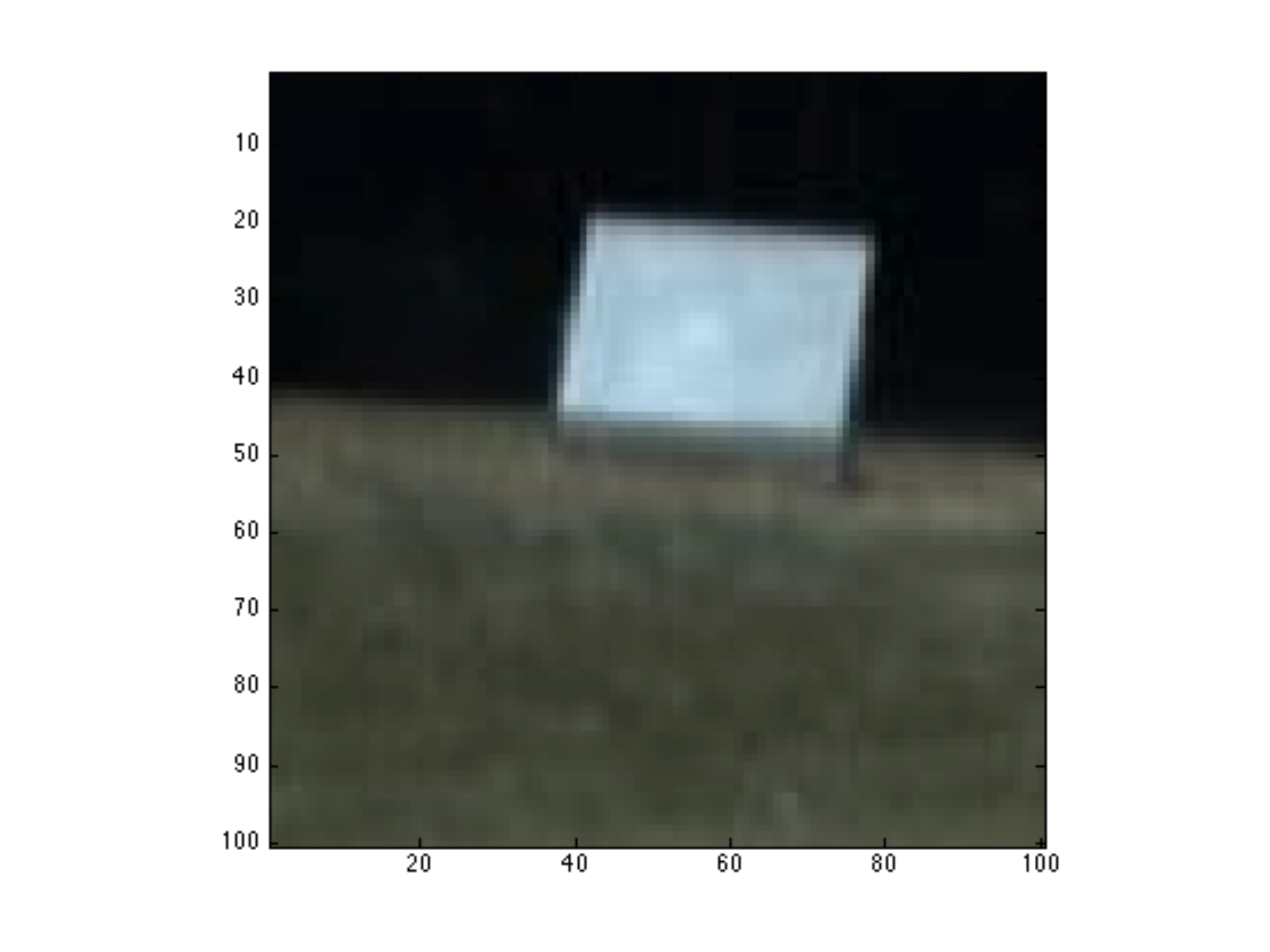}
}
\subfigure[September]{
\includegraphics[width=1.2in]{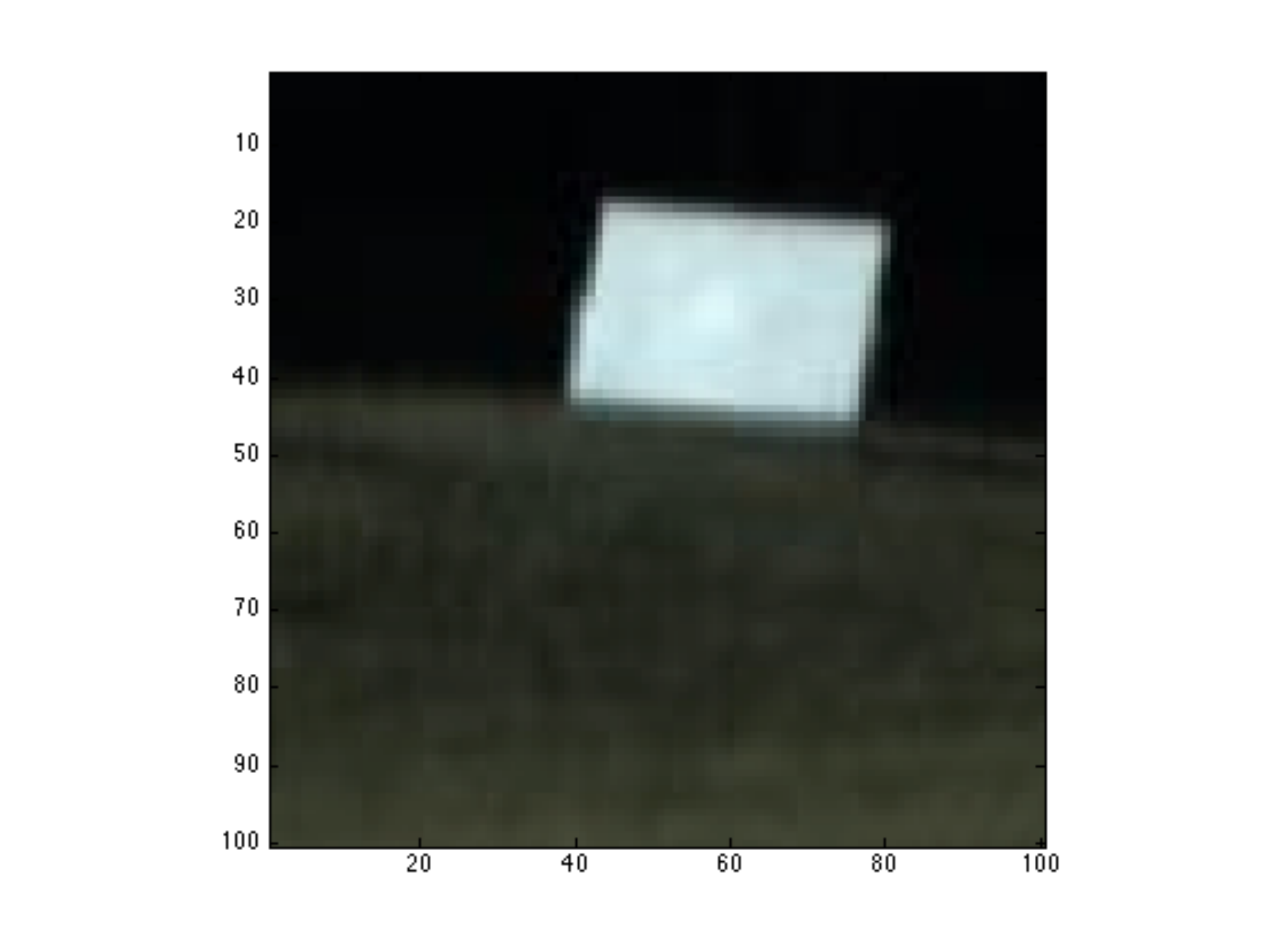}
}
\subfigure[October]{
\includegraphics[width=1.2in]{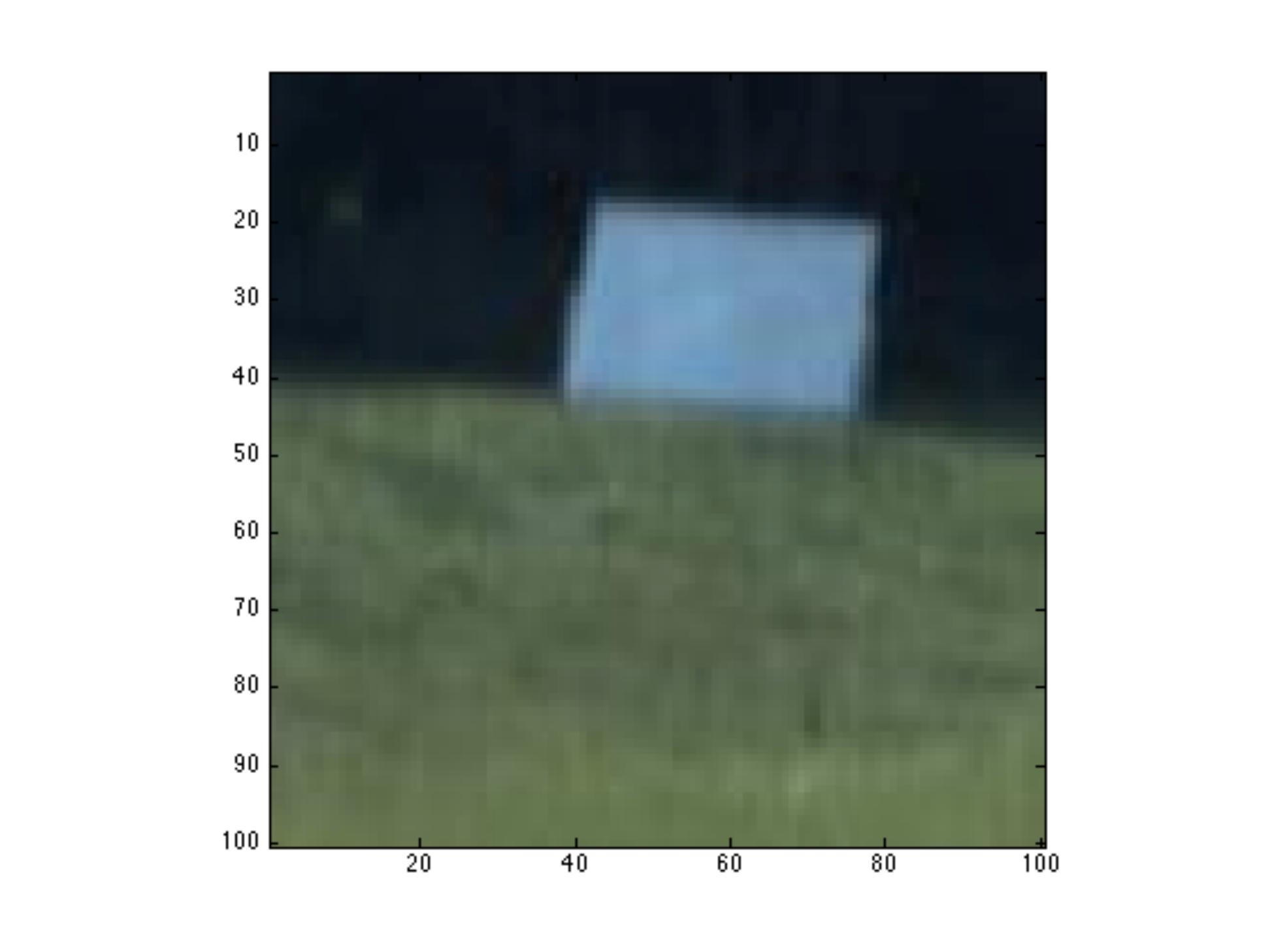}
}
\subfigure[November]{
\includegraphics[width=1.2in]{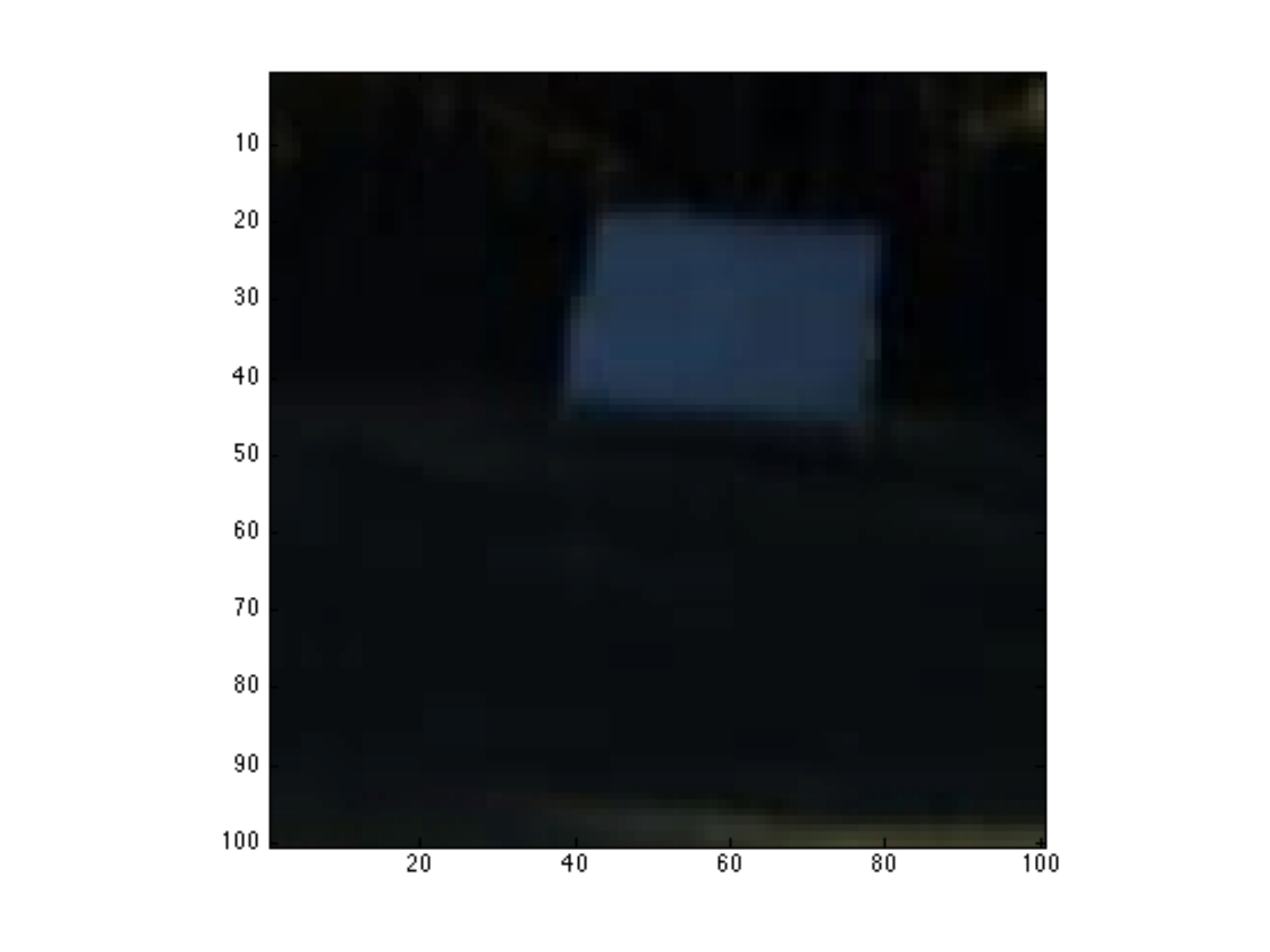}
}
\subfigure[October with tarp]{
\includegraphics[width=1.2in]{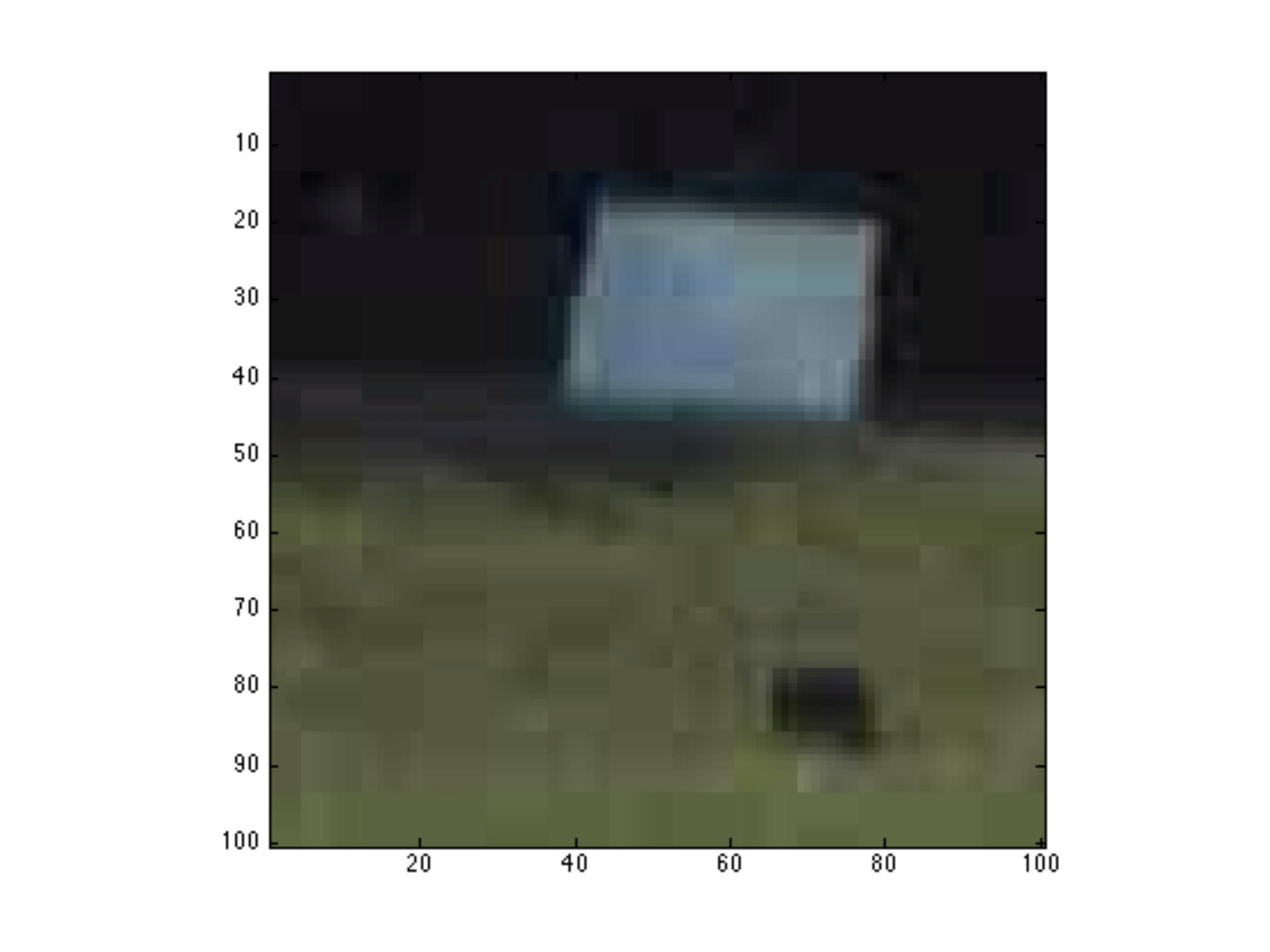}
}
\caption{Color images of the four months.}
\label{fig: color images of four months}
\end{figure}

\begin{figure}
\center
\subfigure[Original camera spectra]{
\includegraphics[width=2.5in]{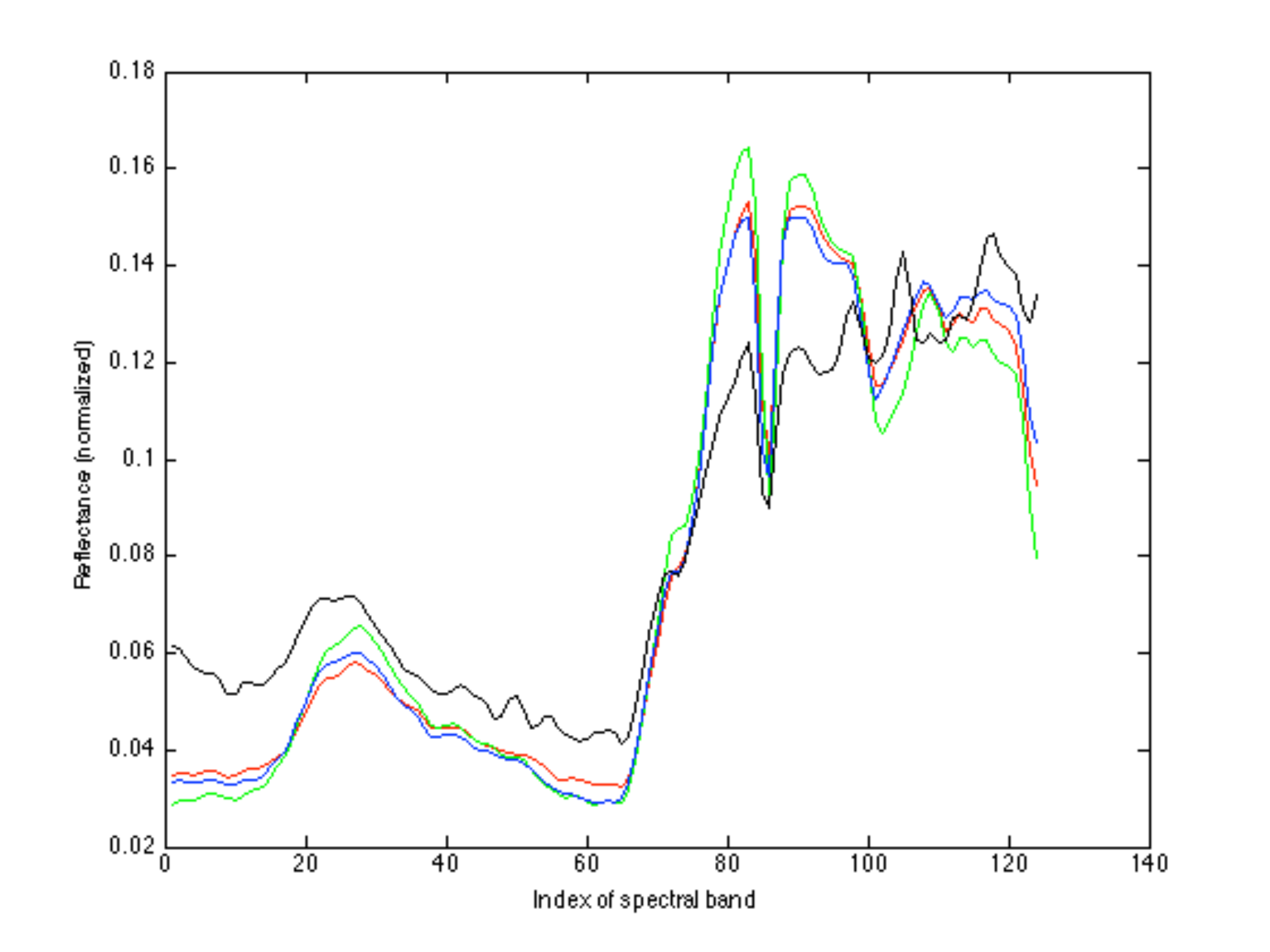}
\label{fig: original spectra}
}
\subfigure[Random camera spectra]{
\includegraphics[width=2.5in]{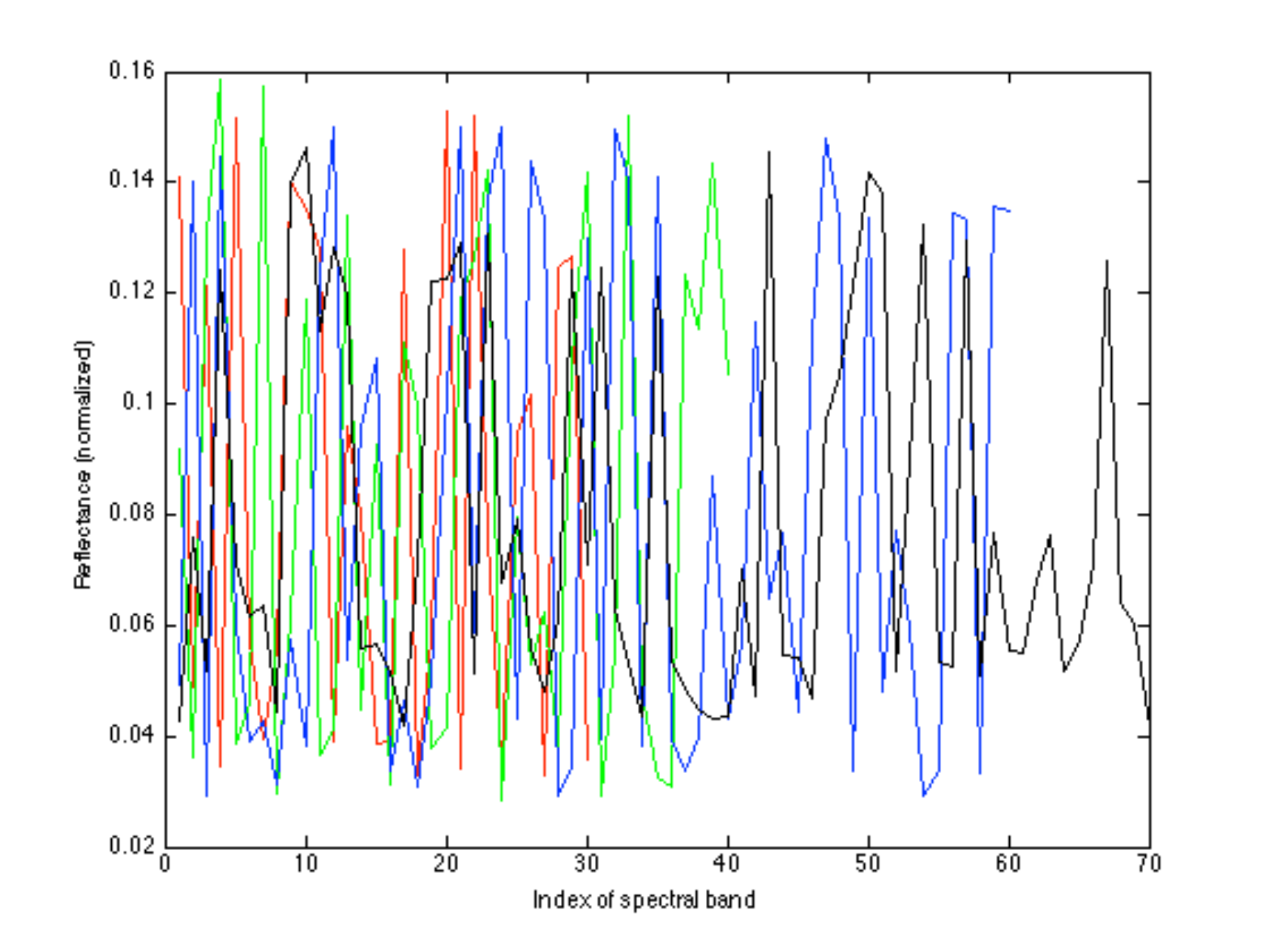}
\label{fig: random spectra}
}
\caption{Spectrum of a single grass pixel across the four months. Red:
  August, green: September, blue: October, black: November.}
\end{figure}

We set the parameter space as $\I = \{ \text{aug, sep, oct, nov, chg}
\}$, where $\text{chg}$ denotes the October data set with the tarp in
it. We also set $\I^{(4)} \triangleq \{ \text{aug, sep, oct, nov} \}
\subset \I$. For each $\alpha \in \I$, we let $X_{\alpha}$ denote the
corresponding $100 \times 100 \times 124$ hyperspectral image. The
data points $x \in X_{\alpha}$ are the spectral signatures of
each pixel; that is, $|X_{\alpha}| = 10000$ and $x \in
\R^{124}$ for each $\alpha \in \I$. For each month as well as the
changed data set, we computed a Gaussian kernel of the form: 
\begin{equation*}
k_{\alpha}(x,y) = e^{-\|x-y\|^2/\varepsilon(\alpha)^2}, \quad
\text{for all } \alpha \in \I, \enspace x,y \in X_{\alpha},
\end{equation*}
where $\|\cdot\|$ is the Euclidean distance and $\varepsilon(\alpha)$
was selected so that the corresponding
symmetric diffusion operator (matrix) $A_{\alpha}$ would have second
eigenvalue $\lambda_{\alpha}^{(2)} \approx 0.97$. By forcing each
diffusion operator to have approximately the same second eigenvalue,
the five diffusion processes will spread at approximately the same
rate. We kept the top $20$ eigenvectors and eigenvalues and computed
the diffusion distance between a pixel $x$
taken from $X_{\text{chg}}$ and its corresponding pixel in
$X_{\alpha}$ for each $\alpha \in \I^{(4)}$, i.e., we computed
$D^{(t)}(x_{\text{chg}}, x_{\alpha})$. The results for $t = 1$ are
given in Figures \ref{fig: t1 aug orig}, \ref{fig: t1 sep orig},
\ref{fig: t1 oct orig}, \ref{fig: t1 nov orig}, while the
asymptotic diffusion distance as $t \rightarrow \infty$ is given in
Figures \ref{fig: tinf aug orig}, \ref{fig: tinf sep orig}, \ref{fig:
  tinf oct orig}, \ref{fig: tinf nov orig}. We also computed the
global diffusion distances between the changed data set and the four
months. The results are given in Figure \ref{fig: global diff dist orig}. Note that the diffusion
distance at diffusion time $t=1$ was computed via Theorem \ref{thm:
  common embedding}, the asymptotic diffusion distance was
computed using \eqref{eqn: asymptotic diff dist} from Remark \ref{rem:
  asymptotic diff dist}, and the global diffusion distance was
computed using Theorem \ref{thm: simplified global diff dist}.

\begin{figure}
\center
\subfigure[August (original)]{
\includegraphics[width=1.2in]{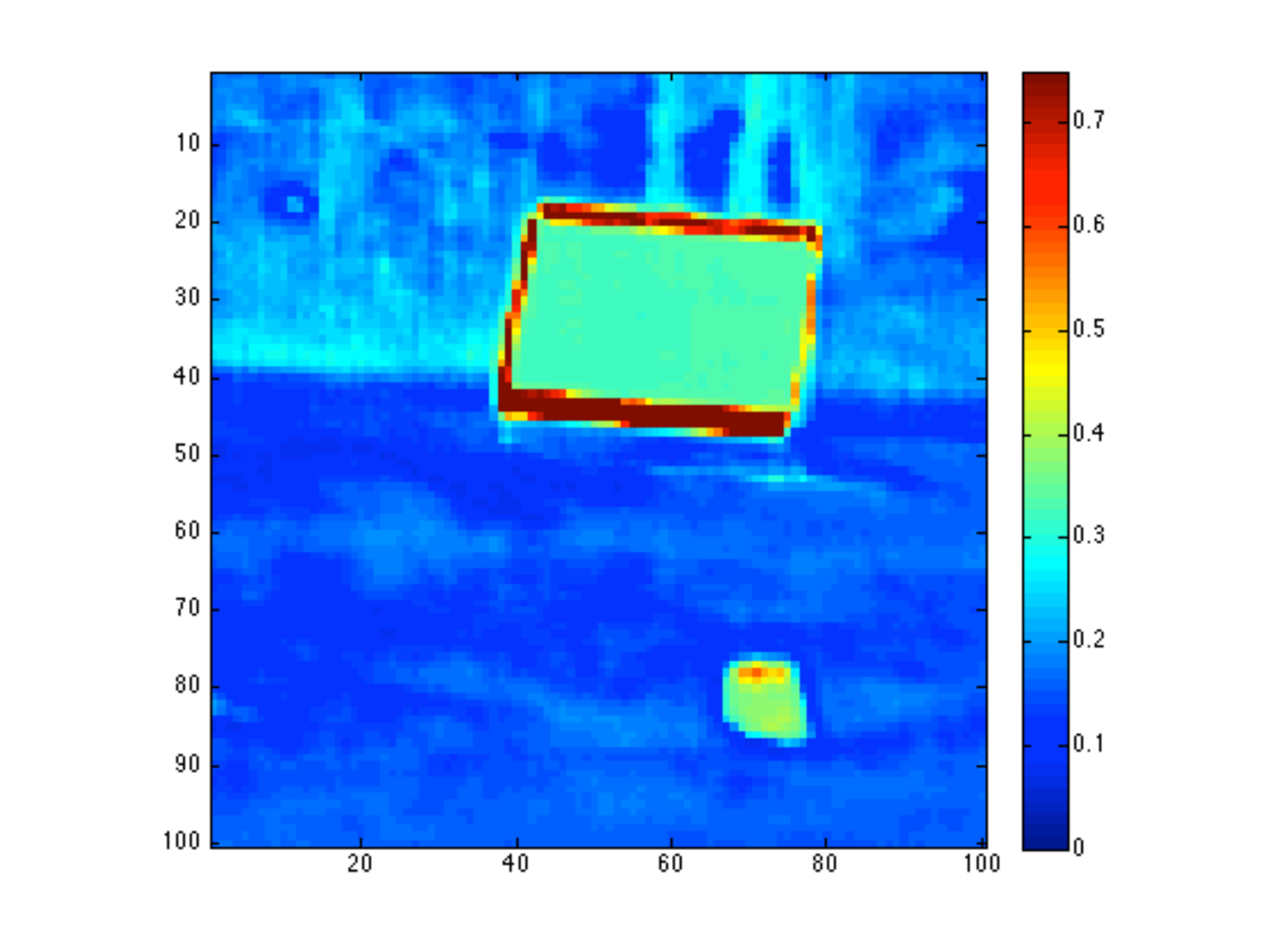}
\label{fig: t1 aug orig}
}
\subfigure[September (original)]{
\includegraphics[width=1.2in]{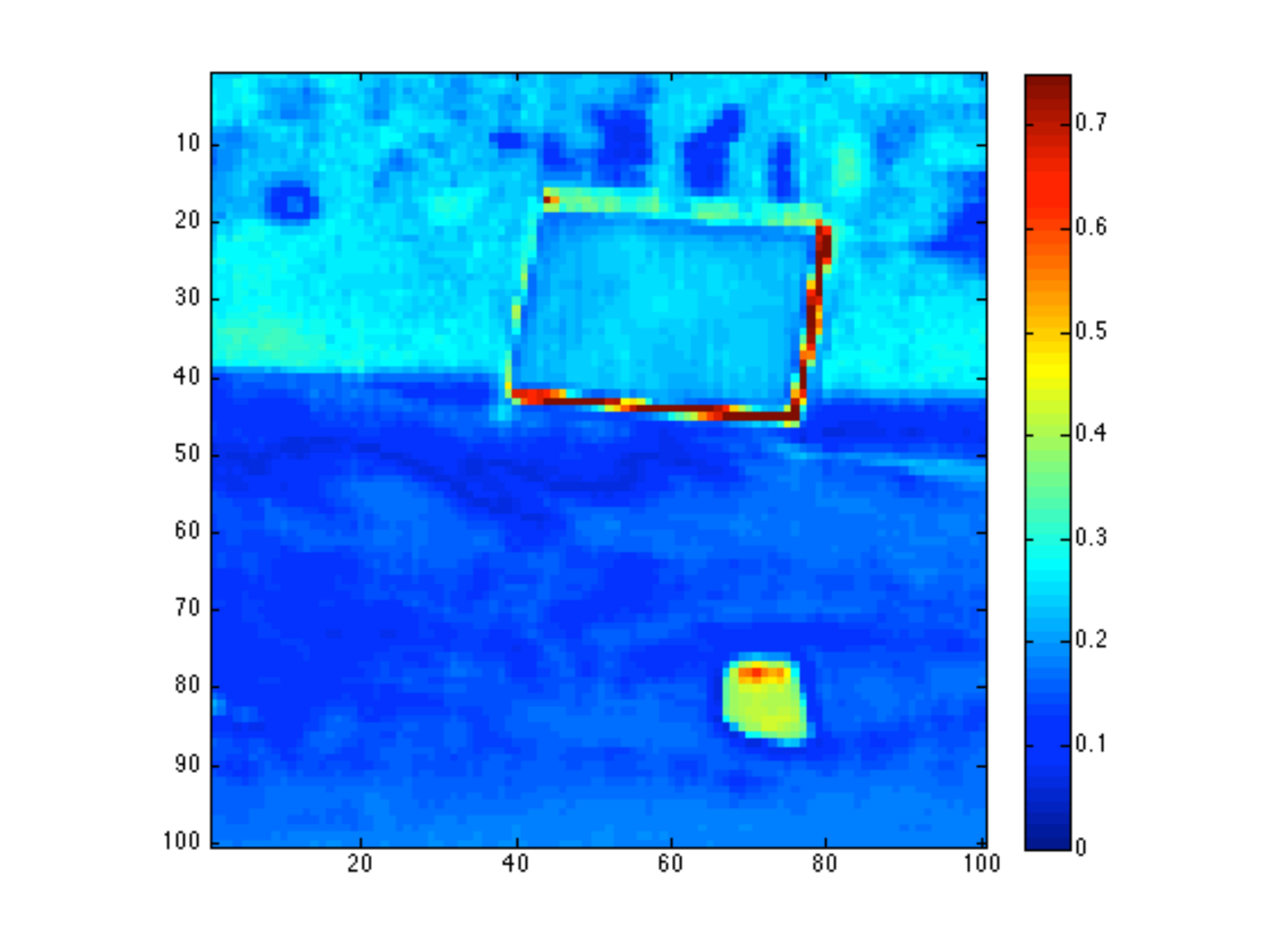}
\label{fig: t1 sep orig}
}
\subfigure[October (original)]{
\includegraphics[width=1.2in]{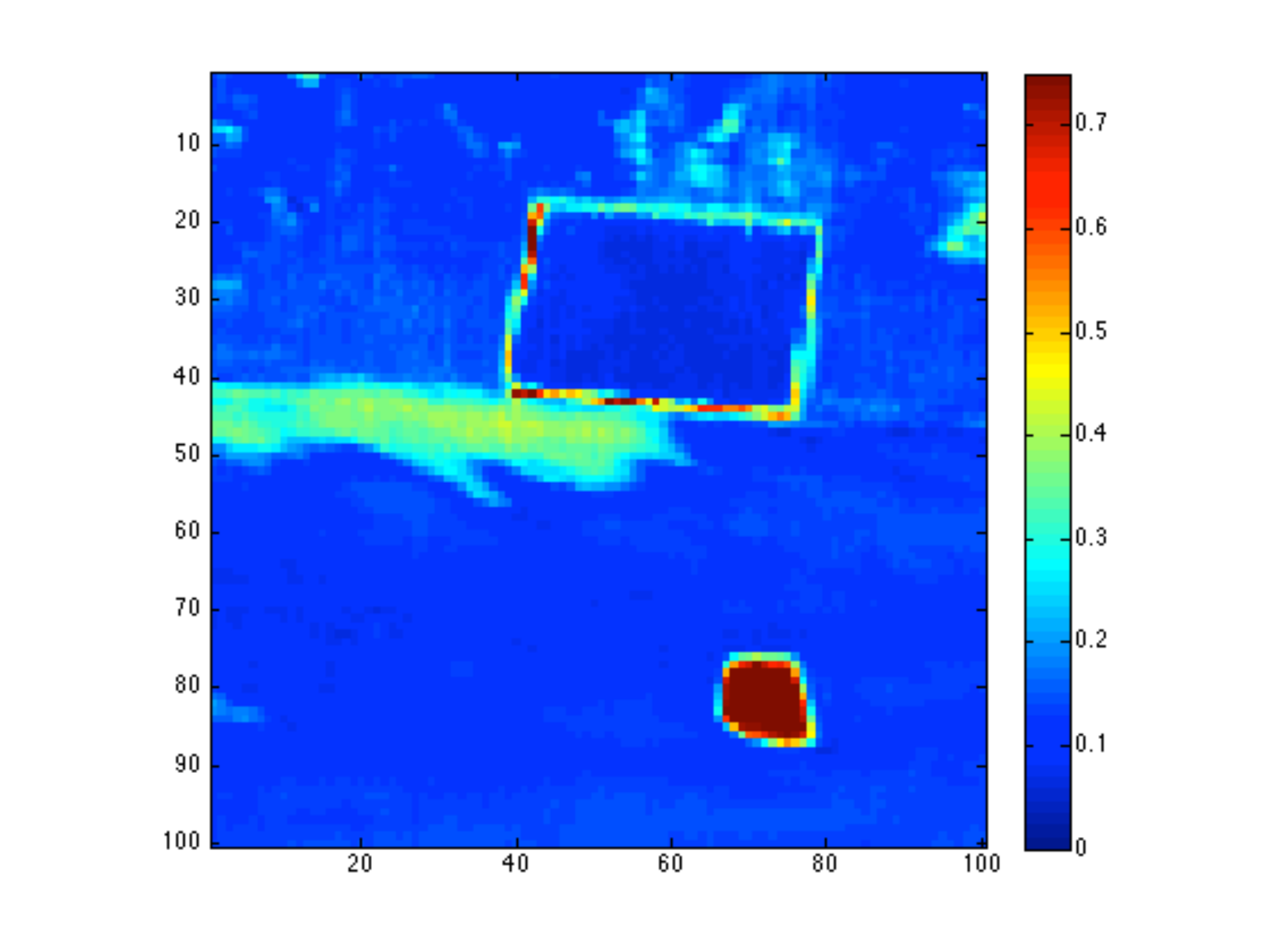}
\label{fig: t1 oct orig}
}
\subfigure[November (original)]{
\includegraphics[width=1.2in]{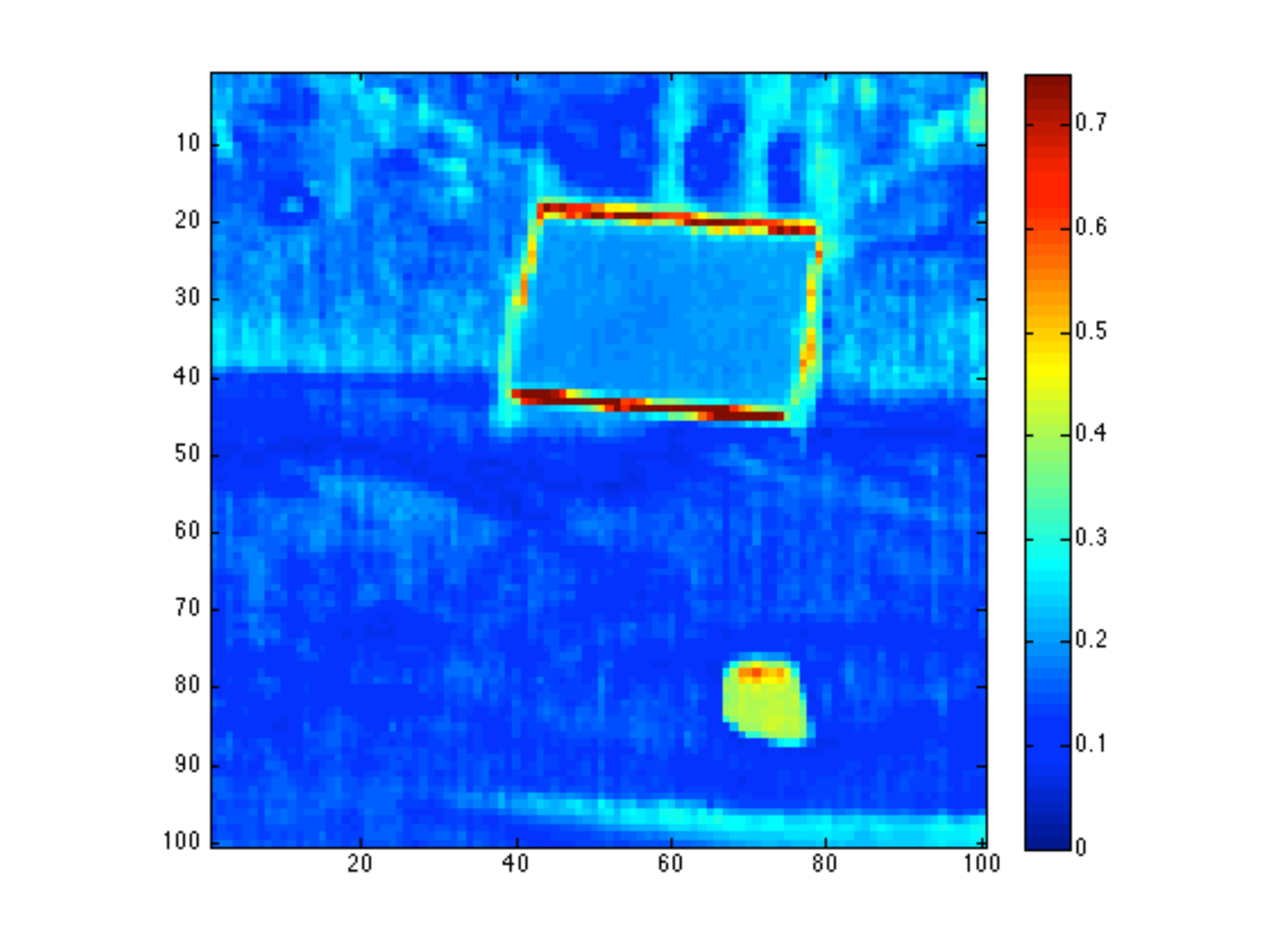}
\label{fig: t1 nov orig}
}
\subfigure[August (random)]{
\includegraphics[width=1.2in]{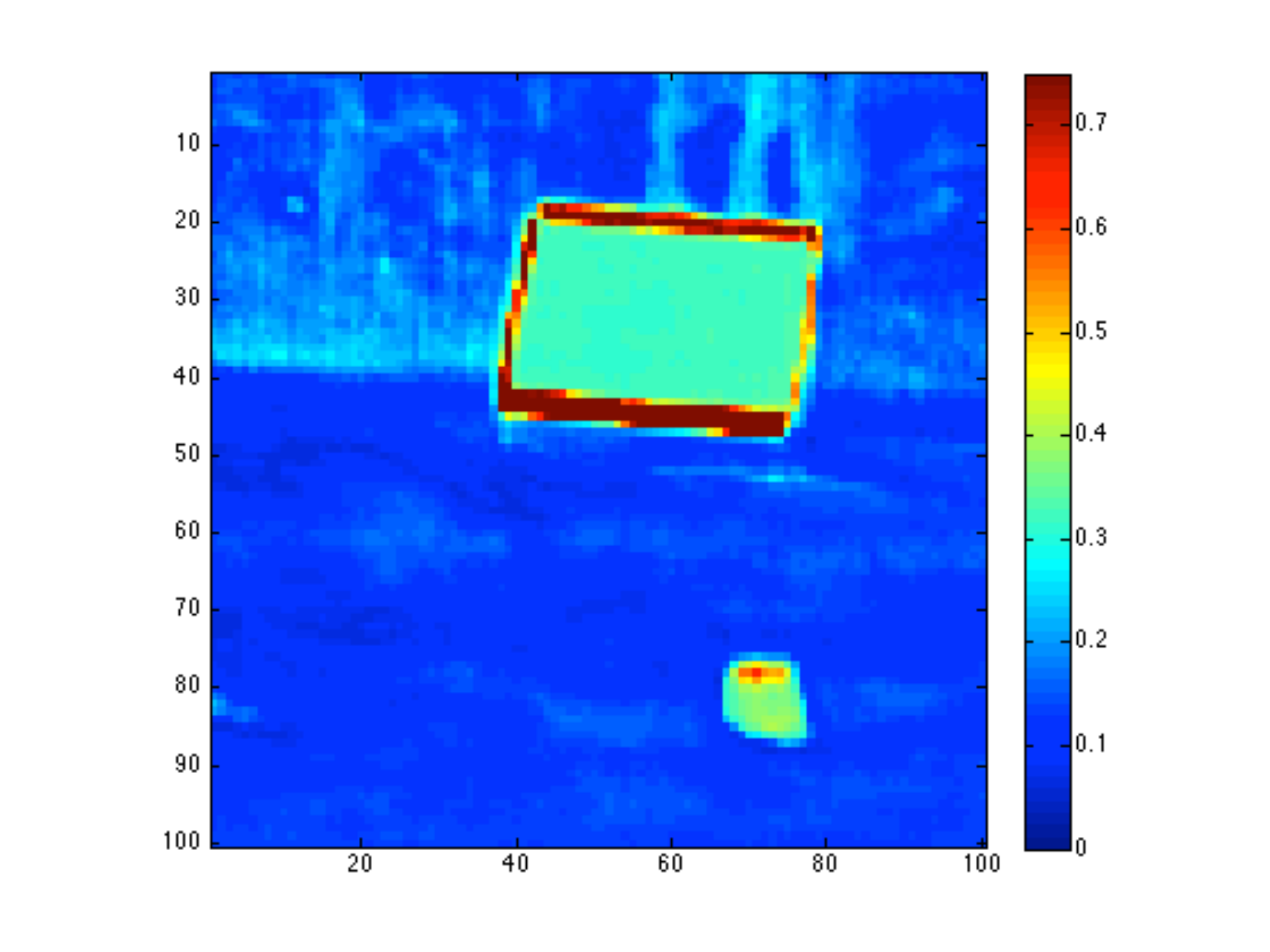}
\label{fig: t1 aug rand}
}
\subfigure[September (random)]{
\includegraphics[width=1.2in]{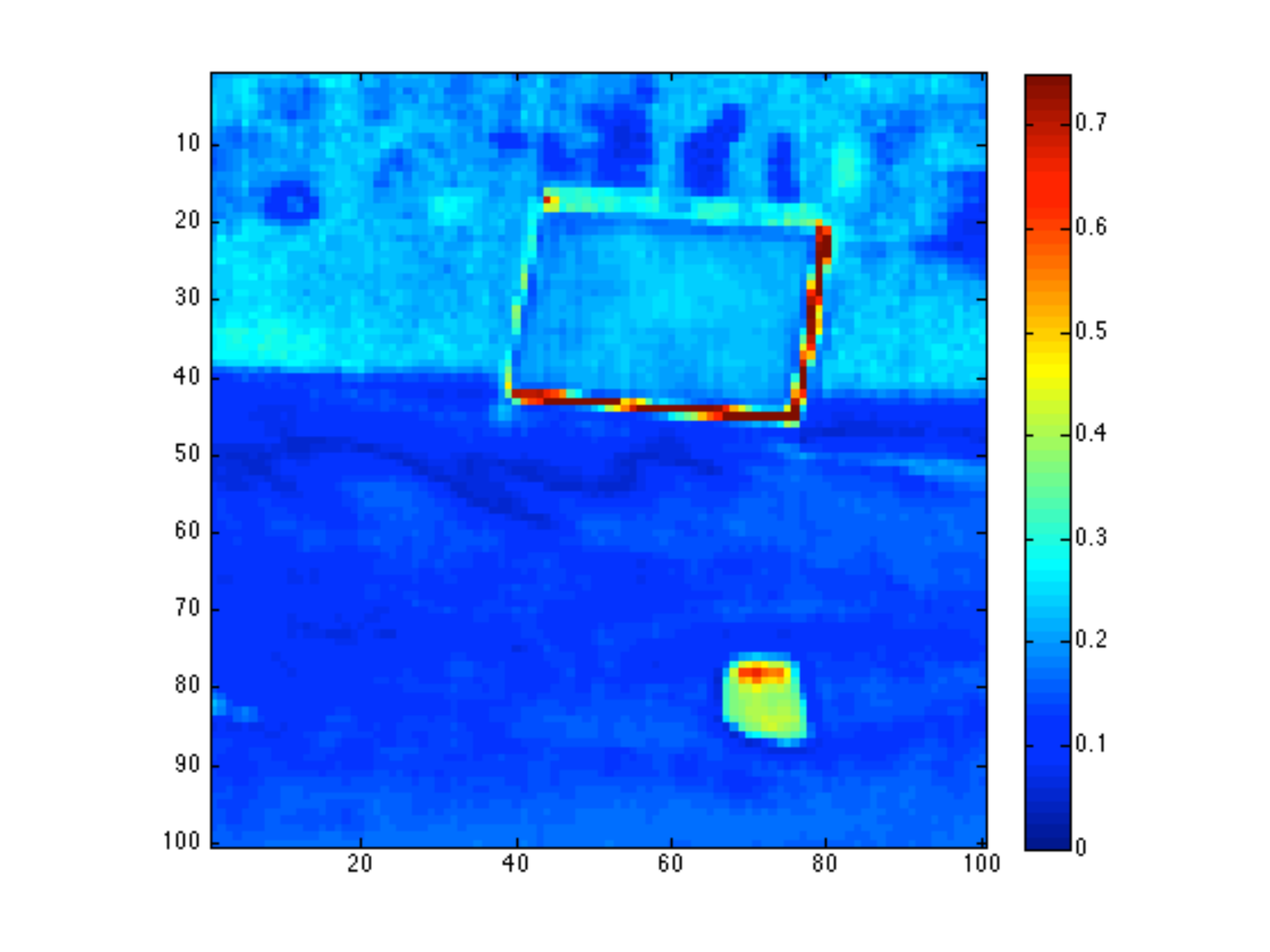}
\label{fig: t1 sep rand}
}
\subfigure[October (random)]{
\includegraphics[width=1.2in]{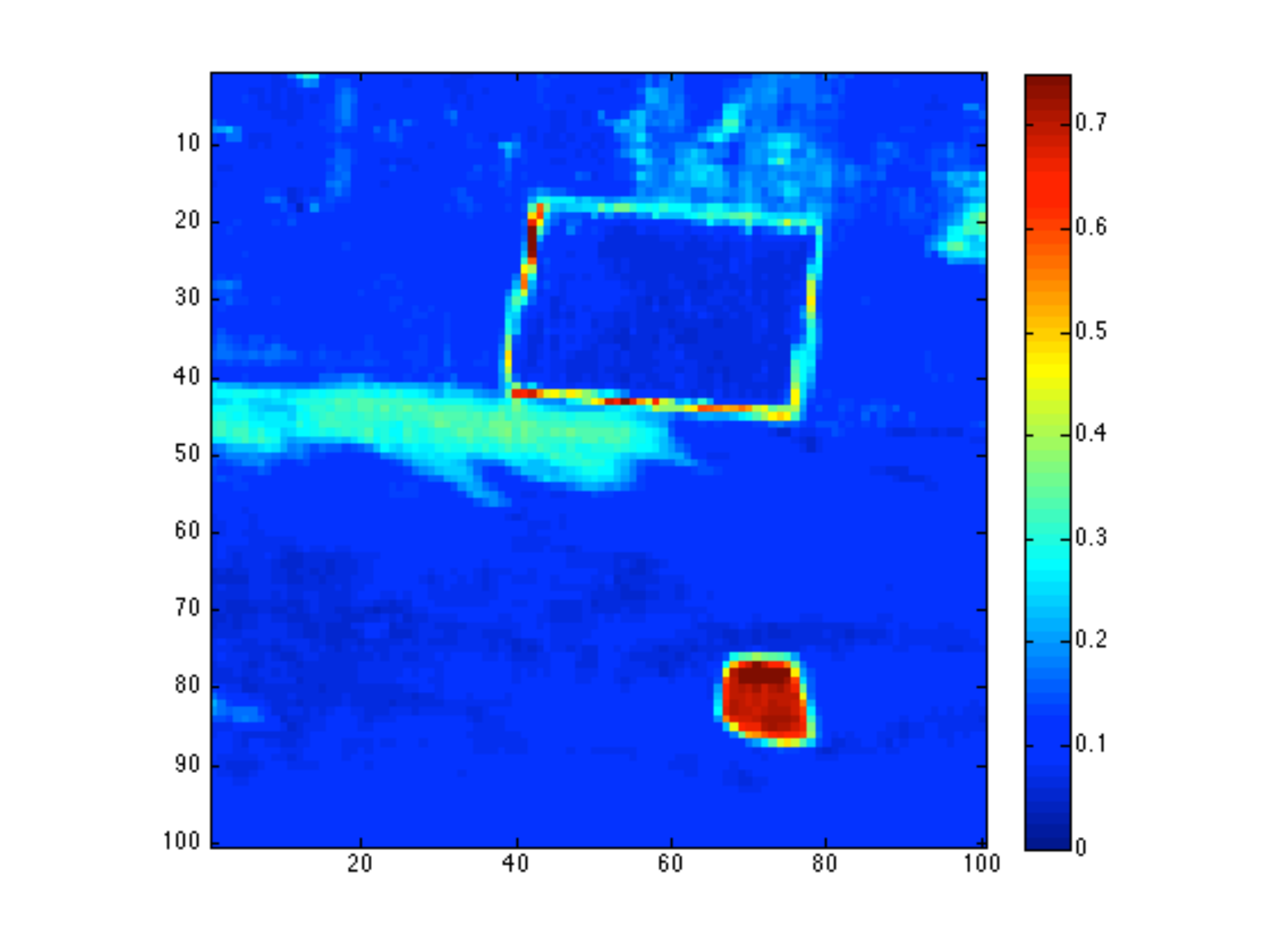}
\label{fig: t1 oct rand}
}
\subfigure[November (random)]{
\includegraphics[width=1.2in]{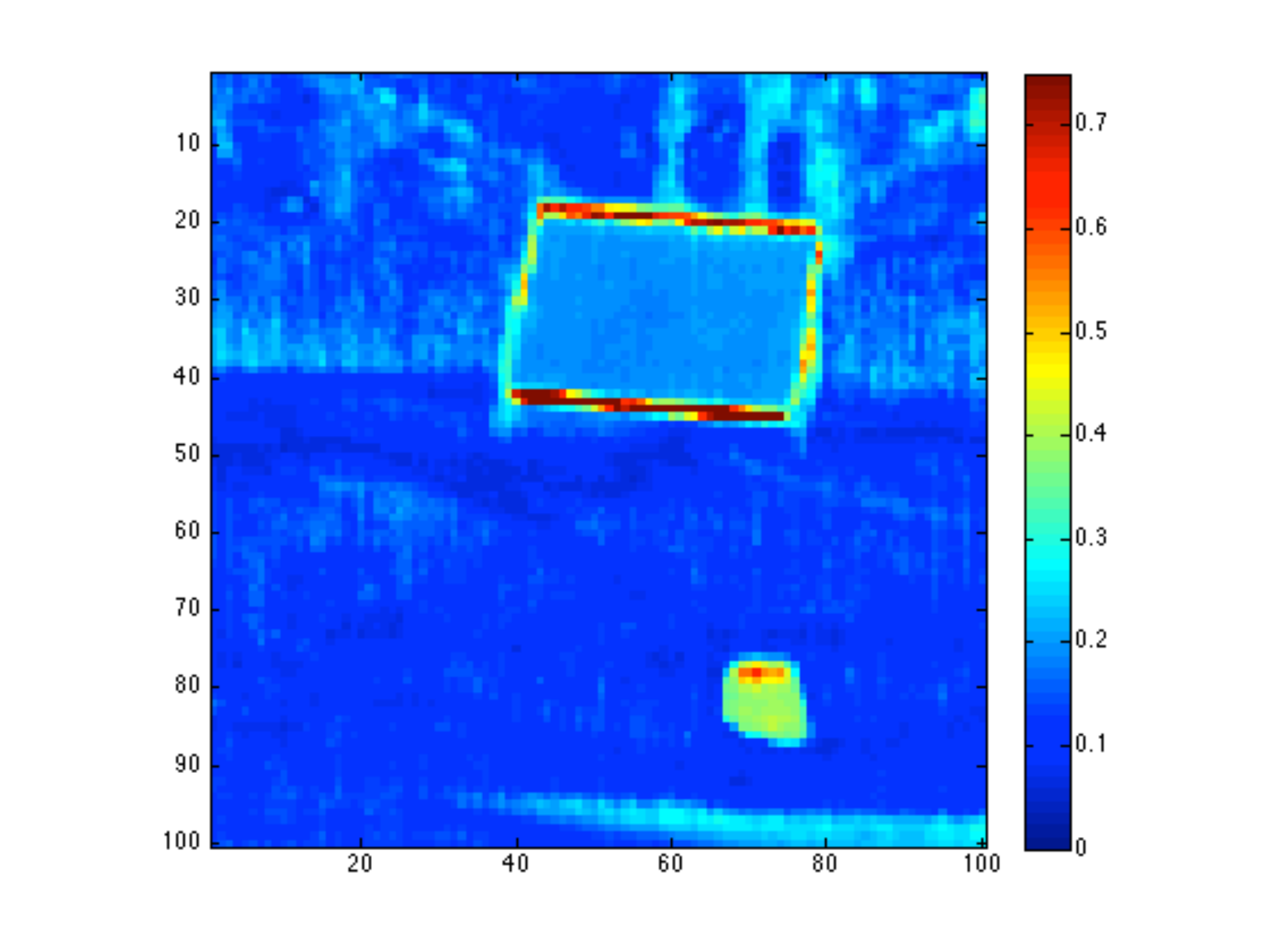}
\label{fig: t1 nov rand}
}
\subfigure[August (noisy random)]{
\includegraphics[width=1.2in]{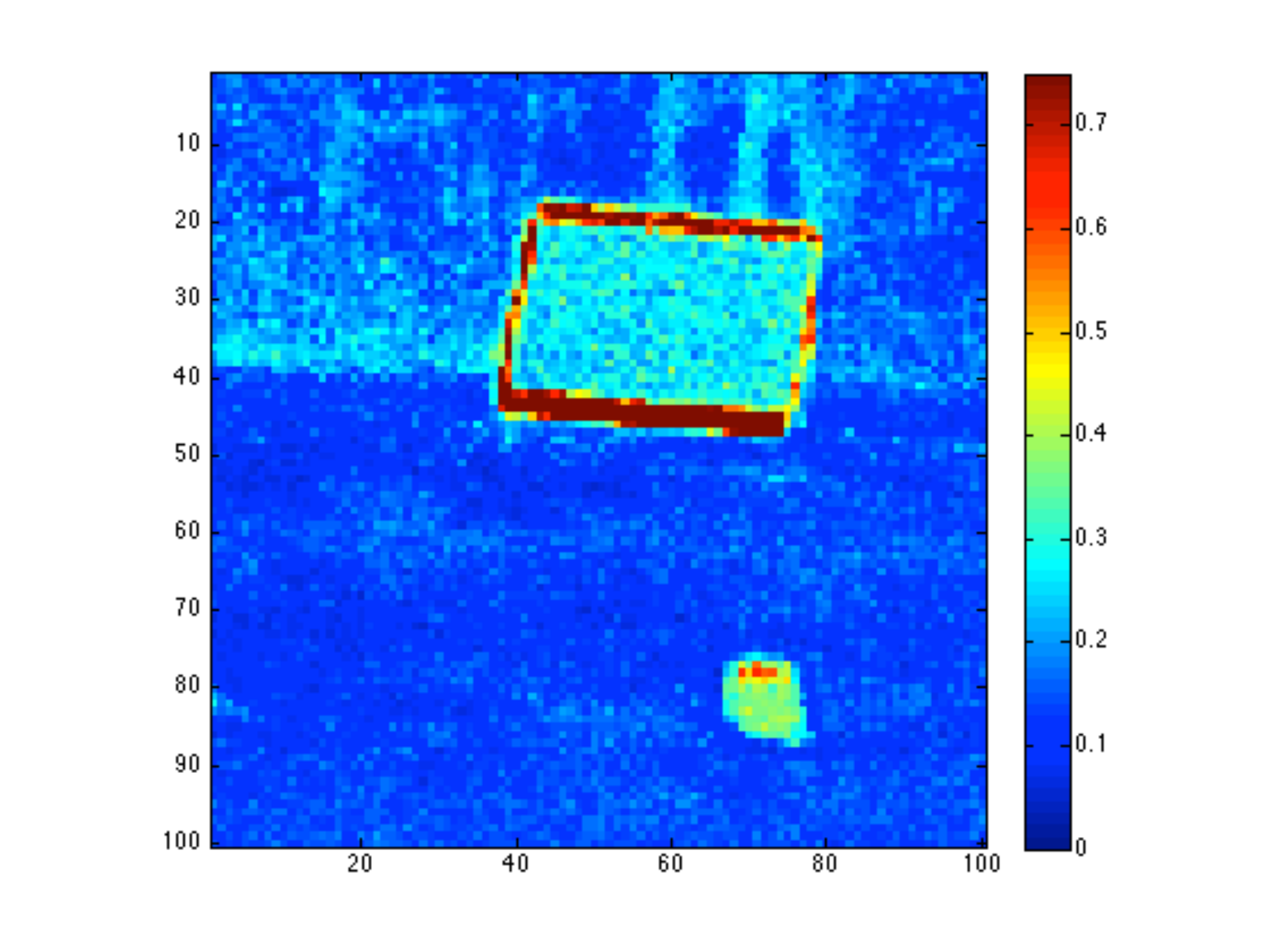}
\label{fig: t1 aug noisy rand}
}
\subfigure[September (noisy random)]{
\includegraphics[width=1.2in]{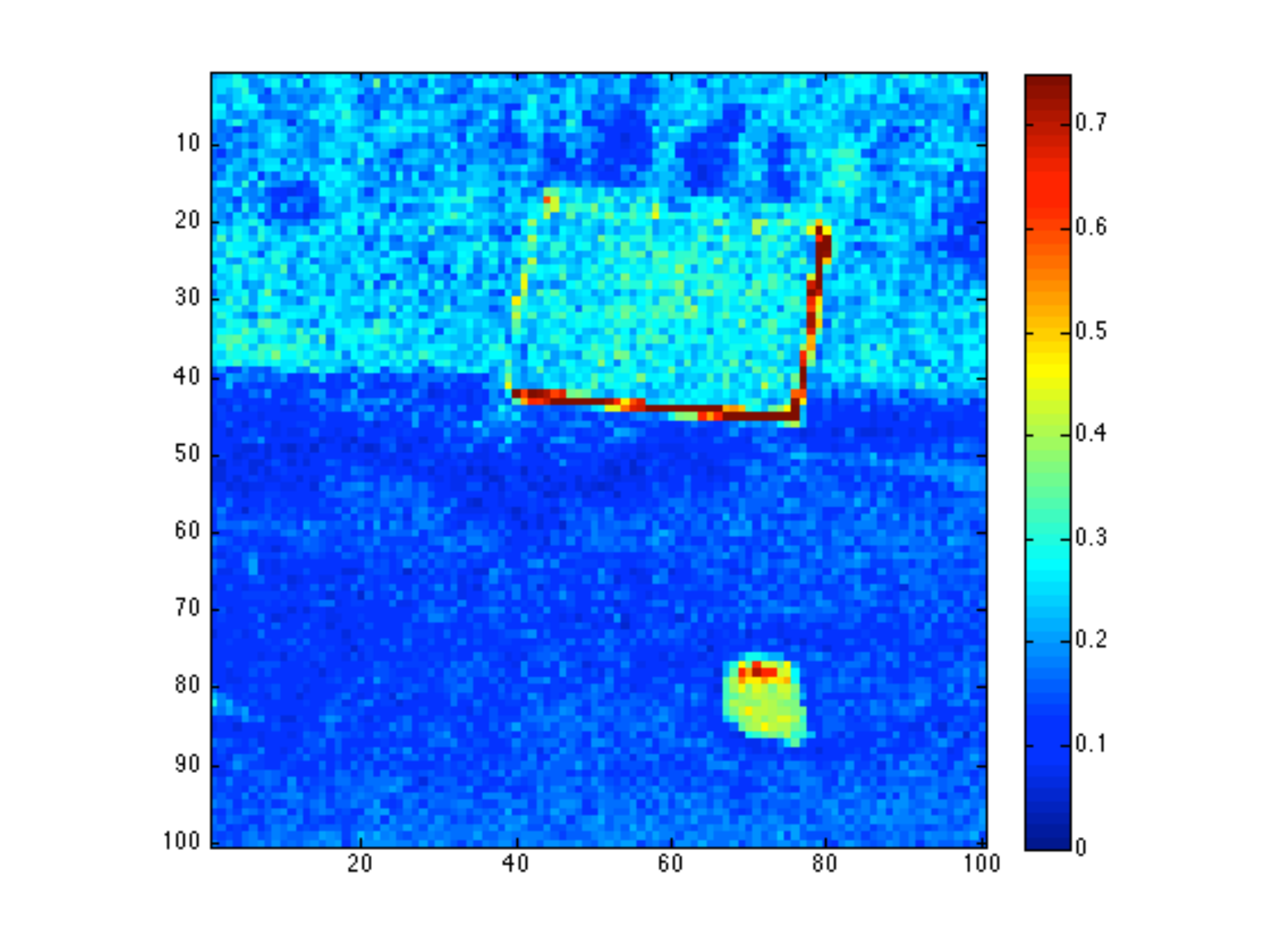}
\label{fig: t1 sep noisy rand}
}
\subfigure[October (noisy random)]{
\includegraphics[width=1.2in]{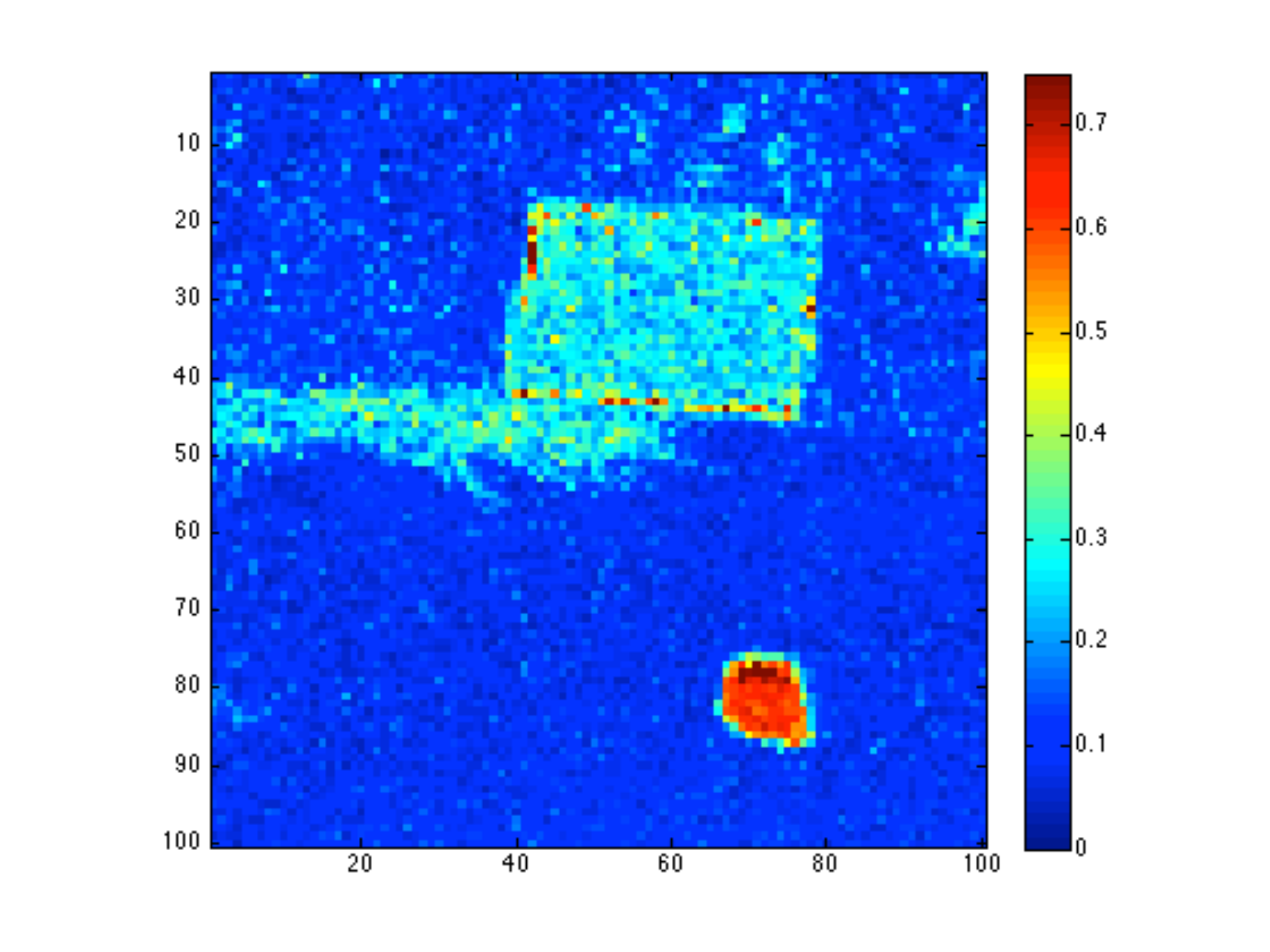}
\label{fig: t1 oct noisy rand}
}
\subfigure[November (noisy random)]{
\includegraphics[width=1.2in]{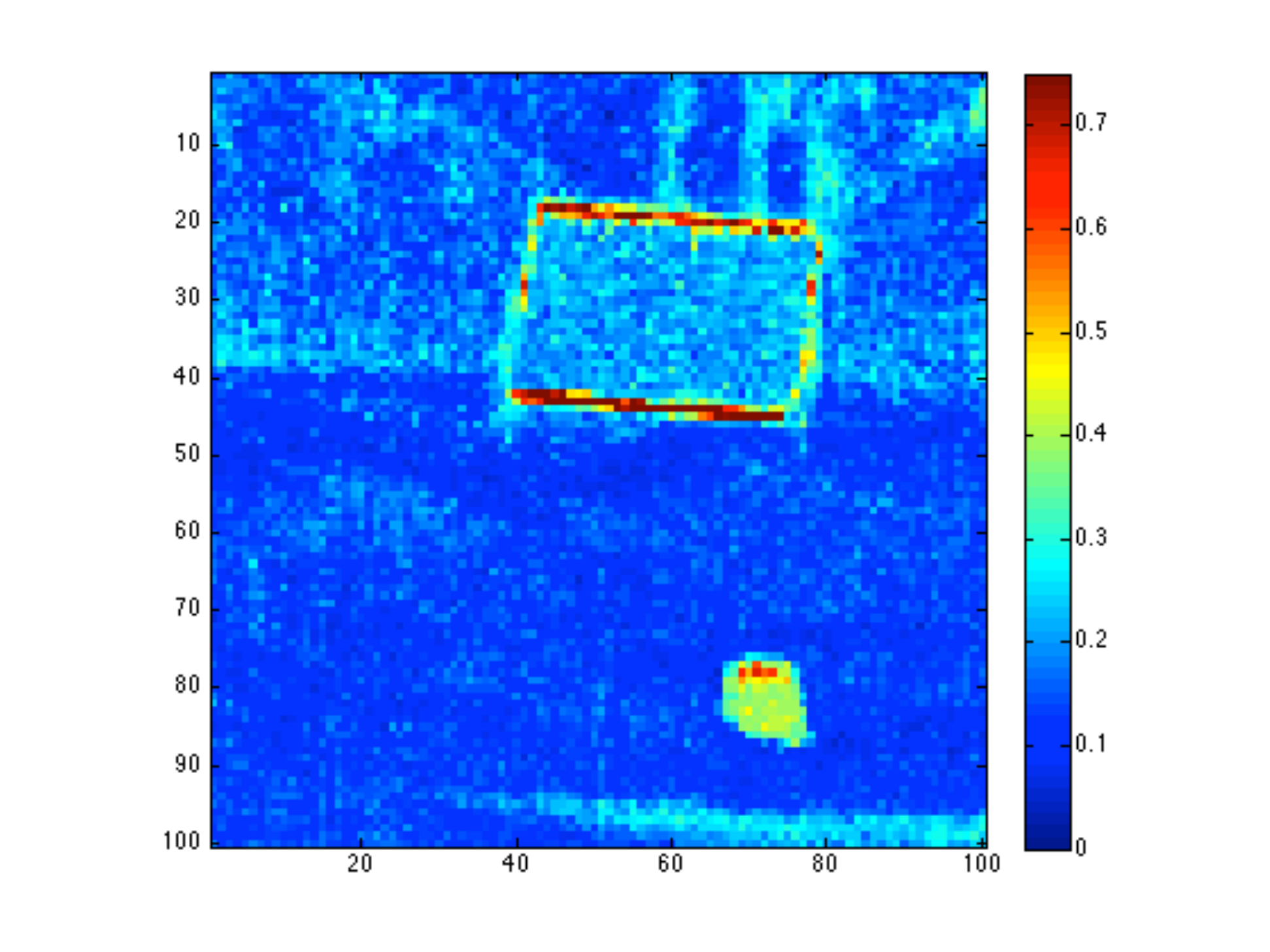}
\label{fig: t1 nov noisy rand}
}
\caption{Map of $D^{(1)}(x_{\text{chg}},x_{\alpha})$ for each $\alpha
  \in \I^{(4)}$ and for each camera type.}
\label{fig: diffusion distance t=1}
\end{figure}

\begin{figure}
\center
\subfigure[August (original)]{
\includegraphics[width=1.2in]{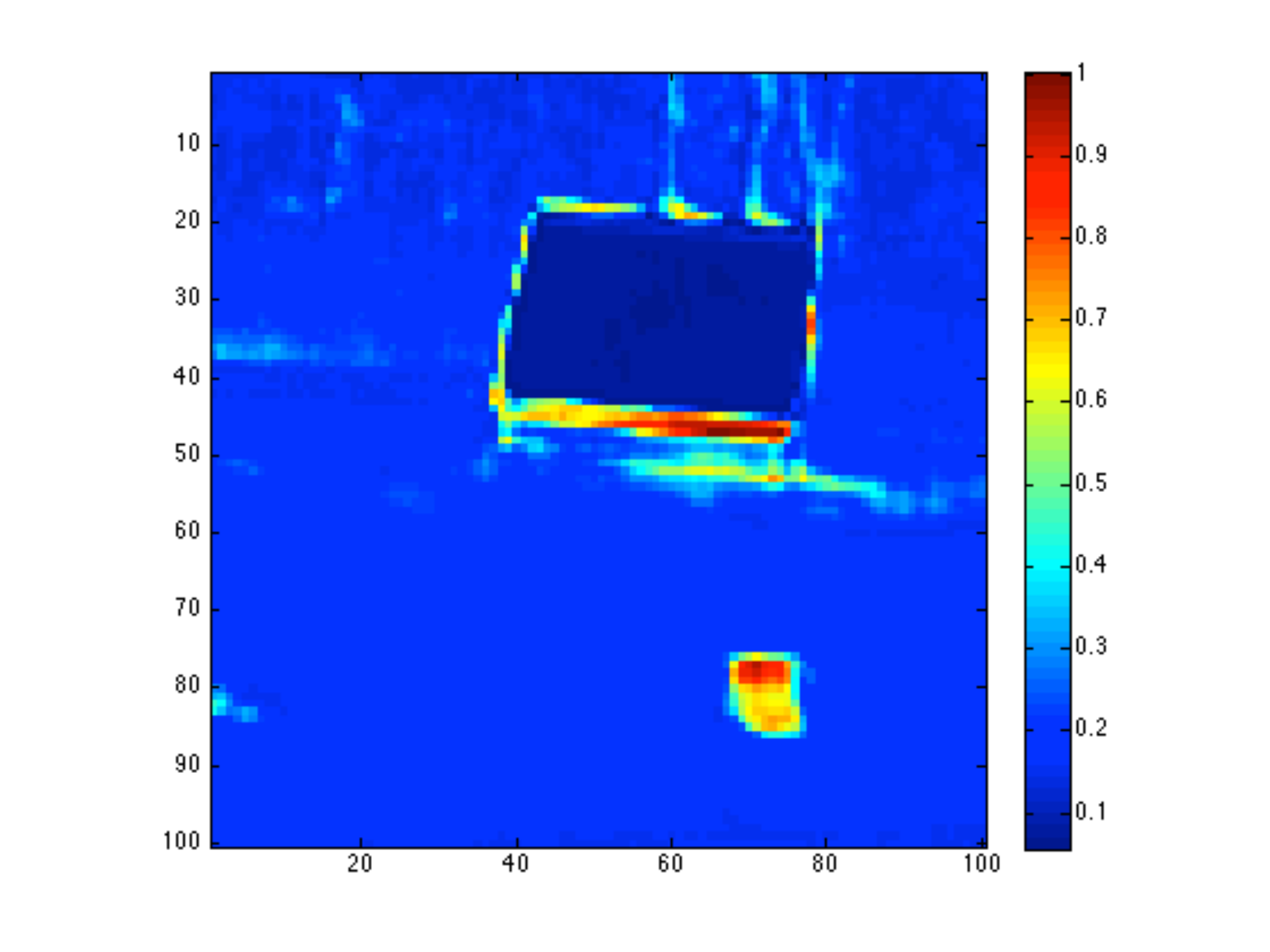}
\label{fig: tinf aug orig}
}
\subfigure[September (original)]{
\includegraphics[width=1.2in]{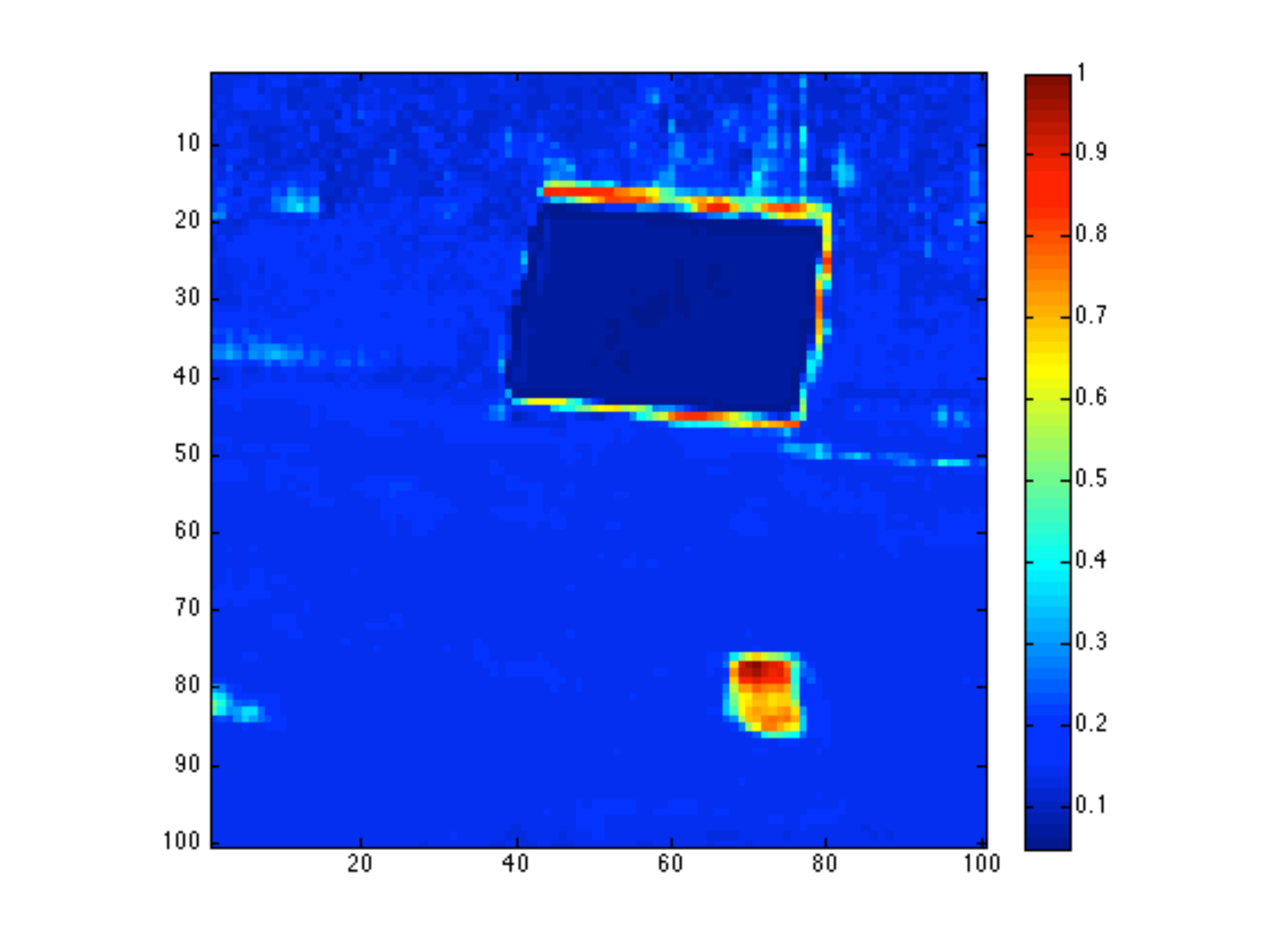}
\label{fig: tinf sep orig}
}
\subfigure[October (original)]{
\includegraphics[width=1.2in]{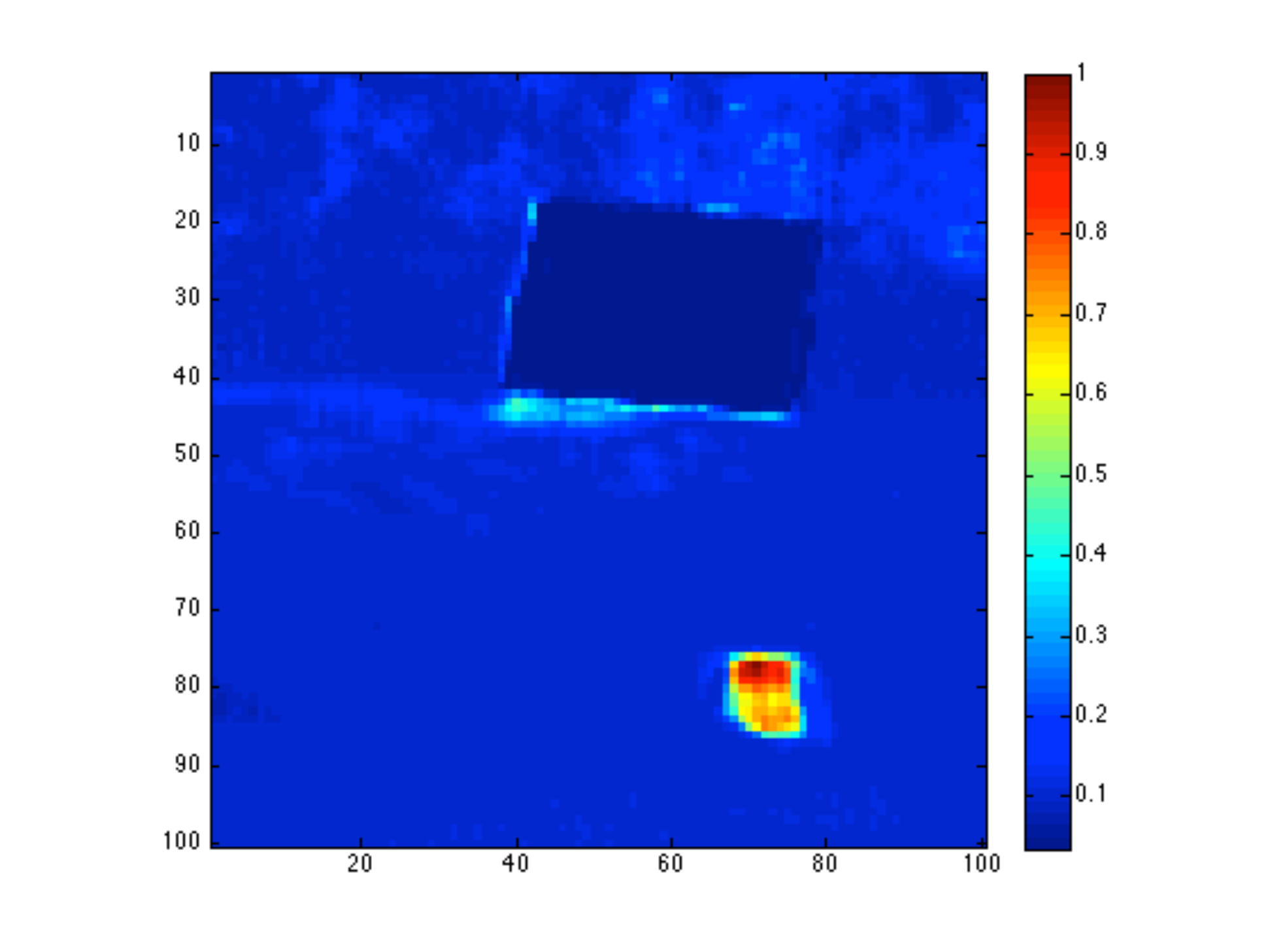}
\label{fig: tinf oct orig}
}
\subfigure[November (original)]{
\includegraphics[width=1.2in]{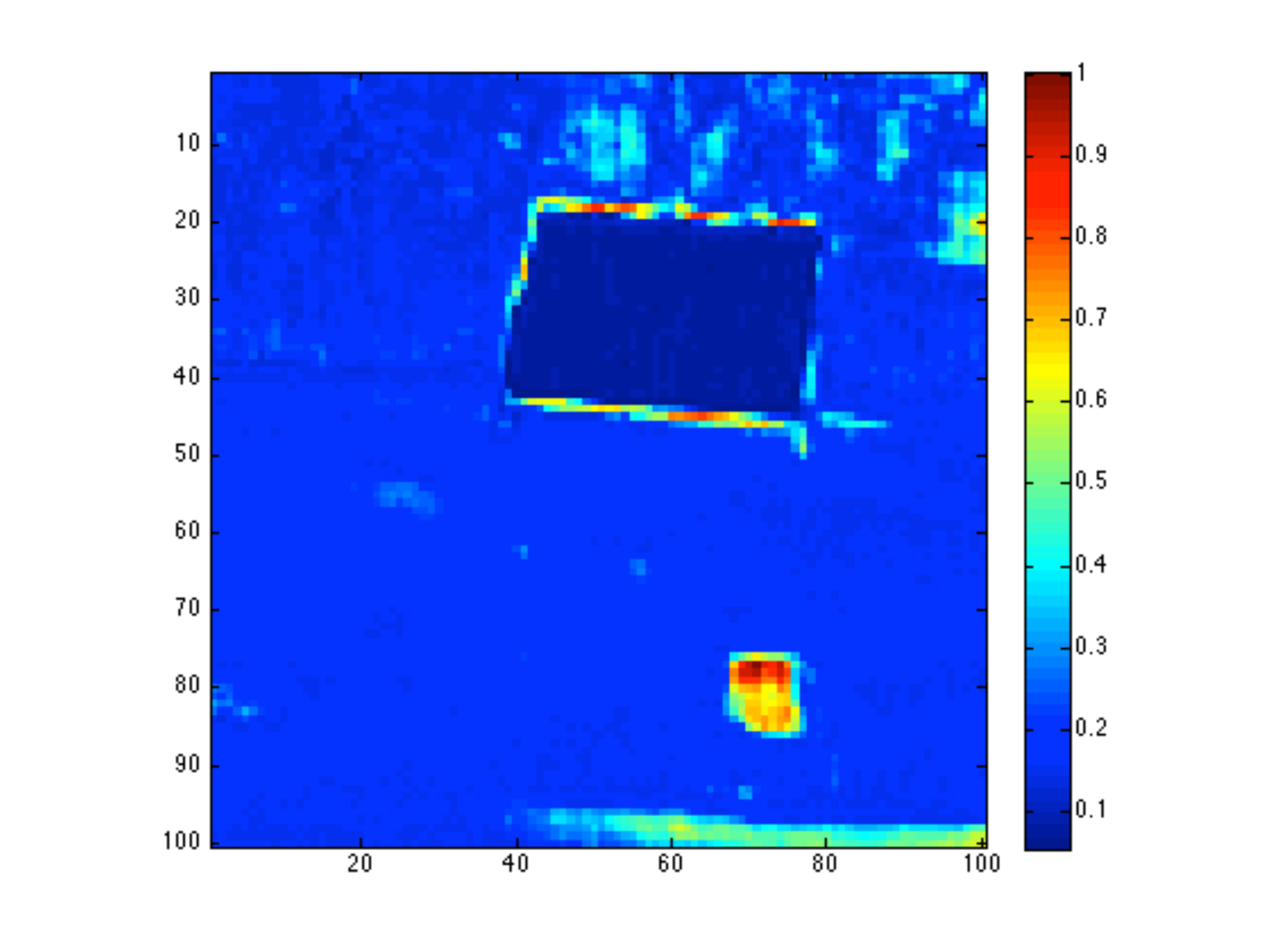}
\label{fig: tinf nov orig}
}
\subfigure[August (random)]{
\includegraphics[width=1.2in]{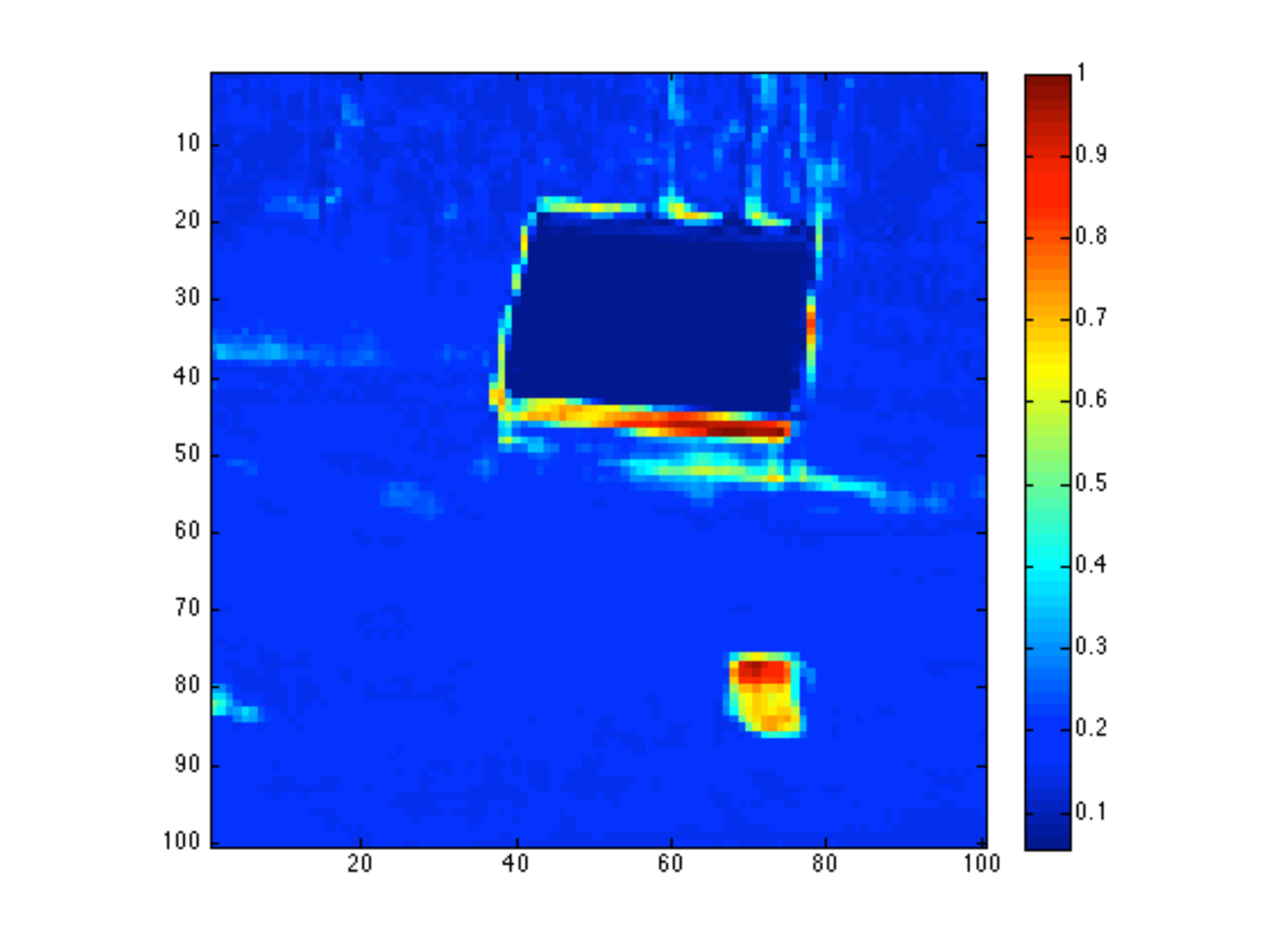}
\label{fig: tinf aug rand}
}
\subfigure[September (random)]{
\includegraphics[width=1.2in]{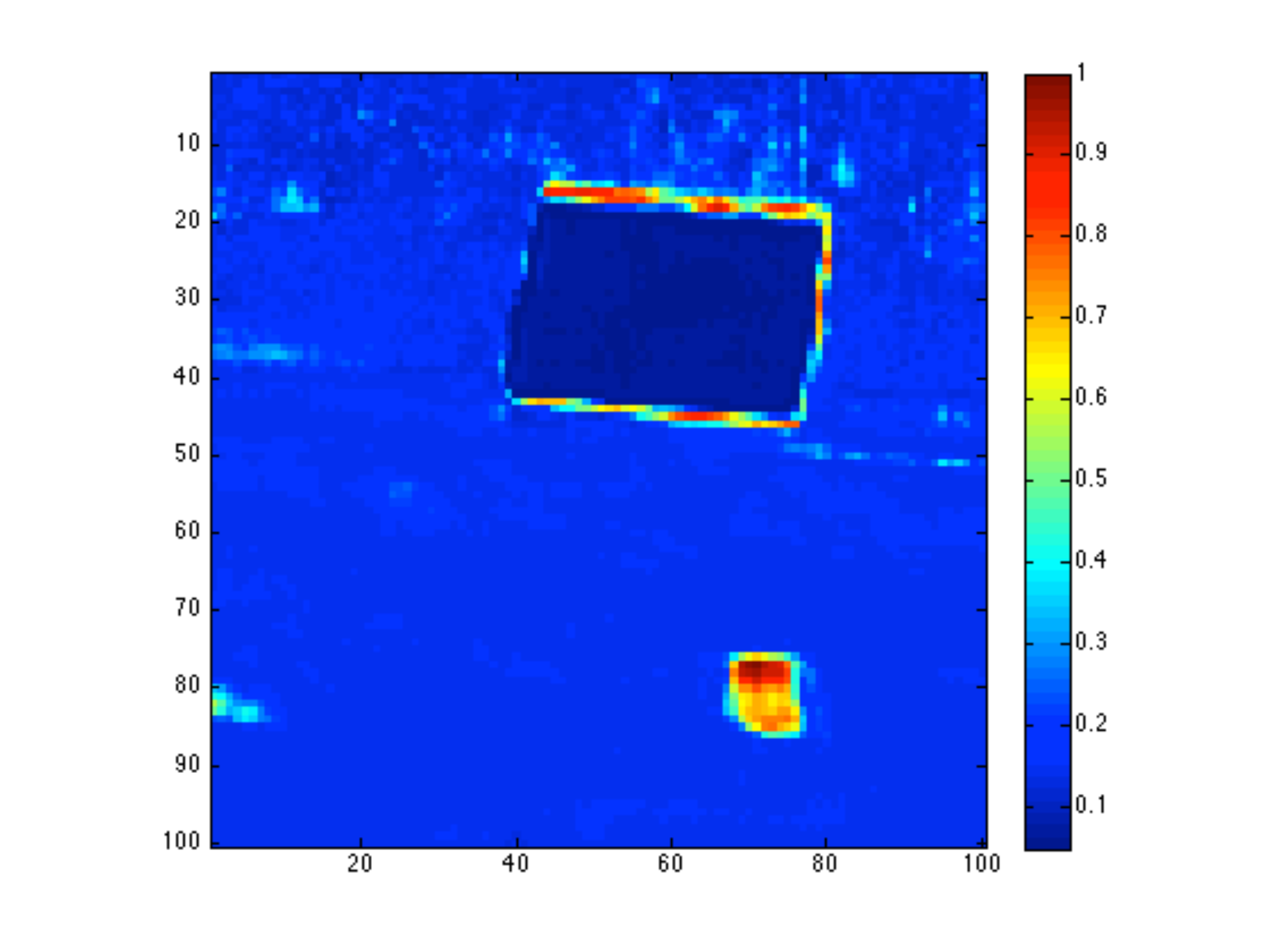}
\label{fig: tinf sep rand}
}
\subfigure[October (random)]{
\includegraphics[width=1.2in]{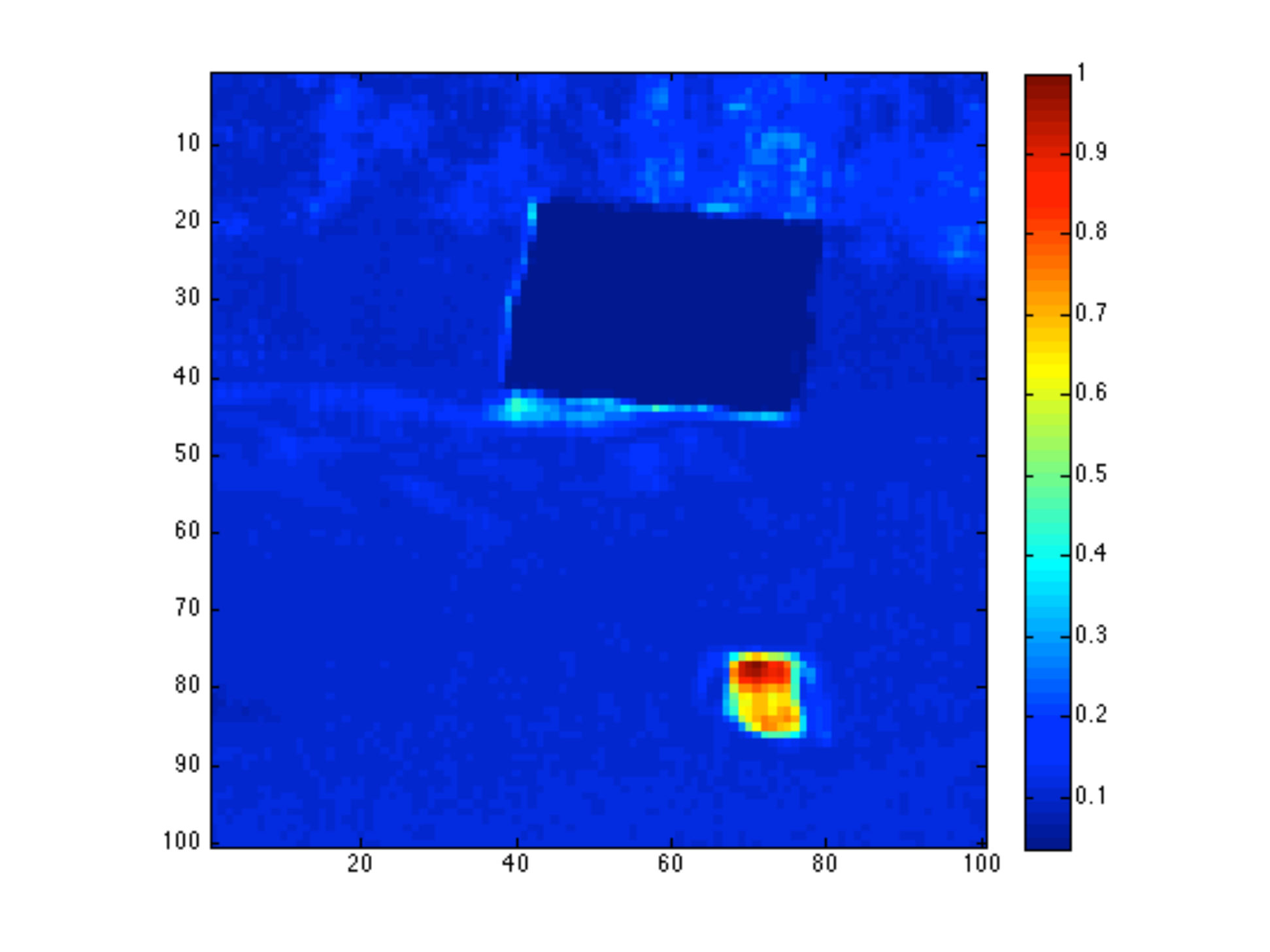}
\label{fig: tinf oct rand}
}
\subfigure[November (random)]{
\includegraphics[width=1.2in]{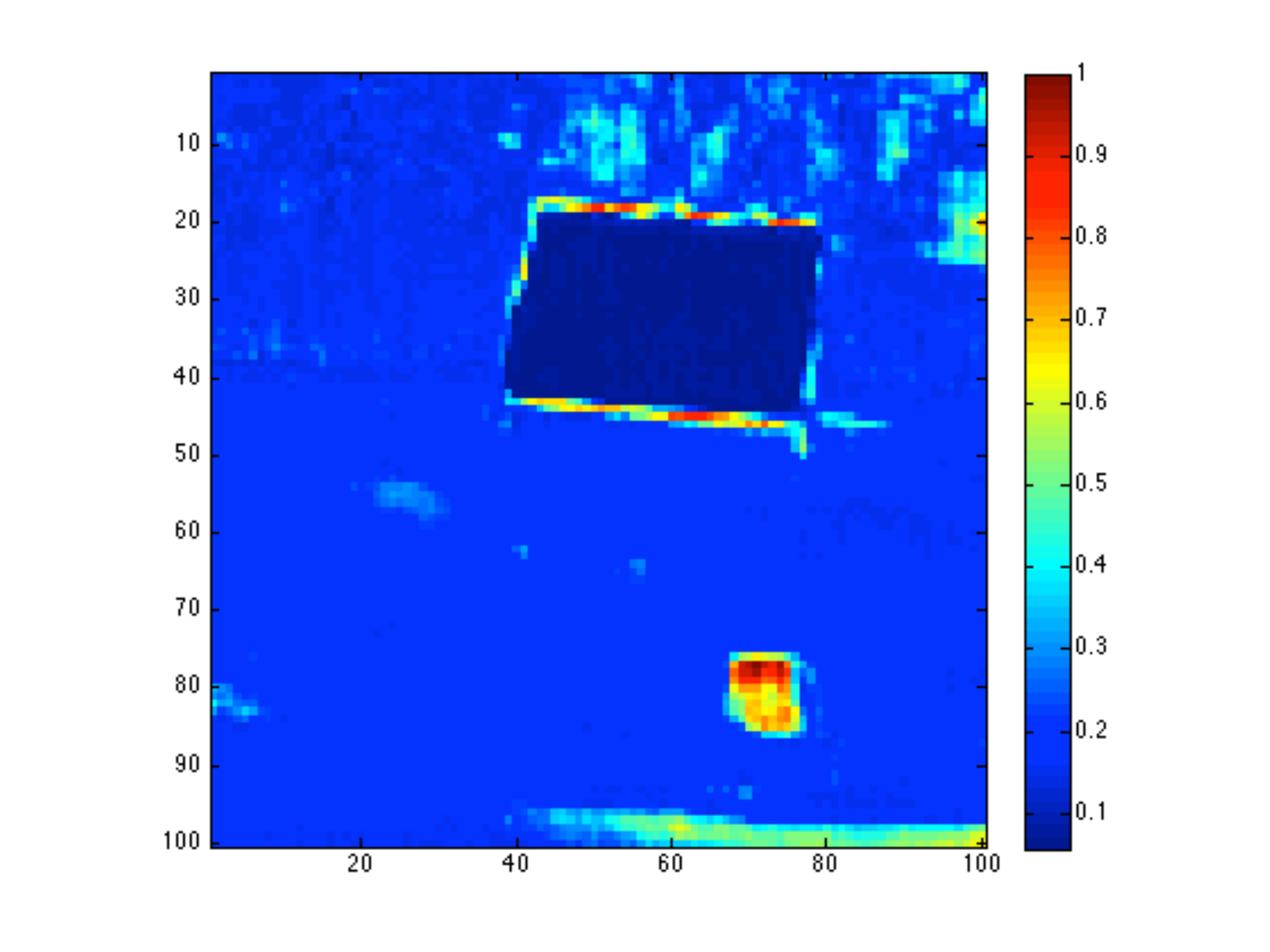}
\label{fig: tinf nov rand}
}
\subfigure[August (noisy random)]{
\includegraphics[width=1.2in]{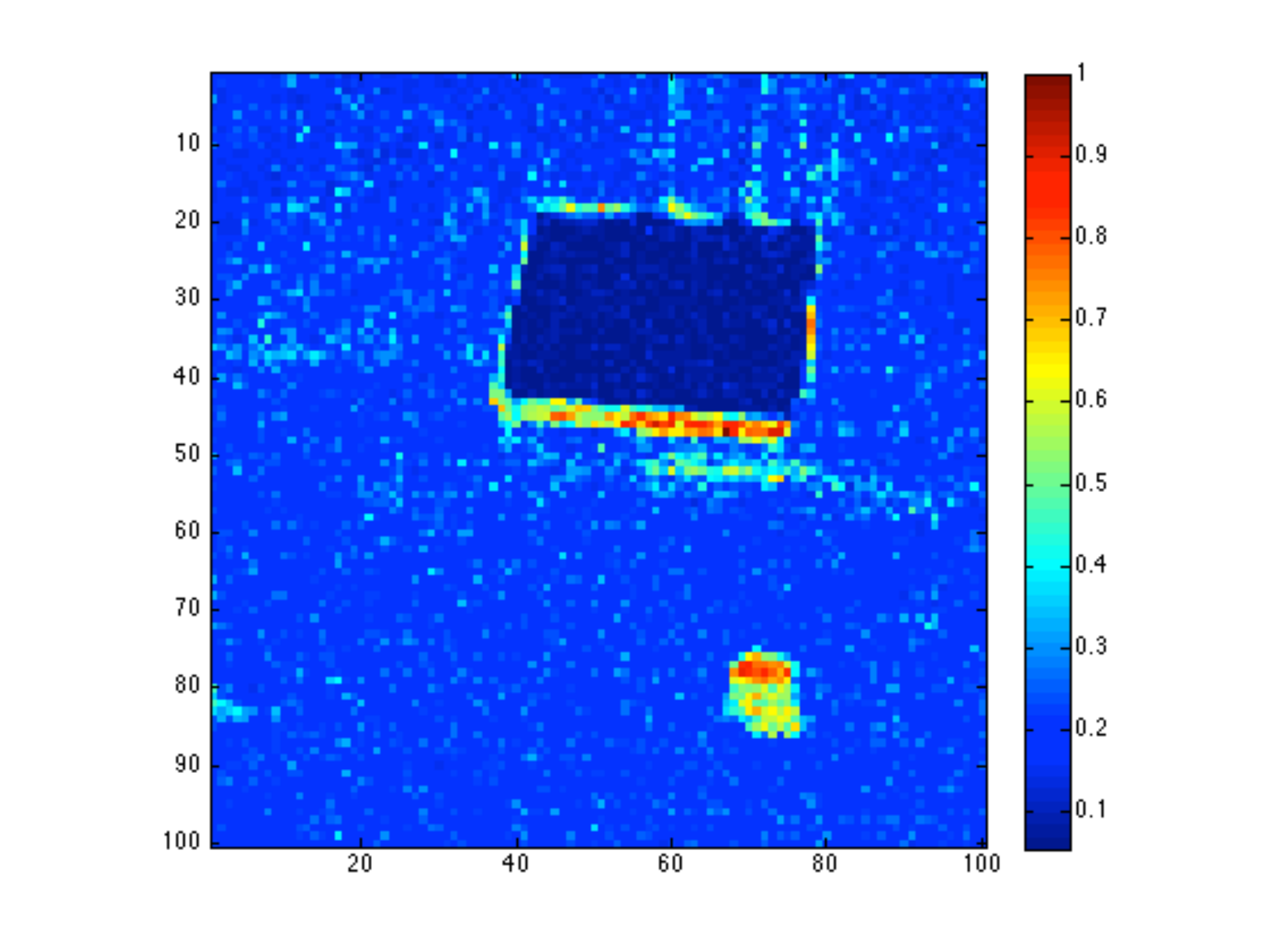}
\label{fig: tinf aug noisy rand}
}
\subfigure[September (noisy random)]{
\includegraphics[width=1.2in]{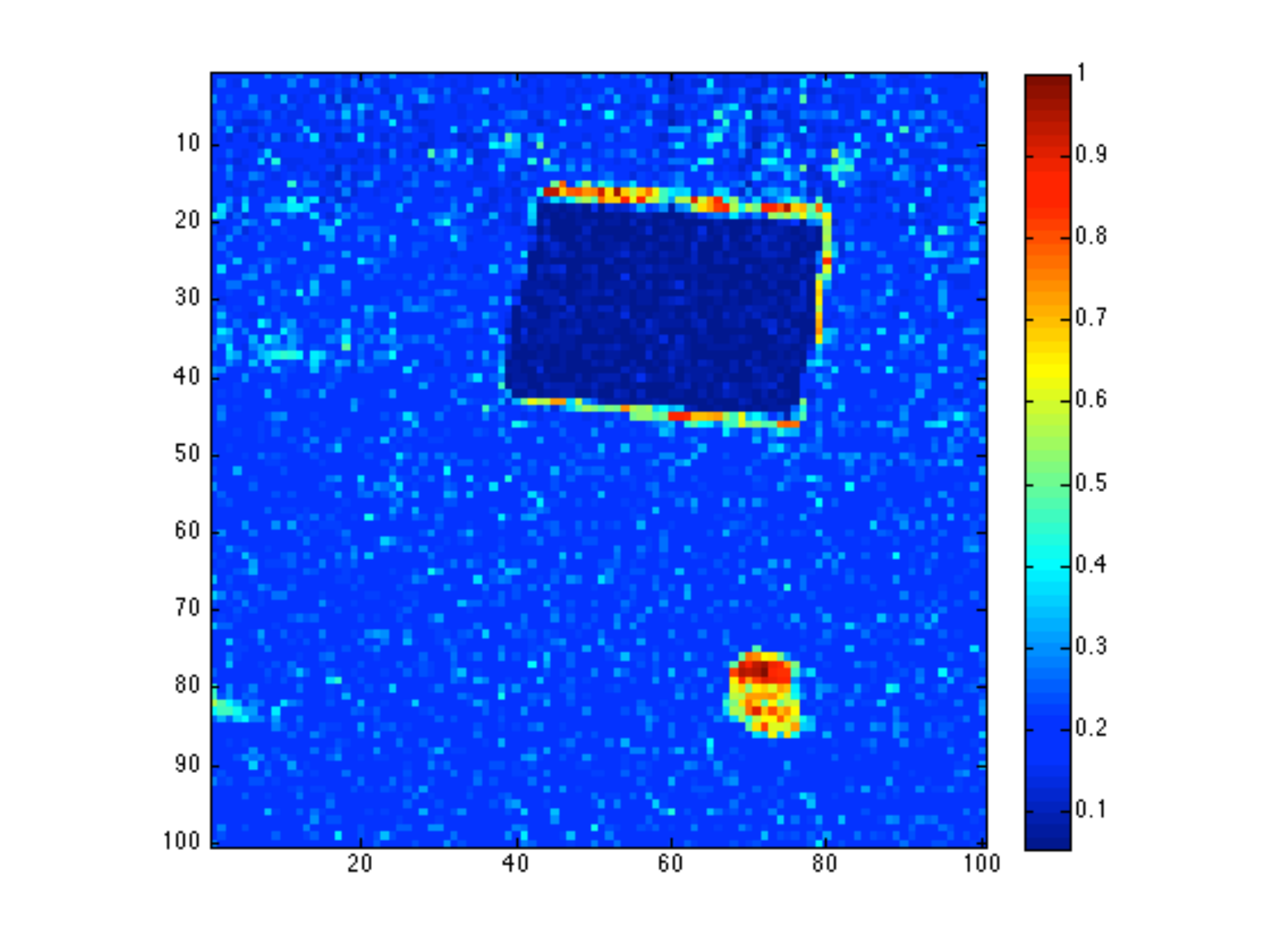}
\label{fig: tinf sep noisy rand}
}
\subfigure[October (noisy random)]{
\includegraphics[width=1.2in]{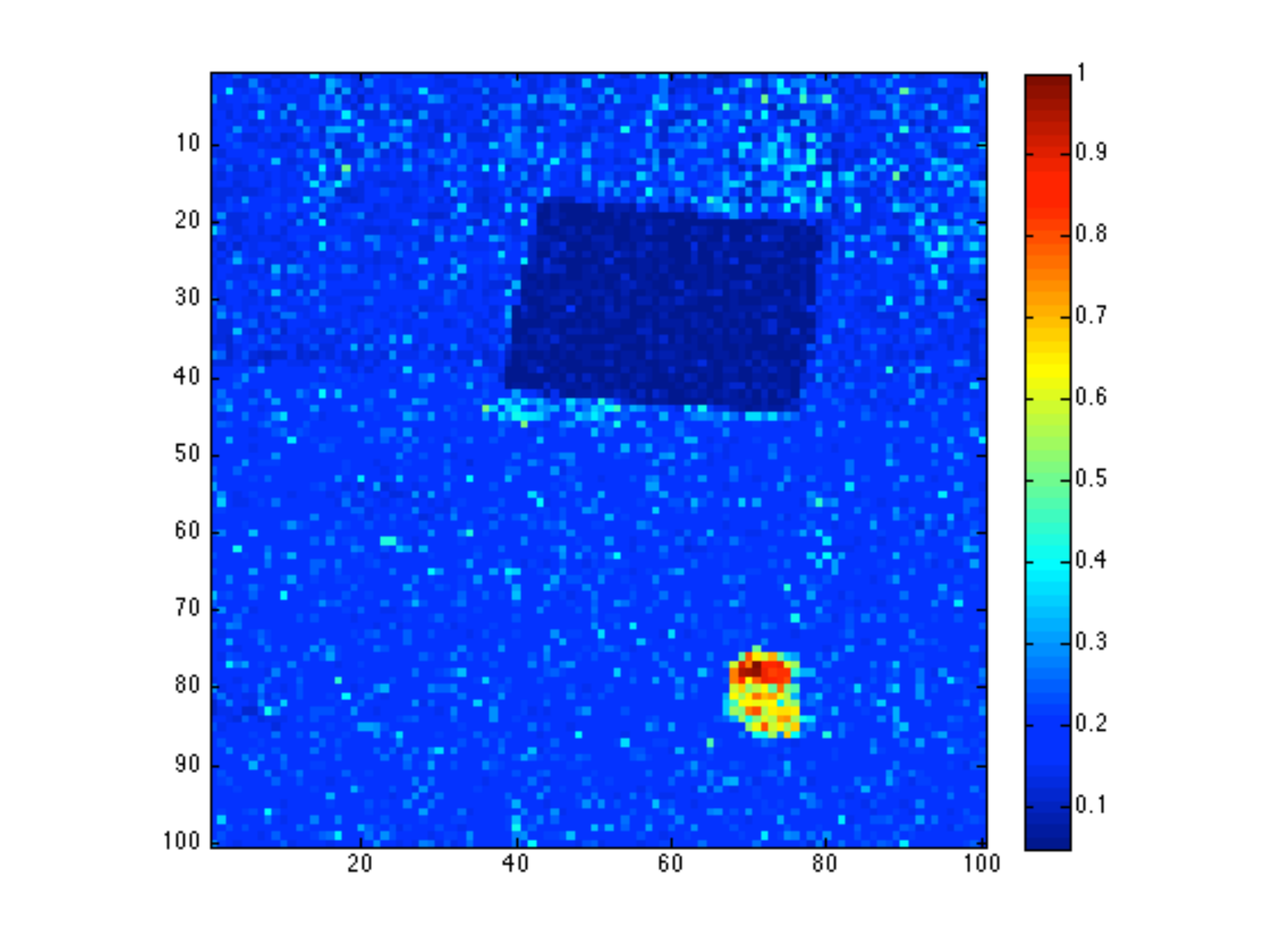}
\label{fig: tinf oct noisy rand}
}
\subfigure[November (noisy random)]{
\includegraphics[width=1.2in]{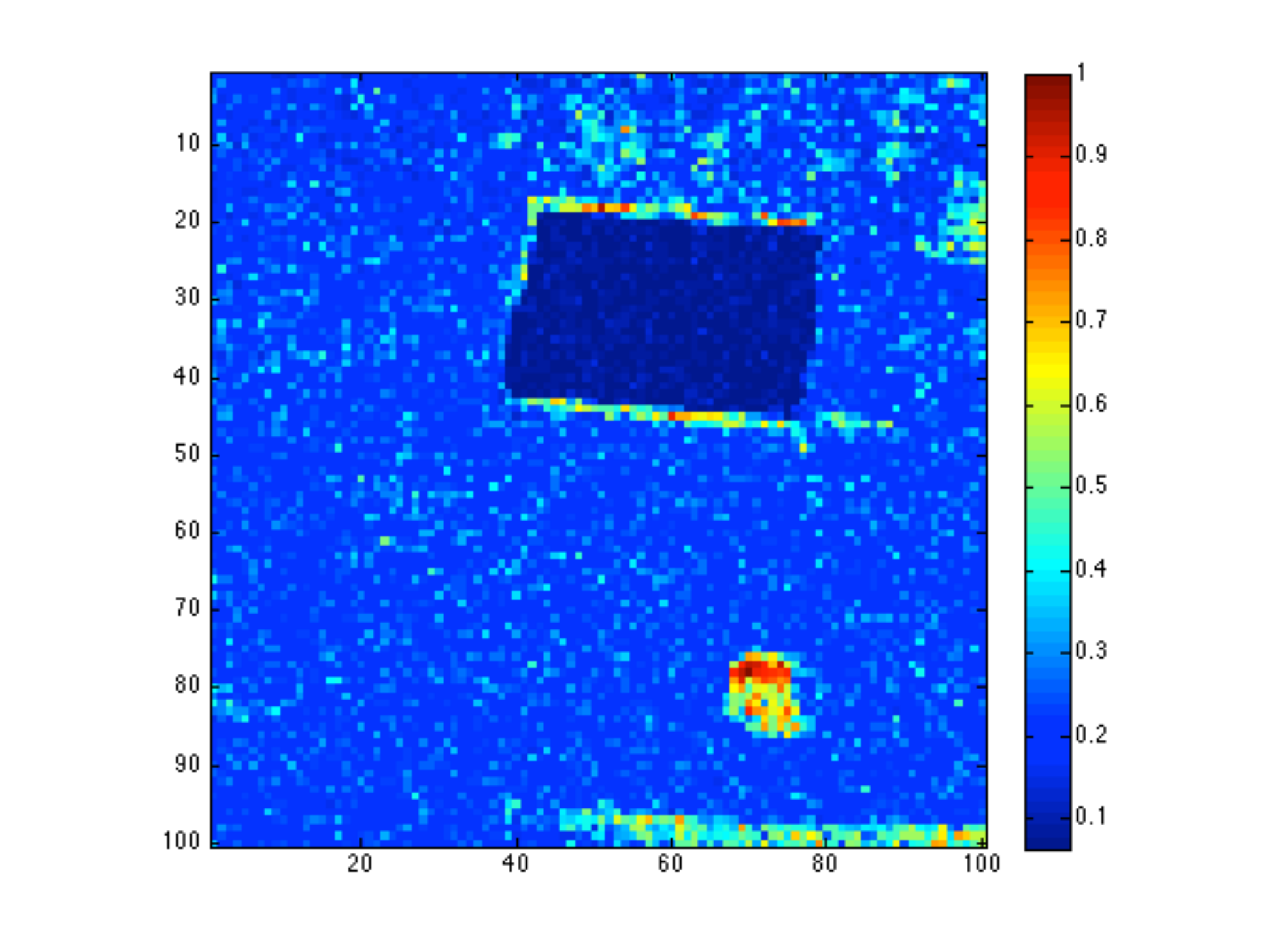}
\label{fig: tinf nov noisy rand}
}
\caption{Map of $\lim_{t \rightarrow
    \infty}D^{(t)}(x_{\text{chg}},x_{\alpha})$ for each $\alpha \in
  \I^{(4)}$ and for each camera type.}
\label{fig: diffusion distance t=infty}
\end{figure}

\begin{figure}
\center
\subfigure[Original camera]{
\includegraphics[width=1.6in]{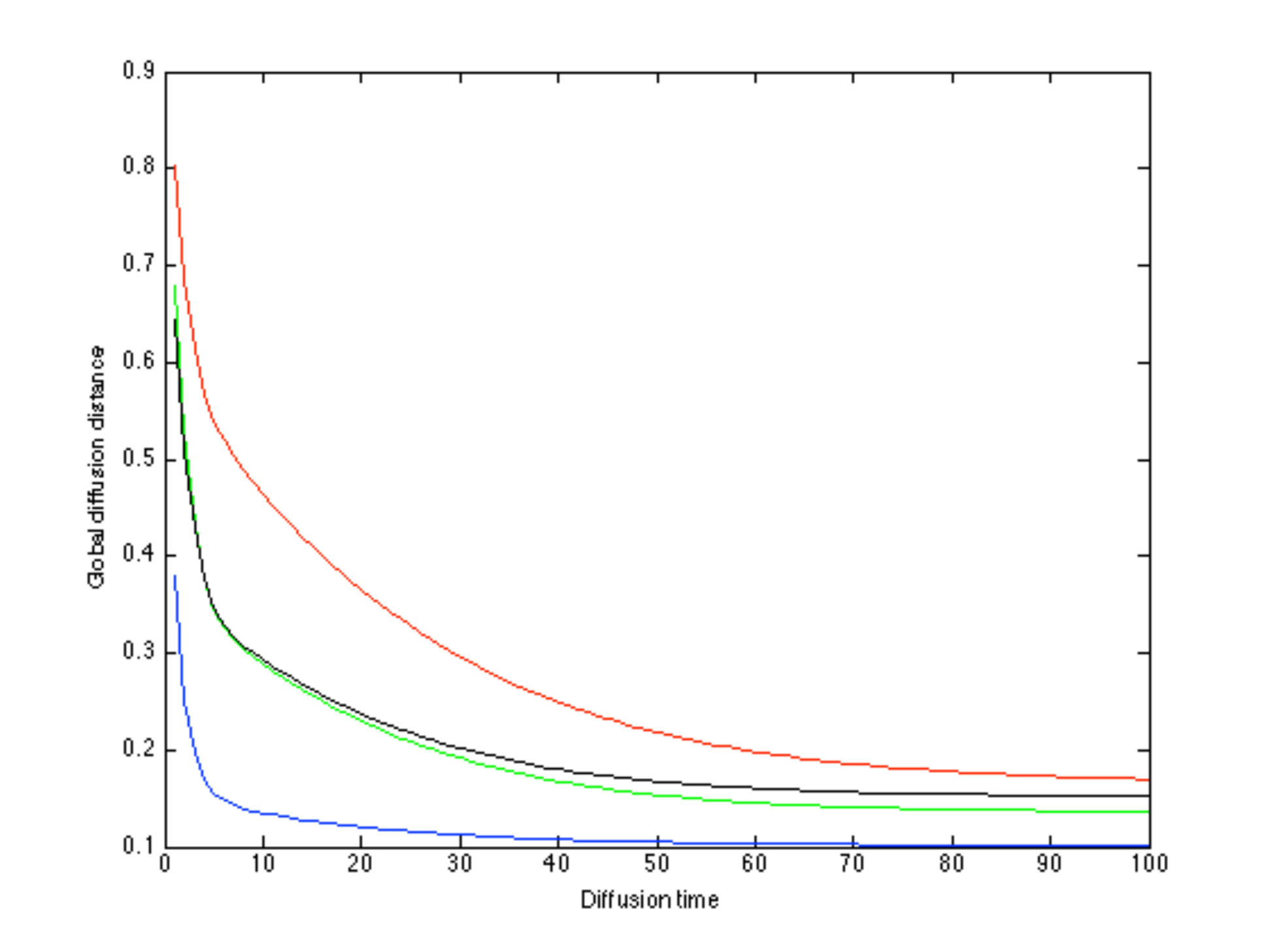}
\label{fig: global diff dist orig}
}
\subfigure[Random camera]{
\includegraphics[width=1.6in]{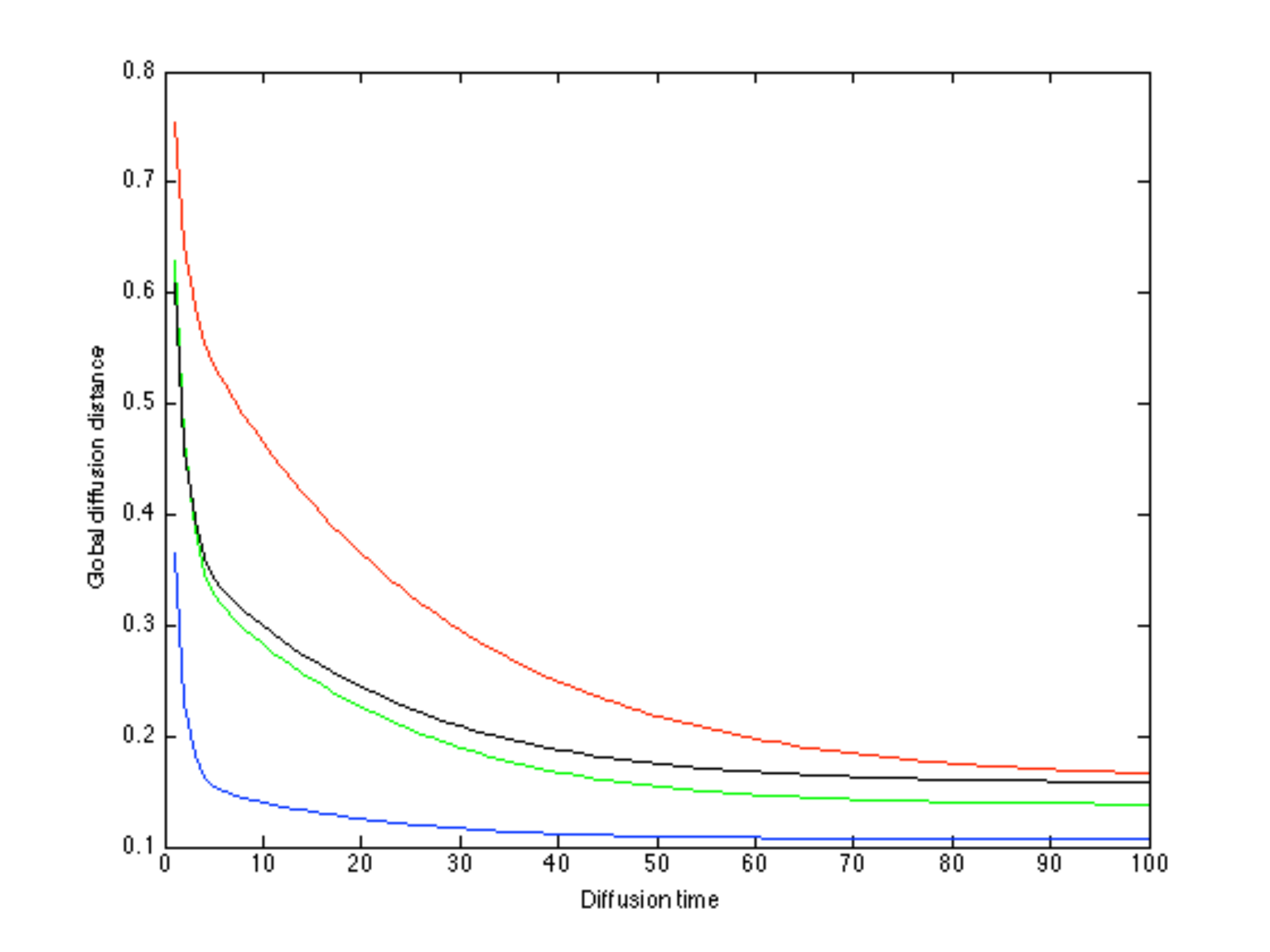}
\label{fig: global diff dist rand}
}
\subfigure[Noisy random camera]{
\includegraphics[width=1.6in]{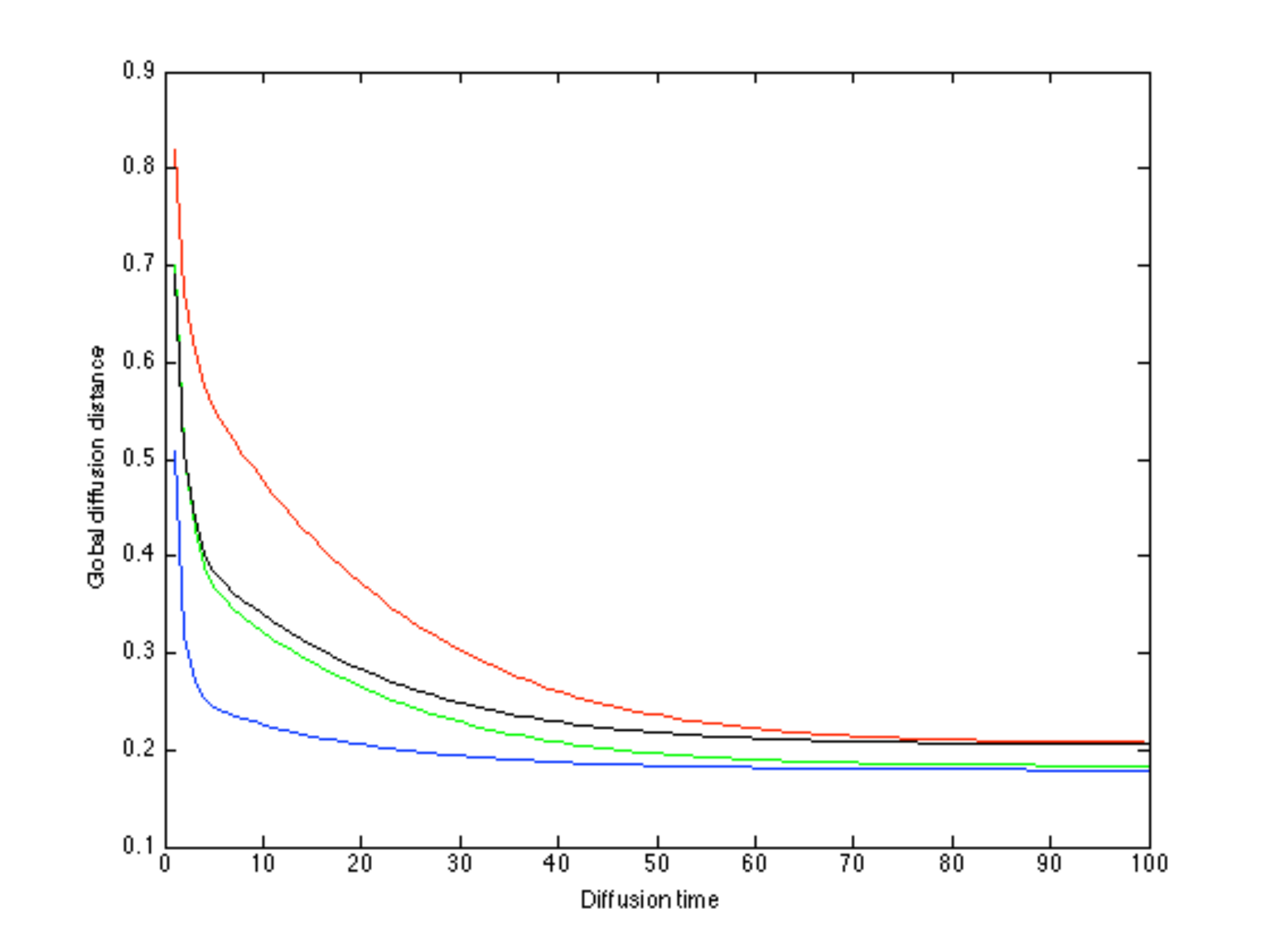}
\label{fig: global diff dist noisy rand}
}
\caption{Global diffusion distance. Red: $\Dt(\Gamma_{\text{chg}},
  \Gamma_{\text{aug}})$, green: $\Dt(\Gamma_{\text{chg}},
  \Gamma_{\text{sep}})$, blue: $\Dt(\Gamma_{\text{chg}},
  \Gamma_{\text{oct}})$, black: $\Dt(\Gamma_{\text{chg}},
  \Gamma_{\text{nov}})$}
\label{fig: global diffusion distance}
\end{figure}

While the spectra of the various months were perturbed by the changing
seasons as well as different lighting conditions, the authors of
\cite{eismann:hsiChangeDetection2008} did use
the same camera for each image so it is reasonable to assume that one
could directly compare spectra across the four months. Thus we simulated
a scenario in which different cameras were used, measuring different
wavelengths. In this test, a direct comparison becomes nearly
impossible, and so one must turn to an indirect comparison such as the
diffusion distance.

To carry out the experiment, we did the following. For each of the
five images, we randomly selected $D_{\alpha}$ bands to use out of
the original $124$ bands; we also randomly reordered each set of $D_{\alpha}$
bands. The values of $D_{\alpha}$ are the following: $D_{\text{aug}} =
30$, $D_{\text{sep}} = 40$, $D_{\text{oct}} = 60$, $D_{\text{nov}} =
70$, and $D_{\text{chg}} = 50$. Thus for this experiment, $X_{\alpha}$, for each $\alpha
\in \I$, contains data points in $\R^{D_{\alpha}}$. To see an example of these new
spectra, we refer the reader to Figure \ref{fig: random
  spectra}. Using the measurements from this ``random camera,'' we
then proceeded to carry out the experiment exactly as before,
computing the diffusion distance for $t=1$ (Figures \ref{fig: t1 aug
  rand}, \ref{fig: t1 sep rand}, \ref{fig: t1 oct rand}, \ref{fig: t1
  nov rand}), the asymptotic diffusion
distance (Figures \ref{fig: tinf aug rand}, \ref{fig: tinf sep rand},
\ref{fig: tinf oct rand}, \ref{fig: tinf nov rand}), and the global diffusion distance (Figure
\ref{fig: global diff dist rand}).

For a third and final experiment, we took the spectra from the random
camera in the previous experiment and added Gaussian noise sampled
from the normal distribution with mean zero and standard deviation
$0.01$. This gave us an average signal to noise ratio (SNR) of $19.2$ dB
(note, we compute $\text{SNR} =
10\log_{10}(\text{mean}(x^2)/\text{mean}(\eta^2))$, where $x$ is the
signal and $\eta$ is the noise). Once more we carried out the
experiment, the same as before, computing the diffusion distance for
$t=1$ (Figures \ref{fig: t1 aug noisy rand}, \ref{fig: t1 sep noisy
  rand}, \ref{fig: t1 oct noisy rand}, \ref{fig: t1 nov noisy rand}),
the asymptotic diffusion distance (Figures \ref{fig: tinf aug noisy
  rand}, \ref{fig: tinf sep noisy rand}, \ref{fig: tinf oct noisy
  rand}, \ref{fig: tinf nov noisy rand}), and the global diffusion
distance (Figure \ref{fig: global diff dist noisy rand}).

Examining Figures \ref{fig: diffusion distance t=1}, \ref{fig:
  diffusion distance t=infty}, and \ref{fig: global diffusion
  distance}, we see that the results are similar across all three
cameras (the original camera, the random camera, and the noisy random
camera). This result points to the two properties mentioned at the
beginning of this section: that the common embedding defined by
Theorem \ref{thm: common embedding} is sensor independent and robust
against noise. Thus the method is consistent under a variety of
different conditions.

In terms of the change detection task, the diffusion
distance is also accurate. For the diffusion time $t = 1$, we see from the maps in Figure
\ref{fig: diffusion distance t=1} that the tarp is recognized as a
change. However, other changes due to the lighting or the change in
seasons also appear. For example, even in October, the small change in
the shadow is visible, while in August, September, and November the
change in lighting causes the panel to be highlighted. Also, in some
months even the trees have a weak, but noticeable difference in the
their diffusion distances. When we allow $t \rightarrow \infty$
though, the smaller clusters merge together and the changes due to
lighting and seasonal differences are filtered out. As one can see
from Figure \ref{fig: diffusion distance t=infty}, all that is left is
the change due to the added tarp (note that the change around the
border of the panel is due to it being slightly shifted from month to
month). Thus we see that the diffusion distance and
corresponding diffusion map gives a
natural representation of the data that can be
used to filter types of changes at different scales. In practice,
after these mappings and distances have been computed, the images can
be handed off to an analyst who should be able to pick out the changes
with ease; alternatively, a classification algorithm can be used on
the backend (for example, one that looks for diffusion distances
across images that are larger than a certain prescribed scale). 

For the global diffusion distances in Figure \ref{fig: global
  diffusion distance}, we see several intuitions borne
out in this particular application. First, the closer the month in
real time to October (the month in which the changed data set was
recorded), the smaller the global diffusion distance. Secondly, we see
that as the diffusion time $t$ gets larger, the smaller the global
diffusion distance. 

\subsection{Parameterized difference equations}

In this section we consider discrete time dynamical systems
(difference equations) that depend on input parameters. The idea is to
use the diffusion geometric principles outlined in this paper to
understand how the geometry of the system changes as one changes the
parameters of the system.

To illustrate the idea we use the following example of the Standard
Map, first brought to our attention by Igor Mezi\'{c} and Roy Lederman (personal
correspondence). The Standard Map is an area preserving chaotic map
from the torus $\T^2 \triangleq 2\pi (S^1 \times S^1)$ onto
itself. Let $(p, \theta) \in \T^2$ denote an arbitrary coordinate of
the torus. For any initial condition $(p_0, \theta_0) \in \T^2$, the
Standard Map is defined by the following two equations:
\begin{align*}
p_{\ell+1} &\triangleq p_{\ell} + \alpha\sin(\theta_{\ell}) \text{ mod
  } 2\pi, \\
\theta_{\ell+1} &\triangleq \theta_{\ell} + p_{\ell+1} \text{ mod } 2\pi,
\end{align*}
where $\alpha \in \I = [0,\infty)$ is a parameter, $\ell \in
\N \cup \{0\}$, and $(p_{\ell}, \theta_{\ell}) \in \T^2$ for all $\ell
\geq 0$. The sequence of points $\gamma(p_0,\theta_0) \triangleq \{(p_{\ell},\theta_{\ell})\}_{\ell
  \geq 0}$ constitutes the orbit derived from the initial condition
$(p_0,\theta_0)$. When $\alpha = 0$, the Standard Map consists solely
of periodic and quasiperiodic orbits. For $\alpha > 0$, the map is is increasingly nonlinear as
$\alpha$ grows, which in turn increases the number of initial
conditions that lead to chaotic dynamics.

We take the data set $X_{\alpha}$ to be the set of orbits of the
Standard Map for the parameter $\alpha$. Using the ideas developed in
\cite{levnajic:ergodicTheoryVisualI, levnajic:ergodicTheoryVisualII},
it is possible to define a kernel $k_{\alpha}$ that acts on this data
set. One can in turn use this kernel to define a diffusion map on the
orbits. For the purposes of this experiment, we discretize the orbits
by selecting a grid of initial conditions on $\T^2$ and let the system
run forward a prescribed number of time steps. An example for small
$\alpha$ is given in Figure \ref{fig: diffmap and standard
  map}. Notice how the diffusion map embedding into $\R^3$ organizes
the Standard Map according to the geometry of the orbits. 

For each $\alpha \in [0, \infty)$, we have a similar embedding. Using the ideas contained in
Section \ref{sec: graph distance} (in particular Theorem \ref{thm:
  common embedding}), it is possible to map each
embedding, for all $\alpha \in \I$, into a single low dimensional
Euclidean space. Doing so allows one to observe how the geometry of
the system changes as the parameter $\alpha$ is increased; see Figure
\ref{fig: common embedding standard map} for more details. In the forthcoming paper
\cite{coifman:diffMapStdMap2012}, we give a full treatment of these
ideas. 

\begin{figure}
\center
\includegraphics[width=4.7in]{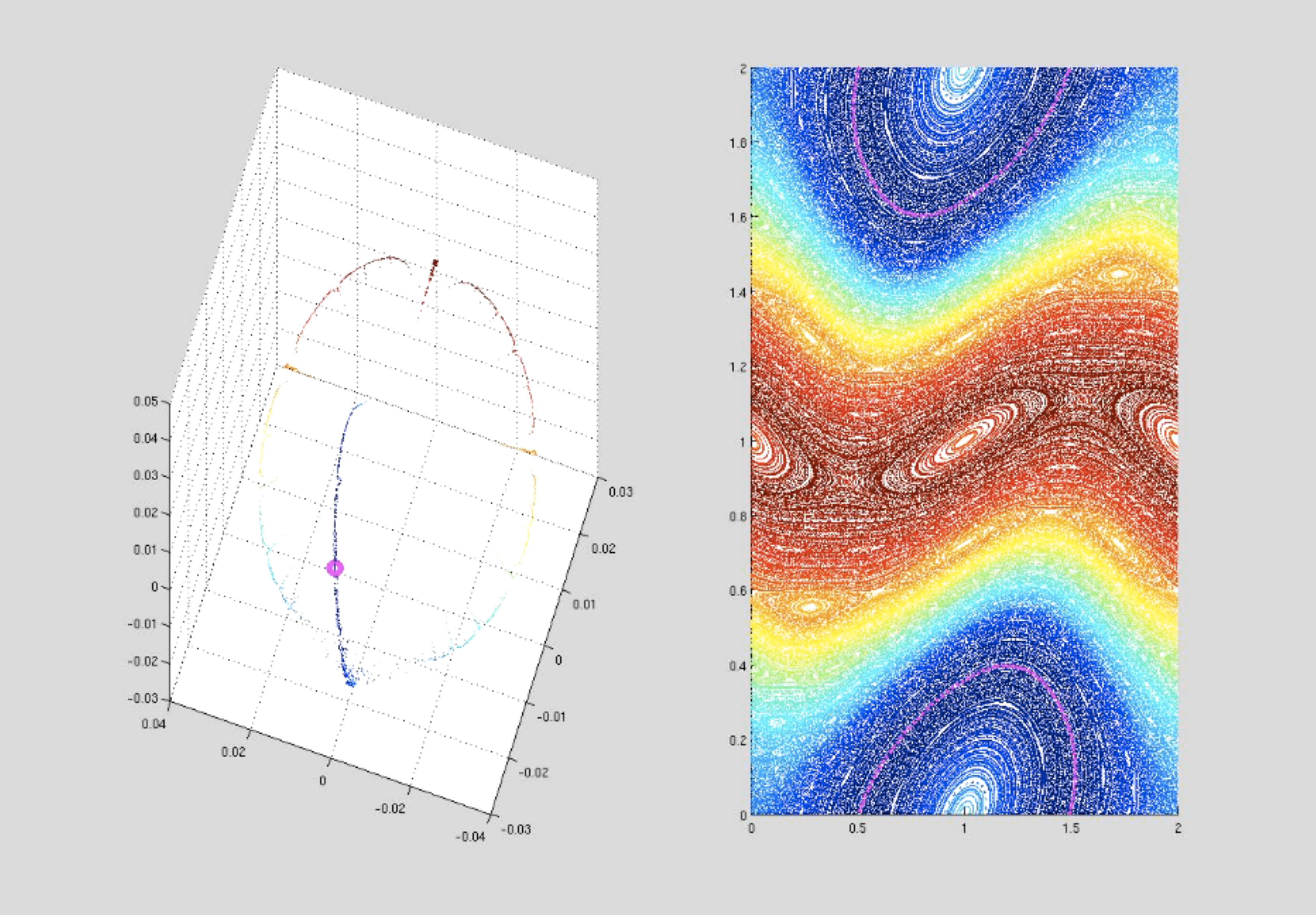}
\caption{Diffusion map of the orbits of the Standard Map for a small $\alpha$. The color of the embedded point on the left corresponds to the orbit of the same color on the right. A particular embedded point and orbit are highlighted in purple.}
\label{fig: diffmap and standard map}
\end{figure}

\begin{figure}
\center
\includegraphics[width=4.0in]{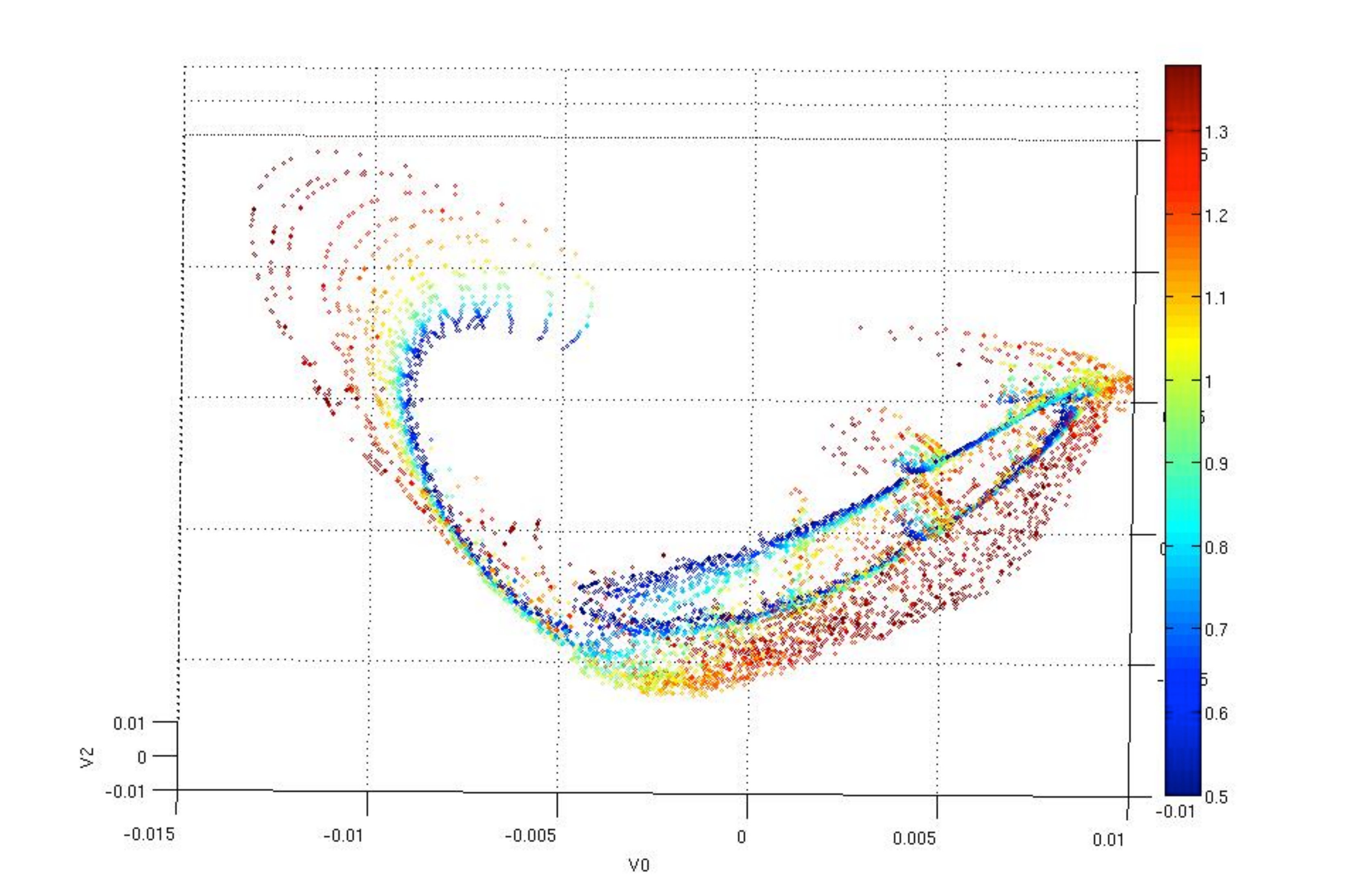}
\caption{Common diffusion embedding of the orbits of the Standard Map across several values of $\alpha$. The color of the embedded point indicates the value of $\alpha$ used in the Standard Map. Notice, in particular, that many of the periodic and quasiperiodic orbits for low values of $\alpha$ that are embedded into the central ring of the embedding turn into chaotic orbits for higher values of $\alpha$. This in turn is realized by the diffusion map as the embedding has less structure.}
\label{fig: common embedding standard map}
\end{figure}

\subsection{Global embeddings}

In this section we seek to illustrate how the global diffusion
distance can be used to recover the parameters governing the global
geometrical behavior of $X_{\alpha}$ as $\alpha$ ranges over $\I$. As an example,
we shall take a torus that is being deformed according to two
parameters:
\begin{enumerate}
\item
The location of the deformation, which in this case is a pinch
(imagine squeezing the torus at a certain spot).
\item
The strength of the deformation, i.e., how hard we pinch the torus.
\end{enumerate}
Let $\I = \{0, 1, \ldots, 30\}$. $X_0$ shall be the standard torus
with no pinch; for $1 \leq \alpha \leq 30$, $X_{\alpha}$ will
have a pinch at a prescribed location on the torus with a prescribed
strength. For an image of the standard torus as well as a pinched
torus, see Figure \ref{fig: torii}.

\begin{figure}
\center
\subfigure[Regular torus]{
\includegraphics[width=2.2in]{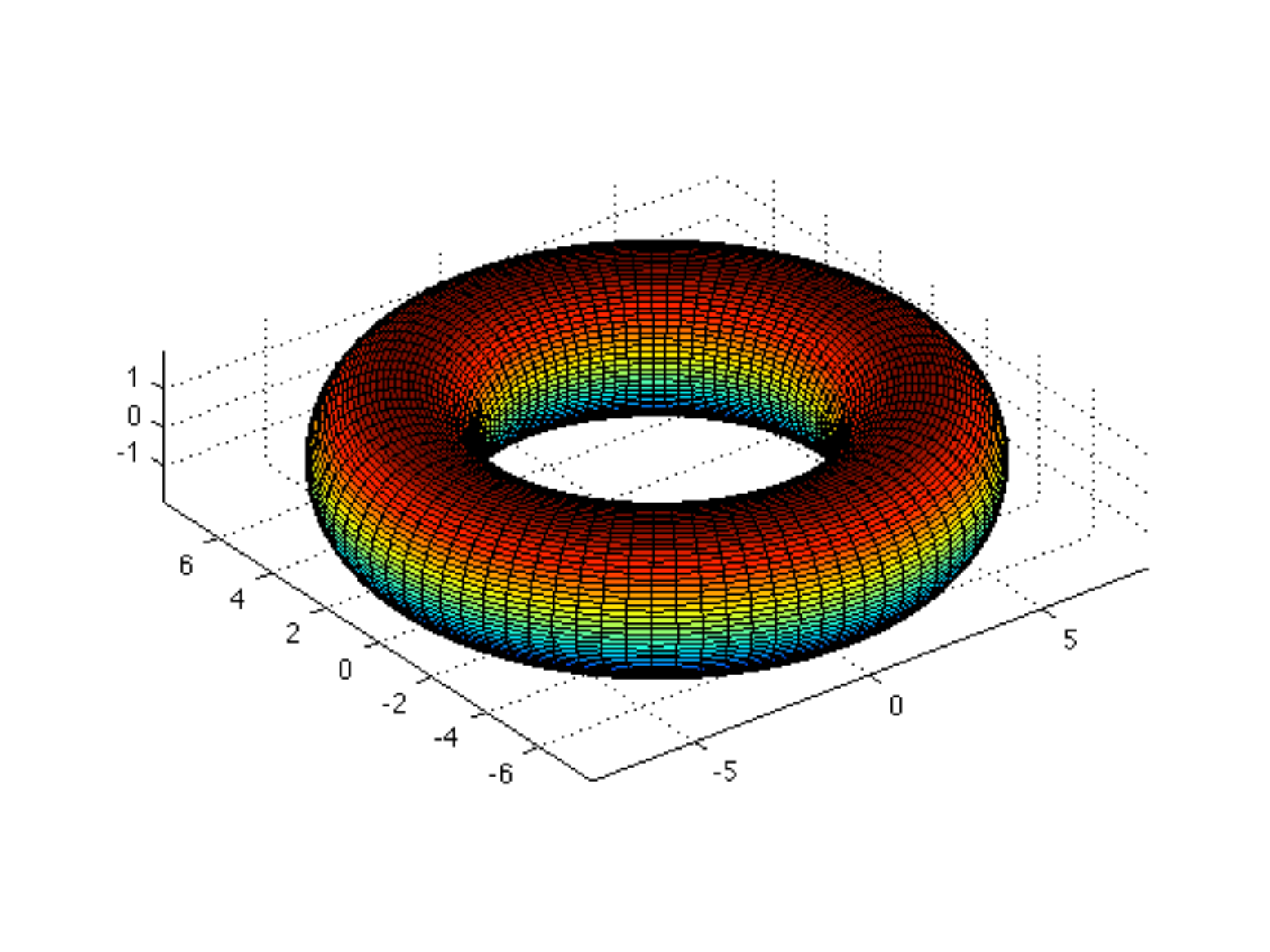}
\label{fig: regular torus model}
}
\subfigure[Pinched torus]{
\includegraphics[width=2.2in]{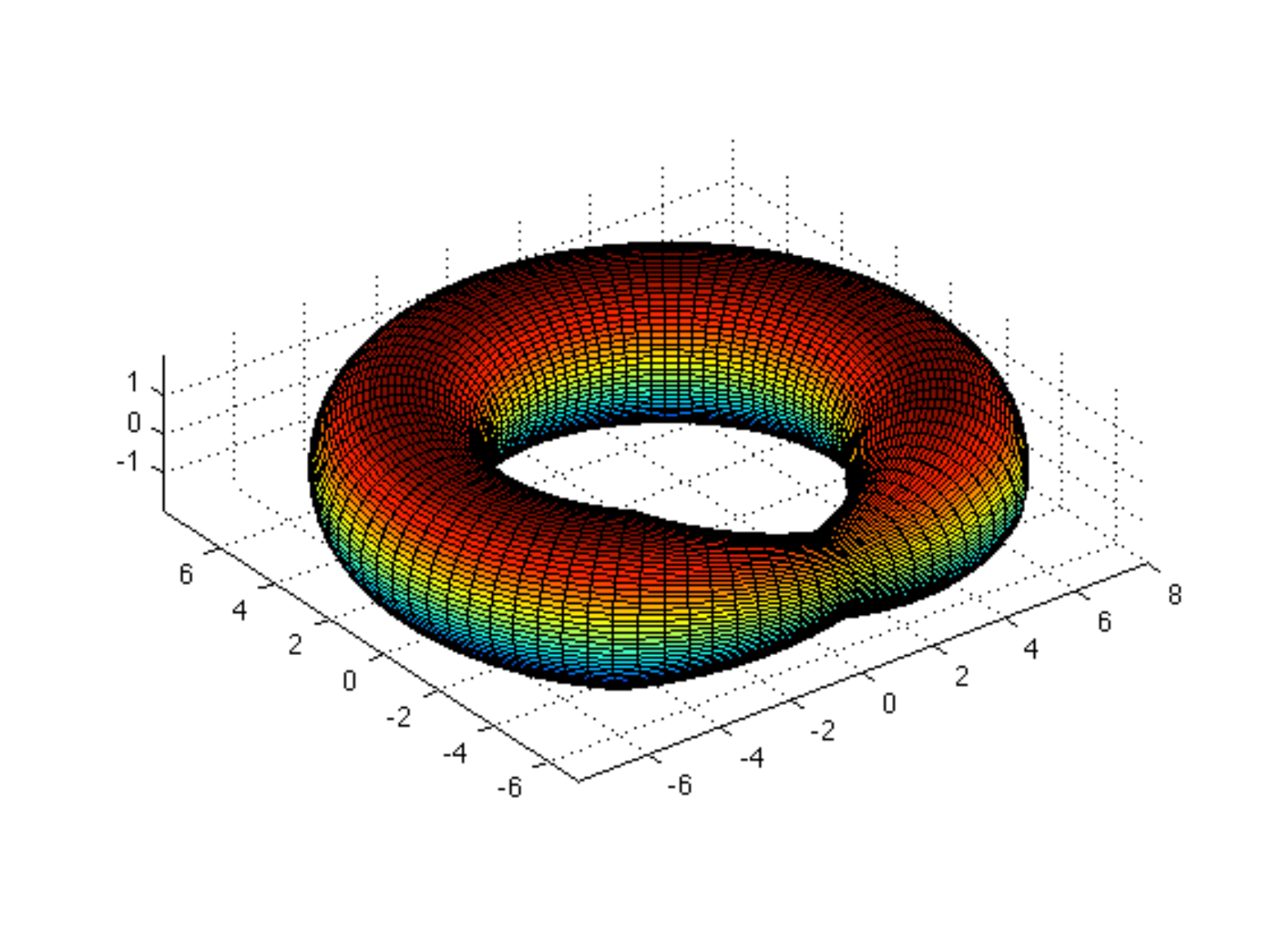}
\label{fig: pinched torus model}
}
\caption{Regular and pinched torii}
\label{fig: torii}
\end{figure}

More specifically, we take $X_0$ to be a torus with a central radius of six and a lateral radius of
two, i.e., $X \triangleq (6S^1) \times (2S^1)$. We
assume that the central circle $6S^1$ and the lateral circle $2S^1$
are oriented, so that each point on the torus has a specific
coordinate location (note that while $X_0 \subset \R^3$, the points of
the torus have a two dimensional coordinate system consisting of two
angles, one for the central circle and one for the lateral circle).

From $X_0$ we build a family of ``pinched'' torii as follows. We pick an
angle on the central circle $6S^1$, say $\theta_0$, and we pinch
the torus at $\theta_0$ so that its lateral radius at this angle is
now $r_0$, where $r_0 < 2$. So that we do not rip the torus, from a
starting angle $\theta_s$, the lateral radius will decrease linearly
from $2$ at $\theta_s$ to $r_0$ at $\theta_0$, and then increase
linearly from $r_0$ at $\theta_0$ back to $2$ at some ending angle
$\theta_e$. The lateral radius of this new torus will be $2$ at all
other angles on the central circle. This is how Figure
\ref{fig: pinched torus model} was constructed. 

We create several pinched torii as follows. We take three different
angles to pinch the torus at: $\theta_0 = \pi/2, \pi$, and
$3\pi/2$. At each of these three angles, we pinch the torus so that
the lateral radius $r_0$ at $\theta_0$ can take one of ten values:
$r_0 = 1, 1.1, 1.2, \ldots, 1.9$. The starting and ending angles for
each pinch are offset from $\theta_0$ by $\pi/4$ radians, so that $\theta_s =
\theta_0 - \pi/4$ and $\theta_e = \theta_0 + \pi/4$. Thus we have
$30$ different pinched torii, which along with the original torus,
gives us a family of $31$ torii.

In order to recover the two global parameters of the family of torii, we use the global
diffusion distance to compute a ``graph of graphs.'' By this
we mean the following: Let $\G \triangleq \{ \Gamma_{\alpha}
\}_{\alpha \in \I}$ be our family of graphs. We can compute a new graph $\Omega_t \triangleq ( \{
\Gamma_{\alpha} \}_{\alpha \in \I}, \overline{k}_t)$, in which $\G$
are the vertices of
$\Omega_t$ and the kernel $\overline{k}_t: \G \times \G \rightarrow \R$ is a function of the
global diffusion distance $\Dt$. One natural way to define $\overline{k}_t$ is
via Gaussian weights:
\begin{equation} \label{eqn: meta kernel}
\overline{k}_t(\Gamma_{\alpha}, \Gamma_{\beta}) \triangleq
e^{-\Dt(\Gamma_{\alpha}, \Gamma_{\beta})^2/\varepsilon^2}, \quad
\text{for all } \alpha, \beta \in \I.
\end{equation}
Note that for each diffusion time $t$, we have a different kernel
$\overline{k}_t$ which results in a different graph $\Omega_t$. Fixing a
specific, but arbitrary diffusion time $t$, one can in turn construct
a new diffusion operator on the graph $\Omega_t$ by using $\overline{k}_t$ as
the underlying kernel. For example, if $\I $ is finite and we let
$\overline{m}_t: \G \rightarrow \R$ be the density of $\overline{k}_t$,
where
\begin{equation*}
\overline{m}_t (\Gamma_{\alpha}) \triangleq \sum_{\beta \in \I} \overline{k}_t
(\Gamma_{\alpha}, \Gamma_{\beta}), \quad \text{for all } \alpha \in \I,
\end{equation*}
then the corresponding symmetric diffusion kernel $\overline{a}_t: \G
\times \G \rightarrow \R$ would be defined as
\begin{equation*}
\overline{a}_t(\Gamma_{\alpha}, \Gamma_{\beta}) \triangleq
\frac{\overline{k}_t(\Gamma_{\alpha},
  \Gamma_{\beta})}{\sqrt{\overline{m}_t(\Gamma_{\alpha})}
  \sqrt{\overline{m}_t(\Gamma_{\beta})}}, \quad \text{for all }
\alpha, \beta \in \I.
\end{equation*}
Since we are assuming $\I$ is finite, one can think of $\overline{a}_t$ as an
$|\I| \times |\I|$ matrix, and one can compute the eigenvectors and
eigenvalues of $\overline{a}_t$.  This gives us a diffusion map of the form
\begin{equation*}
\overline{\Psi}_t^{(s)}: \G \rightarrow \R^d,
\end{equation*}
where $s$ is the diffusion time for the graph of graphs. This diffusion embedding can then be
used to cluster the family of graphs $\G$, treating each graph
$\Gamma_{\alpha} \in \G$ as a single data point.

Our goal is to build a graph of graphs in which each vertex is one of the $31$ torii. To
do so we approximate the global diffusion distance between each pair
of torii by taking $7744$ random samples from $X_0$ (using the uniform
distribution), and then using the same corresponding samples for each
pinched torus. For each torus we used a Gaussian kernel of the form 
\begin{equation*}
k_{\alpha}(x,y) = e^{-\|x-y\|^2/\varepsilon (\alpha)^2}, \quad
\text{for all } \alpha \in \I,
\end{equation*}
where $\varepsilon (\alpha)$ was selected so that the corresponding symmetric
diffusion operator (matrix) $A_{\alpha}$ would have second eigenvalue
$\lambda_{\alpha}^{(2)} = 0.5$. The pairwise global diffusion distance
was further approximated by taking the top ten eigenvalues and
eigenvectors of each of the $31$ diffusion operators, and was then
computed for diffusion time $t=2$ using Theorem \ref{thm: simplified
  global diff dist}. Two remarks: first, the diffusion time
$t=2=1/(1-\lambda_{\alpha}^{(2)})$ corresponds to the approximate time
it would take for the diffusion process to spread through each of the
graphs; secondly, by Theorem \ref{thm: global diff dist sample}, this
approximate global diffusion distance is, with high probability,
nearly equal to the true global diffusion distance between each of the
torii.

After computing the pairwise global diffusion distances, we
constructed the kernel $\overline{k}_t$, for $t=2$, defined in equation \eqref{eqn: meta
  kernel}. We took $\varepsilon$ in this kernel to be the median of
all pairwise global diffusion distances between the $31$ torii. We
then computed the symmetric diffusion operator for this graph of graphs, which
turned out to have second eigenvalue $\lambda^{(2)} \approx 0.48$. We took
the top three eigenvalues and eigenvectors of the diffusion operator,
and used them to compute the diffusion map into $\R^3$ at diffusion
time $s \approx 1/(1-0.48) = 1.92$. 

A plot of this diffusion map is
given in Figure \ref{fig: graph of graphs}. The central, dark blue,
circle corresponds to the regular torus in both images. In Figure
\ref{fig: torus angle}, the other three colors
correspond to the angle at which the torus was pinched. In Figure
\ref{fig: torus str}, the colors correspond to the strength of the
pinch (dark blue - no pinch, dark red - strongest pinch). As one can
see, the diffusion embedding organizes the torii by both the location
of the pinch (i.e. what arc the embedded torus lies on), and the
strength of the pinch (i.e. how far from the regular torus each
pinched torus lies), giving a global view of how the data set changes
over the parameter space.
\begin{figure}
\center
\subfigure[Colored by location of pinch. Each color corresponds to one
of the angles at which the pinch occurs.]{
\label{fig: torus angle}
\includegraphics[width=2.5in]{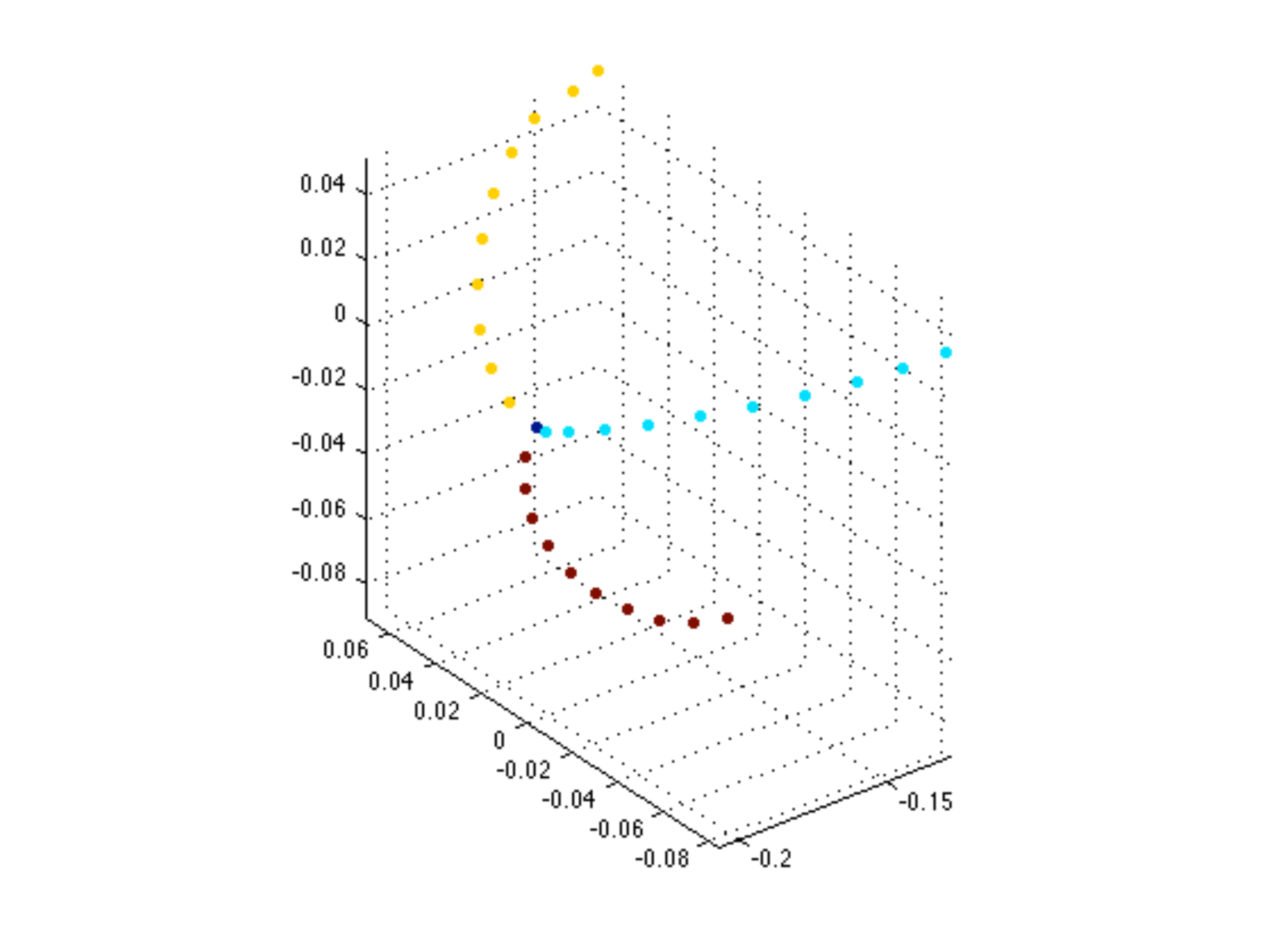}
}
\subfigure[Colored by strength of pinch. Dark blue indicates no pinch,
followed by light blue, green, yellow, orange, and finally dark red
which indicates the strongest pinch.]{
\label{fig: torus str}
\includegraphics[width=2.5in]{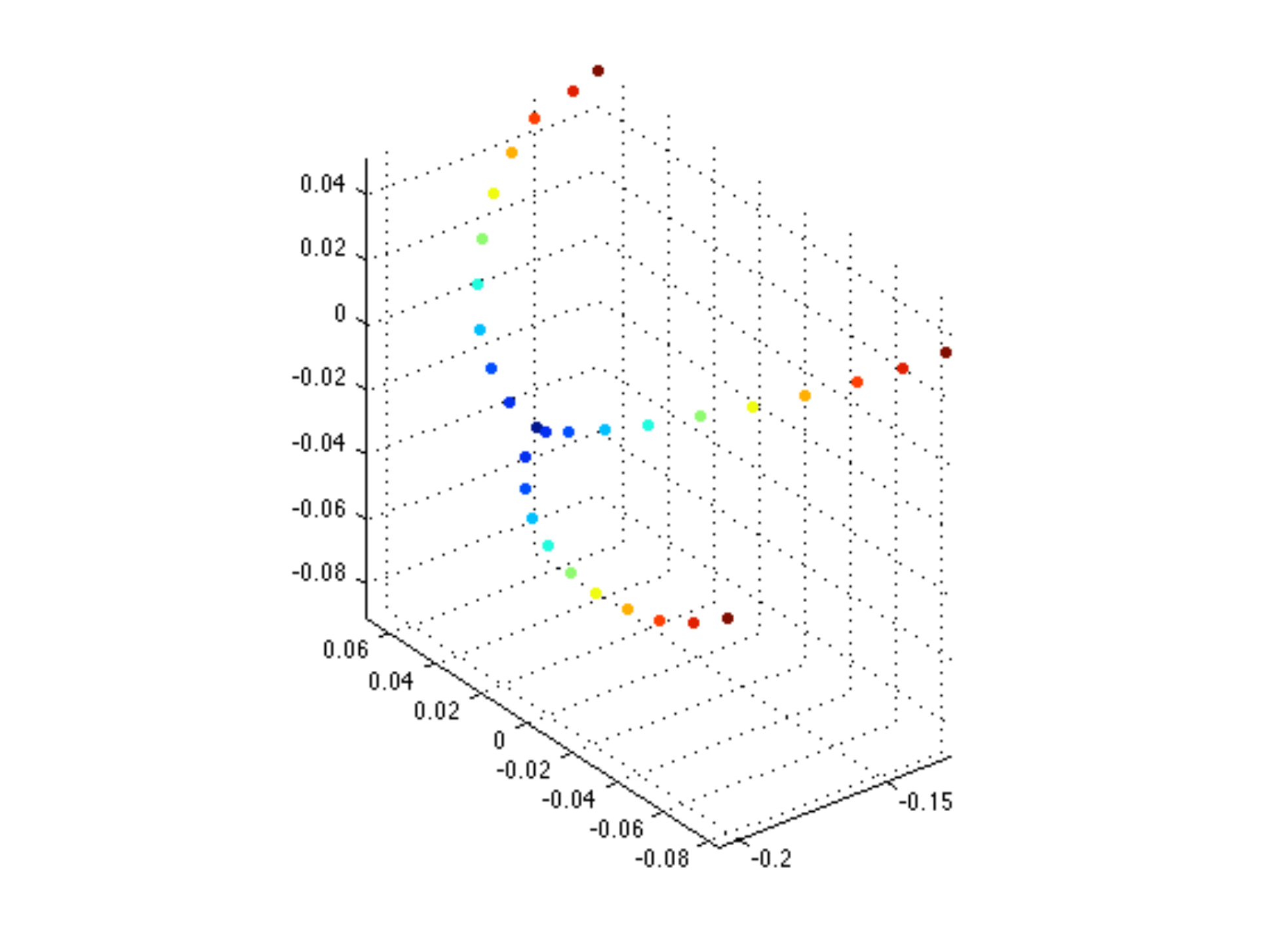}
}
\caption{Diffusion embedding of the $31$ torii. Each data point
  corresponds to a torus. The embedding organizes the torii according
  to the two parameters governing the global geometrical behavior of
  the data over the parameter space.}
\label{fig: graph of graphs}
\end{figure}

\section{Conclusion} \label{sec: conclusion}

In this paper we have generalized the diffusion distance to work on a
changing graph. This new distance, along with the corresponding
diffusion maps, allow one to understand how the intrinsic geometry of
the data set changes over the parameter space. We have also
defined a global diffusion distance between graphs, and used this to
construct meta graphs in which each vertex of the meta graph
corresponds to a graph. Formulas for each of these diffusion distances
in terms of the spectral decompositions of the relevant diffusion
operators have been proven, giving a simple and efficient way to
approximate these diffusion distances. Finally, it was shown that a
random, finite sample of data points from a continuous, changing data
set $X$ is, with high probability, enough to approximate the diffusion
distance and the global diffusion distance to high accuracy.

Future work could include generalizing these notions of diffusion
distance further so that they can apply to sequences of graphs in which there is no
bijective correspondence between the graphs (beyond the simple
generalization of \ref{sec: non-bijective correspondence}). Also, it would be
interesting to investigate how this work fits in with the recent
research on vectorized diffusion operators contained in
\cite{singer:vectorDiffMaps2011, wolf:linearProjDiff2011}.

\section{Acknowledgements}

This research was supported by Air Force Office of Scientific Research
STTR FA9550-10-C-0134 and by Army Research Office MURI
W911NF-09-1-0383. We would also like to thank the anonymous reviewers
for their extremely helpful comments and suggestions, which greatly
improved this paper.

\begin{appendix}

\section{Non-bijective correspondence} \label{sec: non-bijective correspondence}

In this appendix we consider the case in which our changing data set
does not have a single bijective correspondence across the parameter
set $\I$. We make a few small changes to the notation. Continue to let $\I$
denote the parameter space, but let $(\X,\mu)$ denote a ``global''
measure space. Our changing data is given by
$\{X_{\alpha}\}_{\alpha \in \I}$ with data points $x_{\alpha} \in X_{\alpha}$, and satisfies
\begin{equation*}
X_{\alpha} \subseteq \X, \quad \text{for all } \alpha \in \I.
\end{equation*}
We assume that each data set $X_{\alpha}$ is a measurable set under
$\mu$. Suppose, additionally, that there exists a sufficiently large
set $S \subset \X$ such that
\begin{equation*}
S \subset X_{\alpha}, \quad \text{for all } \alpha \in \I.
\end{equation*}
We maintain the remaining notations and assumptions from Section
\ref{sec: graph distance}, and simply update them to apply for each
$X_{\alpha}$. In particular, for each $\alpha \in \I$, we have the
symmetric diffusion kernel $a_{\alpha}: X_{\alpha} \times X_{\alpha}
\rightarrow \R$, with corresponding trace class operator $A_{\alpha}:
L^2(X_{\alpha},\mu) \rightarrow L^2(X_{\alpha},\mu)$. The set of
functions $\{\psi_{\alpha}^{(i)}\}_{i \geq 1} \subset
L^2(X_{\alpha},\mu)$ still denote a set of orthonormal eigenfunctions
for $A_{\alpha}$, with corresponding eigenvalues $\{
\lambda_{\alpha}^{(i)} \}_{i \geq 1}$. The diffusion map is still
given by $\Psi_{\alpha}^{(t)}: X_{\alpha} \rightarrow \ell^2$, with
$\Psi_{\alpha}^{(t)}(x_{\alpha}) = \left(\left(\lambda_{\alpha}^{(i)}\right)^t
\psi_{\alpha}^{(i)}(x_{\alpha}) \right)_{i \geq 1}$.

Under this more general setup, for any $\alpha,\beta \in \I$, the sets
$X_{\alpha} \setminus X_{\beta}$ and $X_{\beta} \setminus X_{\alpha}$
may be nonempty. Thus it is not possible to compare the diffusions on
$\Gamma_{\alpha}$ and $\Gamma_{\beta}$ as they spread through each
graph. On the other hand, since we have a common set $S \subset
X_{\alpha} \cap X_{\beta}$, we can compare the diffusion centered at
$x_{\alpha} \in X_{\alpha}$ with the diffusion centered at $y_{\beta}
\in X_{\beta}$ as they spread through the subgraphs of
$\Gamma_{\alpha}$ and $\Gamma_{\beta}$ with common vertices
$S$. Formally, we define this diffusion distance as:
\begin{equation*}
D^{(t)}(x_{\alpha},y_{\beta};S)^2 \triangleq \int\limits_S \left(
  a_{\alpha}^{(t)}(x_{\alpha},s) - a_{\beta}^{(t)}(y_{\beta},s)
\right)^2 \, d\mu(s), \quad \text{for all } \alpha, \beta \in \I,
\enspace (x_{\alpha},y_{\beta}) \in X_{\alpha} \times X_{\beta}.
\end{equation*}

A result similar to Theorem \ref{thm: common embedding} can be had for
this subgraph diffusion distance. Since the eigenfunctions for
$A_{\alpha}$ will not be orthonormal when restricted to $L^2(S,\mu)$,
one must use an additional orthonormal basis $\{e^{(i)}\}_{i \geq 1}$
for $L^2(S,\mu)$ when rotating the diffusion maps
across $\I$ into a common embedding. In particular, we define a new
family of rotation maps $O_{\alpha,S}: \ell^2 \rightarrow \ell^2$ as:
\begin{equation*}
O_{\alpha,S}v \triangleq \left( \sum_{j \geq 1} v[j] \, \langle e^{(i)},
  \psi_{\alpha}^{(j)} \rangle_{L^2(S,\mu)} \right)_{i \geq 1}.
\end{equation*}
Using these rotation maps, along with the same ideas from Section
\ref{sec: graph distance}, one can show:
\begin{equation*}
D^{(t)}(x_{\alpha},y_{\beta};S) = \left\| O_{\alpha,S}
  \Psi_{\alpha}^{(t)}(x_{\alpha}) - O_{\beta,S}
  \Psi_{\beta}^{(t)}(y_{\beta}) \right\|_{\ell^2}, \quad \text{with
  convergence in } L^2(X_{\alpha} \times X_{\beta}, \mu \otimes \mu).
\end{equation*}

\begin{remark}
Analogously to Remark \ref{rem: gamma choice}, one should be careful
when choosing the basis
$\{e^{(i)}\}_{i \geq 1}$ for $L^2(S,\mu)$. Ideally it will depend on
the desired application, and can thus prioritize certain features in
the data.
\end{remark}

\section{Proof of random sampling theorems} \label{sec:
  proof of finite graph approximation estimate}

In this appendix we prove the random sampling Theorems \ref{thm:
  ptwise diff dist sample} and \ref{thm: global diff dist sample}
from Section \ref{sec: convergence of finite
  approximations}. Throughout the appendix we shall assume that
$(X,\mu)$ and $\{k_{\alpha}\}_{\alpha \in \I}$ satisfy Assumption
\ref{assumption 1.5}. 

The proof shall rely upon a result from
\cite{devito:learningExamples2005} as well as several results on the
asymmetric graph Laplacian $I - P$ that are contained in
\cite{rosasco:learningIntegralOps2010}. All of these results are
easily translated for our family of operators $\{A_{\alpha}\}_{\alpha
  \in \I}$, and we shall simply restate the needed results from
\cite{rosasco:learningIntegralOps2010} in these terms. 

\subsection{Reproducing kernel Hilbert spaces} \label{sec: RKHS}

Critical to our analysis will the be existence of a single reproducing kernel
Hilbert space (RKHS) that contains the set of kernels $\{ a_{\alpha} \}_{\alpha
  \in \I}$, their empirical approximations, and related functions. In
\cite{rosasco:learningIntegralOps2010} such a RKHS is
constructed. Here we recall the definition of a RKHS as well the
aforementioned construction. 

A set $\HH$ is a RKHS \cite{aronszajn:rkhs1950} if it is a Hilbert
space of functions $f: X \rightarrow
\R$ such that for each $x \in X$, there exists a constant $C(x)$ so
that
\begin{equation*}
f(x) \leq C(x) \, \| f \|_{\HH}.
\end{equation*}
The name RKHS comes from the fact that one can show that there is a
unique symmetric, positive definite kernel $h: X \times X \rightarrow
\R$ associated with $\HH$ such that for each $f \in \HH$, 
\begin{equation*}
f(x) = \langle f, h(x,\cdot) \rangle_{\HH}, \quad \text{for all } x
\in X.
\end{equation*}

We utilize a specific RKHS first presented in
\cite{rosasco:learningIntegralOps2010}; the construction is rewritten
here for completeness. Let $l$ be a positive integer, and define the
Sobolev space $\HH^l$ as
\begin{equation*}
\HH^l \triangleq \{ f \in L^2(X, dx) : D^{\gamma} f \in L^2(X, dx) \text{ for
  all } |\gamma| = l \},
\end{equation*} 
where $D^{\gamma} f$ is the weak derivative of $f$ with respect to the
multi-index $\gamma \triangleq (\gamma_1, \ldots, \gamma_d) \in \mathbb{N}^d$,
$|\gamma| \triangleq \gamma_1 + \cdots + \gamma_d$, and $dx$ denotes
the Lebesgue measure. The space $\HH^l$ is a separable Hilbert space with scalar
product
\begin{equation*}
\langle f, g \rangle_{\HH^l} \triangleq \langle f, g \rangle_{L^2(X, dx)} +
\sum_{|\gamma| = l} \langle D^{\gamma} f, D^{\gamma} g \rangle_{L^2(X, dx)}.
\end{equation*}
Also note that the space $C_b^l(X)$ is a Banach space with respect to
the norm
\begin{equation*}
\|f\|_{C_b^l(X)} \triangleq \sup_{x \in X}|f(x)| + \sum_{|\gamma|=l}
\sup_{x \in X} |D^{\gamma}f(x)|.
\end{equation*}
As explained in \cite{rosasco:learningIntegralOps2010}, since $X$ is
bounded, we have $C_b^l(X) \subset \HH^l$ and $\|f\|_{\HH^l} \leq
C(l)\|f\|_{C_b^l(X)}$. Via Corollary $21$ of section $4.6$ from
\cite{burenkov:sobolevSpacesDomains1998}, if $m \in \N$ and $l-m >
d/2$, then we also have:
\begin{equation} \label{eqn: sobolev embedding}
\HH^l \subset C_b^m(X) \quad \text{and} \quad \|f\|_{C_b^m(X)} \leq
C(l,m) \, \|f\|_{\HH^l}.
\end{equation}
Following \cite{rosasco:learningIntegralOps2010}, if one takes $s \triangleq
\lfloor d/2 \rfloor + 1$, then using \eqref{eqn: sobolev embedding}
with $l = s$ and $m = 0$ we see that $\HH^s$ is a RKHS with a
continuous, real valued, bounded kernel $h_s$.

\subsection{Additional operators} \label{sec:
  additional operators and previous results}

In this section we define several operators that will bridge the gap
between the matrix $\A_{\alpha}$ and the operator $A_{\alpha}$. All of
these definitions are based on those from
\cite{rosasco:learningIntegralOps2010} for the asymmetrical diffusion operators
(i.e. $P$). To start, define the empirical density maps
$m_{\alpha,n}: X \rightarrow \R$ in terms of the samples $X_n = \{x^{(1)},
\ldots, x^{(n)}\}$ as  
\begin{equation*}
m_{\alpha,n}(x) \triangleq \frac{1}{n} \sum_{i=1}^n k_\alpha(x, x^{(i)}), \quad \text{for all
} \alpha \in \I, \enspace x \in X.
\end{equation*}
Note that $m_{\alpha,n}(x^{(i)}) = \D_\alpha[i,i]$. We also define the
empirical kernels $a_{\alpha,n}: X \times X \rightarrow \R$ as
\begin{equation*}
a_{\alpha,n} (x,y) \triangleq \frac{k_\alpha(x,y)}{\sqrt{m_{\alpha,n}(x)} \sqrt{m_{\alpha,n}(y)}},
\quad \text{for all } \alpha \in \I, \enspace x,y \in X.
\end{equation*}
We then have the following lemma from
\cite{rosasco:learningIntegralOps2010}, adapted for symmetric
diffusion operators.
\begin{lemma}[Lemma $16$ from
  \cite{rosasco:learningIntegralOps2010}] \label{lem: fcn bounds}
Assume that $(X,\mu)$ and $\{k_{\alpha}\}_{\alpha \in \I}$ satisfy the
conditions of Assumption \ref{assumption 1.5}. Then, for all $\alpha
\in \I$ and for all $x \in X$,
\begin{equation*}
k_{\alpha}(x,\cdot), m_{\alpha}, m_{\alpha,n}, \frac{1}{m_{\alpha}},
\frac{1}{m_{\alpha,n}} \in C_b^{d+1}(X) \subset \HH^{d+1} \subset
\HH^s,
\end{equation*}
\begin{equation*}
\left\|k_{\alpha}(x,\cdot)\right\|_{C_b^{d+1}(X)}, \left\|m_{\alpha}\right\|_{C_b^{d+1}(X)},
\left\|m_{\alpha,n}\right\|_{C_b^{d+1}(X)}, \left\|\frac{1}{m_{\alpha}}\right\|_{C_b^{d+1}(X)},
\left\|\frac{1}{m_{\alpha,n}}\right\|_{C_b^{d+1}(X)} \leq C(\alpha,d),
\end{equation*}
\begin{equation*}
a_{\alpha}(x,\cdot), a_{\alpha,n}(x,\cdot) \in C_b^{d+1}(X) \subset
\HH^{d+1} \subset \HH^s, 
\end{equation*}
\begin{equation*}
\left\|a_{\alpha}(x,\cdot)\right\|_{\HH^s},
\left\|a_{\alpha,n}(x,\cdot)\right\|_{\HH^s} \leq C(\alpha,d).
\end{equation*}
\end{lemma}

Lemma \ref{lem: fcn bounds} allows one to define the
operators $A_{\alpha,\HH^s}: \HH^s \rightarrow \HH^s$ and $A_{\alpha,n}: \HH^s
\rightarrow \HH^s$,
\begin{align*}
(A_{\alpha, \HH^s}f)(x) &\triangleq \int\limits_X a_\alpha(x,y) \, \langle f, h_s(y,\cdot)
\rangle_{\HH^s} \, d\mu(y),  &\text{for all } \alpha \in \I,
\enspace f \in \HH^s, \\
(A_{\alpha,n}f)(x) &\triangleq \frac{1}{n} \sum_{i=1}^n
a_{\alpha,n}(x, x^{(i)}) \, \langle f,
h_s(x^{(i)},\cdot) \rangle_{\HH^s}, &\text{for all } \alpha \in \I,
\enspace f \in \HH^s.
\end{align*}
We also define similar operators $T_{\HH^s}:\HH^s \rightarrow \HH^s$
and $T_n:\HH^s \rightarrow \HH^s$, but in terms of the reproducing
kernel $h_s$.
\begin{align*}
(T_{\HH^s}f)(x) &\triangleq \int\limits_X h_s(x, y) \, \langle f, h_s(y,\cdot)
\rangle_{\HH^s} \, d\mu (y), &\text{for all } f \in \HH^s, \\
(T_n f)(x) &\triangleq \frac{1}{n} \sum_{i=1}^n h_s(x, x^{(i)}) \, \langle f, h_s(x^{(i)},\cdot)
\rangle_{\HH^s}, &\text{for all } f \in \HH^s.
\end{align*}
The above operators, as well as $A_{\alpha}$ and $\A_{\alpha}$, can be
decomposed in terms of the appropriate restriction and extension
operators. We begin with the two restriction operators,
$R_{\HH^s}:\HH^s \rightarrow L^2(X,\mu)$ and $R_n:\HH^s \rightarrow \R^n$.
\begin{align*}
(R_{\HH^s} f)(x) &\triangleq \langle f,h_s(x,\cdot) \rangle_{\HH^s},
\qquad \quad \enspace \text{for } \mu \text{ a.e. } x \in X, \text{
  for all } f \in \HH^s,\\
R_n f &\triangleq (f(x^{(1)}), \ldots, f(x^{(n)})), \quad \text{for all } f \in \HH^s.
\end{align*}
For each $\alpha \in \I$ we also have two extension
operators, $E_{\alpha, \HH^s}: L^2(X, \mu) \rightarrow \HH^s$ and
$E_{\alpha,n}: \R^n \rightarrow \HH^s$, where
\begin{align*}
(E_{\alpha, \HH^s}f)(x) &\triangleq \int\limits_X a_\alpha (x, y) f(y) \, d\mu(y), \quad
\text{for all }  x \in X, \enspace f \in L^2(X, \mu), \\
(E_{\alpha,n}v)(x) &\triangleq \frac{1}{n} \sum_{i=1}^n v[i] \,
a_{\alpha,n}(x, x^{(i)}), \quad \text{for all } x \in X, \enspace  v \in \R^n.
\end{align*}
Using these operators, one can easily show the following identities:
\begin{align}
A_\alpha = R_{\HH^s} E_{\alpha, \HH^s} \quad &\text{and} \quad A_{\alpha, \HH^s} =
E_{\alpha, \HH^s} R_{\HH^s}, \nonumber \\
\A_\alpha = R_n E_{\alpha,n} \quad &\text{and} \quad A_{\alpha,n} =
E_{\alpha,n} R_n, \label{eqn: operator identities} \\
T_{\HH^s} = R_{\HH^s}^*R_{\HH^s} \quad &\text{and} \quad T_n =
R_n^*R_n. \nonumber
\end{align}

\subsection{Similarity between empirical and continuous operators}

Here we collect remaining results that we shall need that involve the
similarity between the empirical and continuous versions of the
previously defined operators and functions. All of these results can
be found in \cite{devito:learningExamples2005,
  rosasco:learningIntegralOps2010}. 
\begin{theorem}[\cite{devito:learningExamples2005}, also Theorem $7$
  from \cite{rosasco:learningIntegralOps2010}]\label{thm: rbd7}
Suppose that $(X,\mu)$ and $\{k_{\alpha}\}_{\alpha \in \I}$ satisfy
the conditions of Assumption \ref{assumption 1.5}. Let $n \in \N$ and
sample $X_n = \{x^{(1)}, \ldots, x^{(n)}\} \subset X$ i.i.d. according to
$\mu$; also let $\tau > 0$. Then the operators $T_{\HH^s}$ and
$T_n$ are Hilbert-Schmidt, and with probability $1-2e^{-\tau}$,
\begin{equation*}
\left\| T_{\HH^s} - T_n \right\|_{HS} \leq C(d) \frac{\sqrt{\tau}}{\sqrt{n}}.
\end{equation*}
\end{theorem}

\begin{theorem}[Theorem $15$ from
  \cite{rosasco:learningIntegralOps2010}]\label{thm: rbd15}
Suppose that $(X,\mu)$ and $\{k_{\alpha}\}_{\alpha \in \I}$ satisfy
the conditions of Assumption \ref{assumption 1.5}. Let $n \in \N$ and
sample $X_n = \{x^{(1)}, \ldots, x^{(n)}\} \subset X$ i.i.d. according to
$\mu$; also let $\tau > 0$ and $\alpha \in \I$. Then the operators
$A_{\alpha,\HH^s}$ and $A_{\alpha,n}$ are Hilbert-Schmidt, and with
probability $1-2e^{-\tau}$,
\begin{equation*}
\left\| A_{\alpha,\HH^s} - A_{\alpha,n} \right\|_{HS} \leq C(\alpha,d)
\frac{\sqrt{\tau}}{\sqrt{n}}. 
\end{equation*}
\end{theorem}

\begin{lemma}[Lemma $18$ from
  \cite{rosasco:learningIntegralOps2010}]\label{lem: rbd18}
Suppose that $(X,\mu)$ and $\{k_{\alpha}\}_{\alpha \in \I}$ satisfy
the conditions of Assumption \ref{assumption 1.5}. Let $n \in \N$ and
sample $X_n = \{x^{(1)}, \ldots, x^{(n)}\} \subset X$ i.i.d. according to
$\mu$; also let $\tau > 0$ and $\alpha \in \I$. Then, with
probability $1-2e^{-\tau}$,
\begin{equation*}
\left\| m_{\alpha} - m_{\alpha,n} \right\|_{\HH^{d+1}} \leq
C(\alpha,d) \frac{\sqrt{\tau}}{\sqrt{n}}.
\end{equation*}
\end{lemma}

\subsection{Proof of Theorem \ref{thm: ptwise diff dist sample}}

In this section we prove Theorem \ref{thm: ptwise diff dist sample},
which we restate here.

\begin{theorem}[Theorem \ref{thm: ptwise diff dist sample}]
Suppose that $(X,\mu)$ and $\{k_{\alpha}\}_{\alpha \in \I}$ satisfy the
conditions of Assumption \ref{assumption 1.5}. Let $n \in \N$ and
sample $X_n = \{x^{(1)}, \ldots, x^{(n)}\} \subset X$ i.i.d. according to
$\mu$; also let $t \in \N$, $\tau > 0$, and $\alpha,\beta \in \I$. Then, with
probability $1-2e^{-\tau}$, 
\begin{equation*}
\left| D^{(t)}(x_{\alpha}^{(i)},x_{\beta}^{(j)}) -
  D_n^{(t)}(x_{\alpha}^{(i)},x_{\beta}^{(j)}) \right| \leq C(\alpha,\beta,d,t)
\frac{\sqrt{\tau}}{\sqrt{n}}, \quad \text{for all } i,j = 1, \ldots, n.
\end{equation*}
\end{theorem}

\begin{proof}[Proof of Theorem \ref{thm: ptwise diff dist sample}]
First an additional piece of notation. Recall
the $d$-dimensional index $\gamma = (\gamma_1, \ldots,
\gamma_d)$. Let $\partial_x^{\gamma}a_{\alpha}$ denote the
$\gamma^{\text{th}}$ partial derivative of $a_{\alpha}$ with respect
to the variable $x$.

We begin with the empirical diffusion distance. Recall that
$D_n^{(t)}(x_{\alpha}^{(i)},x_{\beta}^{(j)})^2 = n^2\|\A_{\alpha}^t[i,\cdot] -
\A_{\beta}^t[j,\cdot]\|_{\R^n}^2$. For each $i = 1, \ldots, n$, define
the vector $e^{(i)} \in \R^n$ as
\begin{equation*}
e^{(i)}[j] \triangleq \left\{
\begin{array}{ll}
1, & \text{if } j = i, \\
0, & \text{if } j \neq i,
\end{array}
\right. \quad \text{for all } j = 1, \ldots, n.
\end{equation*}
We then have
\begin{align}
D_n^{(t)}(x_{\alpha}^{(i)},y_{\beta}^{(j)})^2 &=
n^2\left\|\A_{\alpha}^te^{(i)} - \A_{\beta}^te^{(j)}\right\|_{\R^n}^2
\nonumber \\
&= n^2\langle \A_{\alpha}^te^{(i)},\A_{\alpha}^te^{(i)} \rangle_{\R^n}
+ n^2\langle \A_{\beta}^te^{(j)},\A_{\beta}^te^{(j)} \rangle_{\R^n} -
2n^2\langle \A_{\alpha}^te^{(i)},\A_{\beta}^te^{(j)}
\rangle_{\R^n}. \label{eqn: discrete inner products}
\end{align}

A similar expression can be had for the continuous diffusion
distance. By Assumption \ref{assumption 1.5}, $k_{\alpha} \in
C_b^{d+1}(X \times X)$ and $k_{\alpha} \geq C_1(\alpha)$. These imply
that $a_{\alpha} \in C_b^{d+1}(X \times X)$. We can then apply
Mercer's Theorem to get that
\begin{equation} \label{eqn: another mercer application}
a_{\alpha}^{(t)}(x,y) = \sum_{\ell \geq 1}
\left(\lambda_{\alpha}^{(\ell)}\right)^t \psi_{\alpha}^{(\ell)}(x) \,
\psi_{\alpha}^{(\ell)}(y), \quad \text{for all } (x,y) \in X \times X,
\end{equation}
with absolute convergence and uniform convergence on compact subsets
of $X$. In fact, since $A_{\alpha}$ is also trace class, we can get
uniform convergence on all of $X$. Indeed,
\begin{equation*}
\mathrm{Tr}(A_{\alpha}) = \sum_{\ell \geq 1} \lambda_{\alpha}^{(\ell)}
< \infty.
\end{equation*}
Therefore, for all $\varepsilon > 0$ and for each $\alpha \in \I$,
there exists $N(\varepsilon, \alpha) \in \N$ such that
\begin{equation*}
\sum_{\ell > N(\varepsilon, \alpha)}
\left(\lambda_{\alpha}^{(\ell)}\right)^t < \varepsilon.
\end{equation*}
Furthermore, since $a_{\alpha}$ is bounded, $\psi_{\alpha}^{(\ell)}
\leq C_2(\alpha)$ for all $\ell \geq 1$. Therefore,
\begin{equation} \label{eqn: uniform conv on all of X}
\sum_{\ell > N(\varepsilon,\alpha)}
\left(\lambda_{\alpha}^{(\ell)}\right)^t \psi_{\alpha}^{(\ell)}(x) \,
\psi_{\alpha}^{(\ell)}(y) \leq C_2(\alpha) \sum_{\ell > N(\varepsilon,\alpha)}
\left(\lambda_{\alpha}^{(\ell)}\right)^t < C_2(\alpha) \, \varepsilon,
\quad \text{for all } (x,y) \in X \times X.
\end{equation}

Now define a family of functions $\varphi_{\alpha}^{(N,i)} \in
L^2(X,\mu)$ for all $N \in \N$ and $i \in \{1,\ldots,n\}$,
\begin{equation*}
\varphi_{\alpha}^{(N,i)}(x) \triangleq \sum_{\ell=1}^N
\psi_{\alpha}^{(\ell)}(x^{(i)}) \, \psi_{\alpha}^{(\ell)}(x).
\end{equation*}
We claim that
\begin{equation} \label{eqn: cont row approx}
\left| a_{\alpha}^{(t)}(x^{(i)},x) - A_{\alpha}^t
  \varphi_{\alpha}^{(N(\varepsilon,\alpha),i)}(x) \right| <
C_2(\alpha) \, \varepsilon, \quad \text{for all } x \in X.
\end{equation}
Indeed,
\begin{align}
A_{\alpha}^t \varphi_{\alpha}^{(N,i)}(x) &= \int\limits_X
a_{\alpha}^{(t)}(x,y) \, \varphi_{\alpha}^{(N,i)}(y) \, d\mu(y),
\nonumber \\
&= \int\limits_X \left(\sum_{m \geq 1}
  \left(\lambda_{\alpha}^{(m)}\right)^t \psi_{\alpha}^{(m)}(x) \,
  \psi_{\alpha}^{(m)}(y) \right) \left( \sum_{\ell=1}^N
  \psi_{\alpha}^{(\ell)}(x^{(i)}) \, \psi_{\alpha}^{(\ell)}(y) \right)
\, d\mu(y), \nonumber \\
&= \sum_{m \geq 1} \sum_{\ell=1}^N
\left(\lambda_{\alpha}^{(m)}\right)^t \psi_{\alpha}^{(m)}(x) \,
\psi_{\alpha}^{(\ell)}(x^{(i)}) \int\limits_X
\psi_{\alpha}^{(m)}(y) \, \psi_{\alpha}^{(\ell)}(y) \, d\mu(y),
\nonumber \\
&= \sum_{\ell=1}^N \left(\lambda_{\alpha}^{(\ell)}\right)^t
\psi_{\alpha}^{(\ell)}(x^{(i)}) \,
\psi_{\alpha}^{(\ell)}(x). \label{eqn: Aphi form}
\end{align}
Therefore, using \eqref{eqn: another mercer application}, \eqref{eqn:
  Aphi form}, and \eqref{eqn: uniform conv on all of X}, we obtain
\begin{equation*}
\left| a_{\alpha}^{(t)}(x^{(i)},x) - A_{\alpha}^t
  \varphi_{\alpha}^{(N(\varepsilon,\alpha),i)}(x) \right| = \left|
  \sum_{\ell > N(\varepsilon,\alpha)}
  \left(\lambda_{\alpha}^{(\ell)}\right)^t
  \psi_{\alpha}^{(\ell)}(x^{(i)}) \, \psi_{\alpha}^{(\ell)}(x) \right|
< C_2(\alpha) \, \varepsilon,
\end{equation*}
and so \eqref{eqn: cont row approx} holds.

Using \eqref{eqn: cont row approx}, it not hard to see that
\begin{equation*}
\left| D^{(t)}(x_{\alpha}^{(i)},y_{\beta}^{(j)}) -
  \left\|A_{\alpha}^t\varphi_{\alpha}^{(N(\varepsilon,\alpha),i)} -
    A_{\beta}^t\varphi_{\beta}^{(N\varepsilon,\beta),i)}
  \right\|_{L^2(X,\mu)} \right| \leq C_3(\alpha,\beta) \, \varepsilon.
\end{equation*}
Thus it is enough to consider
$\|A_{\alpha}^t\varphi_{\alpha}^{(N(\varepsilon,\alpha),i)} -
A_{\beta}^t\varphi_{\beta}^{(N(\varepsilon,\beta),j)}\|_{L^2(X,\mu)}$. Expanding
the square of this quantity one has
\begin{align}
\left\|A_{\alpha}^t\varphi_{\alpha}^{(N(\varepsilon,\alpha),i)} -
A_{\beta}^t\varphi_{\beta}^{(N(\varepsilon,\beta),j)} \right\|_{L^2(X,\mu)}^2
&= \langle A_{\alpha}^t\varphi_{\alpha}^{(N(\varepsilon,\alpha),i)},
A_{\alpha}^t\varphi_{\alpha}^{(N(\varepsilon,\alpha),i)} \rangle_{L^2(X,\mu)}
\nonumber \\
&+ \langle A_{\beta}^t\varphi_{\beta}^{(N(\varepsilon,\beta),j)},
A_{\beta}^t\varphi_{\beta}^{(N(\varepsilon,\beta),j)} \rangle_{L^2(X,\mu)}
-2\langle A_{\alpha}^t\varphi_{\alpha}^{(N(\varepsilon,\alpha),i)},
A_{\beta}^t\varphi_{\beta}^{(N(\varepsilon,\beta),j)}
\rangle_{L^2(X,\mu)}. \label{eqn: cont inner products}
\end{align}
The three inner products in \eqref{eqn: discrete inner products}
correspond to the three inner products in \eqref{eqn: cont inner
  products}. We aim to show that each pair is nearly identical. We
will do so explicitly for the pair $n^2\langle
\A_{\alpha}^te^{(i)},\A_{\beta}^te^{(j)} \rangle_{\R^n}$ and $\langle
A_{\alpha}^t\varphi_{\alpha}^{(N(\varepsilon,\alpha),i)},A_{\beta}^t\varphi_{\beta}^{(N(\varepsilon,\beta),j)}
\rangle_{L^2(X,\mu)}$; the other two pairs are simply special cases of
this one. We begin with the discrete inner product, for which we have
the following with probability $1-2e^{-\tau}$:
\begin{align}
n^2\langle \A_{\alpha}^t e^{(i)}, \A_{\beta}^t e^{(j)} \rangle_{\R^n}
&= n^2 \langle (R_nE_{\alpha,n})^t e^{(i)}, (R_nE_{\beta,n})^t e^{(j)}
\rangle_{\R^n} \label{eqn: discrete 1} \\
&= n^2 \langle (E_{\alpha,n}R_n)^{t-1}E_{\alpha,n} e^{(i)}, R_n^*R_n
(E_{\beta,n}R_n)^{t-1} E_{\beta,n} e^{(j)} \rangle_{\HH^s} \nonumber \\
&= \langle A_{\alpha,n}^{t-1} a_{\alpha,n}(x^{(i)},\cdot), T_n
A_{\beta,n}^{t-1} a_{\beta,n}(x^{(j)},\cdot)
\rangle_{\HH^s} \label{eqn: discrete 2} \\
&\leq \langle A_{\alpha, \HH^s}^{t-1}a_{\alpha,n}(x^{(i)},\cdot),
T_{\HH^s} A_{\beta, \HH^s}^{t-1} a_{\beta,n}(x^{(j)},\cdot)
\rangle_{\HH^s} +
C(\alpha,\beta,d,t)\frac{\sqrt{\tau}}{\sqrt{n}}, \label{eqn: discrete 3}
\end{align}
where \eqref{eqn: discrete 1} follows from \eqref{eqn: operator
  identities}, \eqref{eqn: discrete 2} follows from \eqref{eqn:
  operator identities} and the definitions of $E_{\alpha,n}$ and
$e^{(i)}$, and \eqref{eqn: discrete 3} follows from Lemma \ref{lem:
  fcn bounds}, Theorem \ref{thm: rbd7}, Theorem \ref{thm: rbd15}, and the
Cauchy-Schwarz inequality. Since the argument is symmetric, we have,
with probability $1-2e^{-\tau}$,
\begin{equation} \label{eqn: discrete 4}
\left| n^2\langle \A_{\alpha}^te^{(i)},\A_{\beta}^te^{(j)}
  \rangle_{\R^n} - \langle
  A_{\alpha,\HH^s}^{t-1}a_{\alpha,n}(x^{(i)},\cdot),
  T_{\HH^s}A_{\beta,\HH^s}^{t-1}a_{\beta,n}(x^{(j)},\cdot)
  \rangle_{\HH^s} \right| \leq C(\alpha,\beta,d,t)\frac{\sqrt{\tau}}{\sqrt{n}}.
\end{equation}
Now return to the continuous inner product. With probability
$1-2e^{-\tau}$, we have:
\begin{align}
\langle A_{\alpha}^t\varphi_{\alpha}^{(N(\varepsilon,\alpha),i)},
A_{\beta}^t\varphi_{\beta}^{(N(\varepsilon,\beta),j)} \rangle_{L^2(X,\mu)} &=
\langle (R_{\HH^s}E_{\alpha,\HH^s})^t\varphi_{\alpha}^{(N(\varepsilon,\alpha),i)},
(R_{\HH^s}E_{\beta,\HH^s})^t\varphi_{\beta}^{(N(\varepsilon,\beta),j)}
\rangle_{L^2(X,\mu)} \label{eqn: cont 1} \\
&= \langle (E_{\alpha,\HH^s}R_{\HH^s})^{t-1}
E_{\alpha,\HH^s}\varphi_{\alpha}^{(N(\varepsilon,\alpha),i)},
R_{\HH^s}^*R_{\HH^s} (E_{\beta,\HH^s}R_{\HH^s})^{t-1}
E_{\beta,\HH^s}\varphi_{\beta}^{(N(\varepsilon,\beta),j)} \rangle_{\HH^s}
\nonumber \\
&= \langle A_{\alpha,\HH^s}^{t-1}E_{\alpha,\HH^s}
\varphi_{\alpha}^{(N(\varepsilon,\alpha),i)}, T_{\HH^s}A_{\beta,\HH^s}^{t-1}
E_{\beta,\HH^s} \varphi_{\beta}^{(N(\varepsilon,\beta),j)}
\rangle_{\HH^S}, \label{eqn: cont 2}
\end{align}
where \eqref{eqn: cont 1} and \eqref{eqn: cont 2} both follow from
\eqref{eqn: operator identities}. 

Examining \eqref{eqn: discrete 4}
and \eqref{eqn: cont 2}, it is clear that to complete the proof we
must bound the quantity $\|a_{\alpha,n}(x^{(i)},\cdot) -
E_{\alpha,\HH^s}\varphi_{\alpha}^{(N(\varepsilon,\alpha),i)} \|_{\HH^s}$. We
break it into two parts:
\begin{equation} \label{eqn: mixed 1}
\left\|a_{\alpha,n}(x^{(i)},\cdot) -
  E_{\alpha,\HH^s}\varphi_{\alpha}^{(N(\varepsilon,\alpha),i)}
\right\|_{\HH^s} \leq \left\|a_{\alpha,n}(x^{(i)},\cdot) -
  a_{\alpha}(x^{(i)},\cdot) \right\|_{\HH^s} + \left\|
  a_{\alpha}(x^{(i)},\cdot) - E_{\alpha,\HH^s}
  \varphi_{\alpha}^{(N(\varepsilon,\alpha),i)} \right\|_{\HH^s}.
\end{equation}
For the first part, some simple manipulations give:
\begin{equation*}
a_{\alpha,n}(x^{(i)},x) - a_{\alpha}(x^{(i)},x) = f_{\alpha,n}^{(i)}(x) + g_{\alpha,n}^{(i)}(x),
\end{equation*}
where
\begin{equation*}
f_{\alpha,n}^{(i)}(x) = \frac{k_{\alpha}(x^{(i)},x) \, (\sqrt{m_{\alpha}(x)} -
  \sqrt{m_{\alpha,n}(x)})}{\sqrt{m_{\alpha,n}(x^{(i)})}
  \sqrt{m_{\alpha,n}(x)} \sqrt{m_{\alpha}(x)}}
\end{equation*}
and
\begin{equation*}
g_{\alpha,n}^{(i)}(x) = \frac{k_{\alpha}(x^{(i)},x) \, (\sqrt{m_{\alpha,n}(x^{(i)})} -
  \sqrt{m_{\alpha}(x^{(i)})})}{\sqrt{m_{\alpha,n}(x^{(i)})}
  \sqrt{m_{\alpha}(x^{(i)})} \sqrt{m_{\alpha}(x)}}.
\end{equation*}
For the first of these two functions, using Lemma \ref{lem: fcn
  bounds} and Lemma \ref{lem: rbd18} it is easy to see that $\|f_{\alpha,n}^{(i)}\|_{\HH^s}
\leq C(\alpha,d)\frac{\sqrt{\tau}}{\sqrt{n}}$ with probability
$1-2e^{-\tau}$. For $g_{\alpha,n}^{(i)}$, note that
\begin{align}
\left| m_{\alpha,n}(x^{(i)}) - m_{\alpha}(x^{(i)}) \right| &\leq
\sup_{x \in X}\left| m_{\alpha,n}(x) - m_{\alpha}(x) \right| \nonumber
\\
&= \left\|m_{\alpha,n} - m_{\alpha} \right\|_{C_b^0(X)} \nonumber \\
&\leq C(d) \left\|m_{\alpha,n} - m_{\alpha} \right\|_{\HH^{d+1}}
\nonumber \\
&\leq C(\alpha,d) \frac{\sqrt{\tau}}{\sqrt{n}}, \label{eqn: density 1}
\end{align}
where in \eqref{eqn: density 1} we once again used Lemma \ref{lem:
  rbd18}. Thus $\|g_{\alpha,n}^{(i)}\|_{\HH^s} \leq
C(\alpha,d)\frac{\sqrt{\tau}}{\sqrt{n}}$ with probability
$1-2e^{-\tau}$, and so we have bounded the first term on the right
hand side of \eqref{eqn: mixed 1}. For the second term on the right
hand side of \eqref{eqn: mixed 1}, recall the
definition of $\|\cdot\|_{\HH^s}$. If we can bound
$\|\partial_x^{\gamma} a_{\alpha}(x^{(i)},\cdot) - \partial_x^{\gamma}
E_{\alpha,\HH^s} \varphi_{\alpha}^{(N(\varepsilon,\alpha),i)} \|_{L^2(X,dx)}$,
where $\gamma = 0$ (i.e., no derivative) or $|\gamma| = s$, then we
will have bounded this term as well. Note that $a_{\alpha} \in
C_b^{d+1}(X \times X)$ implies that $\psi_{\alpha}^{(\ell)} \in
C_b^s(X)$ for all $\ell \geq 1$. Furthermore, the derivative
$\partial_x^{\gamma} a_{\alpha}(x^{(i)},\cdot)$ can be computed term
by term from \eqref{eqn: another mercer application}. Thus, using
nearly the same argument we used to show \eqref{eqn: cont row approx},
one can show that
\begin{equation} \label{eqn: cont uniform conv deriv}
\left| \partial_x^{\gamma} a_{\alpha}(x^{(i)},x) - \partial_x^{\gamma}
  E_{\alpha,\HH^s} \varphi_{\alpha}^{(N(\varepsilon,\alpha),i)}(x)
\right| < C_4(\alpha) \, \varepsilon, \quad \text{for all } x \in X, \enspace
|\gamma| \leq s.
\end{equation}
Using \eqref{eqn: cont uniform conv deriv}, we have:
\begin{equation*}
\left\|\partial_x^{\gamma} a_{\alpha}(x^{(i)},\cdot) - \partial_x^{\gamma}
E_{\alpha,\HH^s} \varphi_{\alpha}^{(N(\varepsilon,\alpha),i)}
\right\|_{L^2(X,dx)} \leq \sqrt{|X|} \, C_4(\alpha) \, \varepsilon,
\end{equation*}
where $|X|$ denotes the Lebesgue measure of $X$. Since $X$ was assumed
to be bounded, we have $|X| \leq C$. Returning to \eqref{eqn: mixed
  1}, we have now shown that:
\begin{equation*}
\left\| a_{\alpha,n}(x^{(i)},\cdot) - E_{\alpha,\HH^s}
  \varphi_{\alpha}^{(N(\varepsilon,\alpha),i)} \right\|_{\HH^s} \leq
C(\alpha,d) \frac{\sqrt{\tau}}{\sqrt{n}} + C \varepsilon.
\end{equation*}
Taking $\varepsilon = \frac{\sqrt{\tau}}{\sqrt{n}}$ completes the proof.
\end{proof}

\subsection{Proof of Theorem \ref{thm: global diff dist sample}}

Finally, we prove Theorem \ref{thm: global diff dist sample}.
\begin{theorem}[Theorem \ref{thm: global diff dist sample}]
Suppose that $(X,\mu)$ and $\{k_{\alpha}\}_{\alpha \in \I}$ satisfy the
conditions of Assumption \ref{assumption 1.5}. Let $n \in \N$ and
sample $X_n = \{x^{(1)}, \ldots, x^{(n)}\} \subset X$ i.i.d. according to
$\mu$; also let $t \in \N$, $\tau > 0$, and $\alpha,\beta \in \I$. Then, with
probability $1-2e^{-\tau}$, 
\begin{equation*}
\left| \Dt(\Gamma_{\alpha},\Gamma_{\beta}) -
  \Dt_n(\Gamma_{\alpha,n},\Gamma_{\beta,n}) \right| \leq
C(\alpha,\beta,d,t) \frac{\sqrt{\tau}}{\sqrt{n}}.
\end{equation*}
\end{theorem}

\begin{proof}
Recall that $\Dt(\Gamma_{\alpha},\Gamma_{\beta}) = \|A_{\alpha}^t -
A_{\beta}^t\|_{HS}$. From Proposition $13$ in
\cite{rosasco:learningIntegralOps2010}, we know that $\lambda \in
(0,1]$ is an eigenvalue of $A_{\alpha}$ if and only if it is an
eigenvalue of $A_{\alpha,\HH^s}$. Using the same ideas, one can show
that $\lambda' \neq 0$ is an eigenvalue of $A_{\alpha}^t-A_{\beta}^t$ if
and only if it is an eigenvalue of
$A_{\alpha,\HH^s}^t-A_{\beta,\HH^s}^t$. Therefore,
\begin{equation*}
\left\|A_{\alpha}^t-A_{\beta}^t\right\|_{HS} =
\left\|A_{\alpha,\HH^s}^t-A_{\beta,\HH^s}^t\right\|_{HS}. 
\end{equation*}
Similarly, one can show that
\begin{equation*}
\left\|\A_{\alpha}^t-\A_{\beta}^t\right\|_{HS} =
\left\|A_{\alpha,n}^t-A_{\beta,n}^t\right\|_{HS}. 
\end{equation*}
Thus, using the above and Theorem \ref{thm: rbd15} we have, with
probability $1-2e^{-\tau}$, 
\begin{align*}
\Dt(\Gamma_{\alpha},\Gamma_{\beta}) &=
\|A_{\alpha,\HH^s}^t-A_{\beta,\HH^s}^t\|_{HS} \\
&\leq \|A_{\alpha,n}^t-A_{\beta,n}^t\|_{HS} +
\|A_{\alpha,\HH^s}^t-A_{\alpha,n}^t\|_{HS} +
\|A_{\beta,\HH^s}^t-A_{\beta,n}^t\|_{HS} \\
&\leq \Dt_n(\Gamma_{\alpha,n},\Gamma_{\beta,n}) + C(\alpha,\beta,d,t)
\frac{\sqrt{\tau}}{\sqrt{n}}. 
\end{align*}
Since the argument is symmetric, we get the desired inequality.
\end{proof}
\end{appendix}

\bibliographystyle{model1-num-names}
\bibliography{DiffusionGeometry}

\end{document}